\documentclass[dvipsnames, reqno]{amsart}

\usepackage[utf8]{inputenc}
\usepackage[normalem]{ulem}

\usepackage{amsfonts,amsthm,amssymb,amsmath,stackrel,latexsym,mathabx}

\usepackage{scalerel,stackengine}

\usepackage{lmodern}
\usepackage{microtype}
\usepackage{xcolor}
\usepackage{graphicx}
\usepackage[all]{xy}
\usepackage[colorlinks=true,citecolor=MidnightBlue,linkcolor=MidnightBlue,urlcolor=MidnightBlue,backref=page]{hyperref}  %
\usepackage[inline]{enumitem}
\usepackage{a4wide}
\usepackage{stmaryrd}
\usepackage{mathtools,thmtools}

\makeatletter
\DeclareRobustCommand*\cal{\@fontswitch\relax\mathcal}
\makeatother

\numberwithin{equation}{section}

\usepackage{mathrsfs}

\usepackage{scalerel,stackengine}  

\allowdisplaybreaks

\theoremstyle{plain}
\newtheorem{theorem}{Theorem}[section]
\newtheorem{lemma}[theorem]{Lemma}
\newtheorem{proposition}[theorem]{Proposition}
\newtheorem{corollary}[theorem]{Corollary}

\declaretheorem[name=Theorem,qed={$\square$},numbered=no]{theorem*}
\declaretheorem[name=Proposition,qed={$\square$},numbered=no]{proposition*}
\theoremstyle{definition}
\newtheorem{definition}[theorem]{Definition}
\theoremstyle{remark}
\declaretheorem[name=Example,qed={$\scriptstyle\blacksquare$},sibling=theorem]{example}
\declaretheorem[name=Remark,qed={$\scriptstyle\blacksquare$},sibling=theorem]{remark}

\DeclareRobustCommand{\SkipTocEntry}[5]{}

\def\be{\begin{equation}}
\def\ee{\end{equation}}
\def\ba{\begin{eqnarray}}
\def\ea{\end{eqnarray}}

\newcommand{\cD}{\mathcal{D}}

\newcommand{\cJ}{\mathcal{J}}

\newcommand{\cL}{\mathcal{L}}

\newcommand{\cP}{\mathcal{P}}

\newcommand{\cS}{\mathcal{S}}
\newcommand{\cU}{\mathcal{U}}

\newcommand{\R}{\mathbb{R}} 

\newcommand{\Z}{\mathbb{Z}}

\newcommand{\Der}{\mathfrak{der}}
\newcommand{\End}[2]{\mathrm{End}_{#1}\left({#2}\right)}
\newcommand{\Hom}[6]{{_{#1}^{#2}\mathrm{Hom}_{#3}^{#4}}(\hspace*{1pt}{#5},{#6}\hspace*{1pt})}

\usepackage{xparse}

\def\clap#1{\hbox to 0pt{\hss#1\hss}}

\newcommand{\mf}[1]{\mathfrak{#1}}
\newcommand{\ms}[1]{\mathsf{#1}}
\newcommand{\gr}{\mathsf{gr}} 
\newcommand{\tak}[1]{\times_{{#1}}} 

\newcommand{\id}{\mathsf{id}}
\newcommand{\K}{\Bbbk}

\newcommand{\RR}{{R \otimes_\ahha  \mf{h}}}
\newcommand{\LL}{\mathfrak{L}}
\newcommand{\Lie}[1]{\mathfrak{#1}}

\newcommand{\M}{\mathsf{Mod}} 
\newcommand{\Lmod}[1]{{}_{#1}\M} 
\newcommand{\Bim}[2]{{}_{#1}\M{}_{#2}} 

\newcommand{\ot}{\otimes}

\newcommand{\bla}{~{\raisebox{+1pt}{$\scriptstyle \blacktriangleright$}}~} 
\newcommand{\bra}{~{\raisebox{+1pt}{$\scriptstyle \blacktriangleleft$}}~} 

\newcommand{\ga}{\alpha} 
\newcommand{\gb}{\beta}  

\newcommand{\gd}{\delta} 
\newcommand{\gD}{\Delta} 
\newcommand{\gve}{\varepsilon}

\newcommand{\go}{\omega} 
\newcommand{\gO}{\Omega}

\newcommand{\gs}{\sigma}

\newcommand{\ahha}{{\scriptscriptstyle{A}}}
\newcommand{\behhe}{{\scriptscriptstyle{B}}}
\newcommand{\cehhe}{{\scriptscriptstyle{C}}}

\newcommand{\emme}{{\scriptscriptstyle{M}}}

\newcommand{\erre}{{\scriptscriptstyle{R}}}

\newcommand{\uhhu}{{\scriptscriptstyle{U}}}

\newcommand{\vauu}{{\scriptscriptstyle{V}}}

\newcommand{{\Aop}}{{A^{\rm op}}}
\newcommand{{\Uop}}{{U^{\rm op}}}

\newcommand{{\Aope}}{({A^{\rm op}})^{\rm e}}
\newcommand{{\Ae}}{{A^{\rm e}}}
\newcommand{\Aopp}{{\scriptscriptstyle{\Aop}}}

\newcommand{\Uopp}{{\scriptscriptstyle{\Uop}}}

\newcommand{{\coop}}{{{\rm coop}}}

\newcommand{\lact}{\smalltriangleright}                  
\newcommand{\ract}{\smalltriangleleft}
\newcommand{\blact}{\blacktriangleright} 
\newcommand{\bract}{\blacktriangleleft} 

\newcommand\pig[1]{\scalerel*[5pt]{\big#1}{%
  \ensurestackMath{\addstackgap[1.5pt]{\big#1}}}}

\newcommand{\umod}{U\mbox{-}\mathbf{Mod}}                     

\newcommand{\comodu}{\mathbf{Comod}\mbox{-}U}

\newcommand{\due}[3]{{}_{{#2 }} {#1}_{{ #3}}\,}   
\renewcommand{\hash}[1]{\operatorname{\varhash_{{\!#1}}}}

\newdir{ >}{{}*!/-10pt/@{>}}
\newdir{ (}{{}*!/-10pt/@{(}}

\makeatletter
\patchcmd{\@setaddresses}{\indent}{\noindent}{}{}
\patchcmd{\@setaddresses}{\indent}{\noindent}{}{}
\patchcmd{\@setaddresses}{\indent}{\noindent}{}{}
\patchcmd{\@setaddresses}{\indent}{\noindent}{}{}
\makeatother

\makeatletter
\@namedef{subjclassname@2020}{%
  \textup{2020} Mathematics Subject Classification}
\makeatother

\title[UEA of L-R algebras: crossed products, connections, and curvature]{Universal enveloping algebras of Lie-Rinehart algebras: crossed products, connections, and curvature}

\author[X. Bekaert]{Xavier Bekaert}
\address{X.B.: Institut Denis Poisson, Unit\'e Mixte de Recherche $7013$ du CNRS,
Universit\'e de Tours \& Universit\'e d'Orl\'eans, Parc de Grandmont, 37200 Tours, France}
\email{xavier.bekaert@lmpt.univ-tours.fr}

\author[N. Kowalzig]{Niels Kowalzig}
\address{N.K.: Dipartimento di Matematica, Universit\`a di Roma Tor Vergata, Via della Ricerca Scientifica 1, 00133 Roma, Italy}
\email{niels.kowalzig@uniroma2.it}

\author[P. Saracco]{Paolo Saracco}
\address{P.S.: D\'epartement de Math\'ematique, Universit\'e Libre de Bruxelles, Boulevard du Triomphe, B-1050 Bruxelles, Belgium}
\urladdr{\url{sites.google.com/view/paolo-saracco}}
\urladdr{\url{paolo.saracco.web.ulb.be}}
\email{paolo.saracco@ulb.be}



\begin{document}

\begin{abstract}
  We extend a theorem, originally formulated by Blattner-Cohen-Montgomery for crossed products arising from Hopf algebras weakly acting on noncommutative algebras, to the realm of left Hopf algebroids.   
  Our main motivation is an application to universal enveloping algebras of projective Lie-Rinehart algebras: for any given curved (resp.~flat) connection, that is, a linear (resp.~Lie-Rinehart) splitting of a Lie-Rinehart algebra extension, we provide a crossed (resp.~smash) product decomposition of the associated universal enveloping algebra, and vice versa. 
As a geometric example, we describe the associative algebra generated by the invariant vector fields on the total space of a principal bundle as a crossed product of the algebra generated by the vertical ones and the algebra of differential operators on the base.
\end{abstract}

\maketitle

\tableofcontents

\section*{Introduction}

\addtocontents{toc}{\protect\setcounter{tocdepth}{0}}

\subsection*{State of the art} 
In Lie theory, the universal enveloping algebra construction allows to construct an associative algebra $U(\mf{g})$ from any Lie algebra $\mf{g}$, in such a way that the representation theory remains unchanged. Such a construction is functorial, in the sense that it defines a functor $U$ from the category of Lie algebras to the category of associative algebras. The terminology ``universal'' is justified by the fact that such a functor is, in fact, a left adjoint to the well-known functor that sends every associative algebra to the Lie algebra that has the same underlying vector space and the commutator as bracket. A natural question which arises is how this functor behaves with respect to familiar constructions for Lie algebras, such as (semi-)direct sums or splittings of short exact sequences.

In addition, Milnor and Moore exhibited in \cite{MilnorMoore} an equivalence between the category of Lie algebras over a field $\K$ of characteristic zero and the category of cocommutative and primitively generated Hopf algebras over $\K$. This equivalence is induced by the same universal enveloping algebra construction in one direction, and by considering the space of primitive elements in the other. In light of this, Hopf algebras entered the study of Lie groups and differential geometry: Lie groups and their representation theory can be studied in terms of their invariant vector fields (and their representations), which however forces to work with non-associative (Lie) structures or, equivalently, in terms of their algebras of invariant differential operators (and their representations). This equivalent formulation allows to work in the associative setting and to take advantage of all the tools coming from ring theory and Hopf algebra theory. 

In this rich framework, classical constructions from Lie algebra theory can be equivalently reinterpreted in terms of their universal enveloping algebras. In the present work, we will start from the following well-known results.

\smallskip 

\begin{enumerate}[label=\textbf{(\Alph*)},wide,labelindent=0pt]

\item\label{item:B} Let a Lie algebra $\mathfrak{g}$ be a semi-direct sum
\[
\mathfrak{g}=\mathfrak{n}\niplus\mathfrak{h}
\]
of an ideal $\mathfrak{n}\subseteq\mathfrak{g}$ and a Lie subalgebra $\mathfrak{h}\subseteq\mathfrak{g}$.
Then the universal enveloping algebra of $\mathfrak{g}$ is isomorphic to the smash product of the 
enveloping algebras of $\mathfrak{n}$ and $\mathfrak{h}$, that is,
\[
U(\mathfrak{g}) \simeq U(\mathfrak{n}) \hash{} U(\mathfrak{h})
\]
as associative algebras.

This result is a direct consequence of the universal property of universal enveloping algebras, but it admits the following remarkable generalisation:

\smallskip

\item\label{item:C}  Let $\mathfrak{g}$ be a Lie algebra extension of $\mathfrak{h}$ by $\mathfrak{n}$, that is, a
short exact sequence
\be\label{splitses}\tag{I.1}
0 \to \mathfrak{n} \hookrightarrow \mathfrak{g} \twoheadrightarrow \mathfrak{h} \to 0
\ee
of Lie algebras.
Then the universal enveloping algebra of $\mathfrak{g}$ is isomorphic to the crossed
product of the universal enveloping algebras
of $\mathfrak{n}$ and $\mathfrak{h}$, that is,
\be\label{factorisation}\tag{I.2}
U(\mathfrak{g}) \simeq U(\mathfrak{n}) * U(\mathfrak{h})
\ee
as associative algebras.
This isomorphism defines and relies on a $\K$-linear splitting of the short exact sequence \eqref{splitses}.

\smallskip

\item\label{item:A} Let $\mf{n}$ be a Lie subalgebra of a Lie algebra $\mf{g}$. Then the universal enveloping algebra $U(\mf{g})$ is free over the universal enveloping algebra $U(\mf{n})$ of the Lie subalgebra $\mf{n}$. More precisely, consider the
short exact sequence \eqref{splitses} of $\mf{n}$-modules (via the adjoint representation),
where $\mathfrak{h}=\mathfrak{g}/\mathfrak{n}$ is the quotient module.
Then, the universal enveloping algebra of $\mathfrak{g}$ is isomorphic to the tensor product of $U(\mf{n})$ and the symmetric algebra $S(\mf{h})$, that is,
\[
U(\mf{g}) \simeq U(\mf{n}) \otimes S(\mf{h})
\]
as a left $U(\mf{n})$-module.

\end{enumerate}

The statements \ref{item:B} and \ref{item:C} above appear in this form in the textbook of McConnell and Robson \cite[pp.\ 33--34]{McConnell}. The statement \ref{item:C} also appears in the paper \cite[Ex.\  4.20]{BlattnerCohenMontgomery} (see also \cite{Mon:CPOHAAEA}) by Blattner, Cohen, and Montgomery, where it is shown to descend from general results on crossed products of Hopf algebras. Finally, the statement \ref{item:A} is a modern rephrasing of a well-known result in Lie algebra theory (see \cite[Prop.\  2.2.7]{Dixmier}).

In generalising Lie algebras, a distinguished r\^ole is played by Lie-Rinehart algebras. A Lie-Rinehart algebra over a commutative algebra $A$ is a Lie algebra $\mf{g}$ together with additional structures that mimic the interaction between the commutative algebra of smooth functions on a smooth manifold $M$ and the Lie algebra of smooth vector fields on $M$ ({\em i.e.}, the Lie derivative). In fact, the main source of examples of Lie-Rinehart algebras is provided by the global sections of Lie algebroids on smooth manifolds. For Lie-Rinehart algebras as well there is the notion of universal enveloping algebra $\cU_\ahha(\mf{g})$, which, in the geometric case, may be interpreted as an algebra of smooth differential operators. An explicit construction for $\cU_\ahha(\mf{g})$ was given by Rinehart himself in \cite{Rinehart} who also provided a Poincar\'e-Birkhoff-Witt theorem for the latter. Other equivalent constructions appear in \cite[\S18]{sweedler-groups}, \cite[p.~64]{Huebschmann}, or \cite[\S3.2]{LaiachiPaolo2}. The universal property of $\cU_\ahha(\mf{g})$ as an algebra was spelled out in \cite[p.~64]{Huebschmann} and \cite[p.~174]{Malliavin} (where it is attributed to Feld'man). Its universal property as a left $A$-bialgebroid is codified in the Cartier-Milnor-Moore Theorem for $\cU_\ahha(\mf{g})$ proven in \cite[\S3]{MoerdijkLie} and its universal property as an $A$-ring was presented in \cite{Saracco2}. 

\subsection*{Aims and objectives}
The main questions we aim to answer here are whether it is possible to provide analogues of the results \ref{item:B}, \ref{item:C}, and \ref{item:A} for Lie-Rinehart algebras and their universal enveloping algebras, at least in the projective case.
The affirmative answers to our questions above are summarised in the following results (see the main text for details and notation).

\begin{theorem*}[Theorem \ref{thm:mainthmA}, Corollary \ref{cor:mainthmB}, and Theorem \ref{thm:mainthmC}]\hfill
If $0 \to \mf{n} 
\stackrel{\iota}{\hookrightarrow} \mf{g} \stackrel{\pi}{\twoheadrightarrow} \mf{h} \to 0$ is a short exact sequence of Lie-Rinehart algebras which are projective as left $A$-modules, then we have an isomorphism 
\[
\cU_\ahha(\mf{g}) \simeq U_{\ahha}(\mf{n}) \hash{\sigma} \cU_\ahha(\mf{h})
\]
of $A$-rings and right $\cU_\ahha(\mf{h})$-comodule algebras,
where $\sigma$ is a suitable $U_{\ahha}(\mf{n})$-valued Hopf $2$-cocycle and $\hash{\sigma}$ denotes the corresponding crossed product.
In particular:
\begin{itemize}[leftmargin=0.7cm]

\item[\ref{item:B}] If $\mf{g} \simeq \mf{n} \niplus \mf{h}$ is a semi-direct sum of an $A$-Lie algebra $(A,\mf{n})$ and a Lie-Rinehart algebra $(A,\mf{h})$, then we have an isomorphism 
\[
\cU_\ahha (\mf{n} \niplus \mf{h}) \simeq U_{\ahha}(\mf{n}) \hash{} \cU_\ahha(\mf{h})
\]
of $A$-rings and right $\cU_\ahha(\mf{h})$-comodule algebras.

\item[\ref{item:C}] If $\mf{g} \simeq \mf{n} \niplus_\tau \mf{h}$ is a curved semi-direct sum of an $A$-Lie algebra $(A,\mf{n})$ and a Lie-Rinehart algebra $(A,\mf{h})$ for a Lie 2-cocycle $\tau$, then for a suitable Hopf $2$-cocycle $\sigma$ we have an isomorphism 
\[
\cU_\ahha (\mf{n} \niplus_\tau \mf{h}) \simeq U_{\ahha}(\mf{n}) \hash{\gs} \cU_\ahha(\mf{h})
\]
of $A$-rings and right $\cU_\ahha(\mf{h})$-comodule algebras.\qedhere
\end{itemize}
\end{theorem*}

\begin{proposition*}[Proposition \ref{prop:anothersplitting}]
\ref{item:A} Let $\mf{n} \subseteq \mf{g}$ be an inclusion of Lie-Rinehart algebras which are projective as left $A$-modules. Suppose that the quotient $A$-module $\mf{h}=\mf{g}/\mf{n}$ is projective as well. Then we have an isomorphism
\[\cU_\ahha(\mf{g}) \simeq \cU_\ahha(\mf{n}) \otimes_\ahha  \cS_\ahha(\mf{h})\]
as left $\cU_\ahha(\mf{n})$-modules. In particular, $\cU_\ahha(\mf{g})$ is projective over $\cU_\ahha(\mf{n})$.
\end{proposition*}

The statements in the theorem above follow as applications of the construction of the symmetrisation map in \S\ref{ssec:sym} and of the main theorem of the present paper (see again the main text for notation and details), which reads as follows:

\begin{theorem*}[Theorem \ref{thm:sigmatwisted}]
Let $(U,A)$ and $(V,A)$ be two left Hopf algebroids over the same base algebra $A$ and 
let $U \stackrel{\pi}{\to} V \to 0$ be an exact sequence of left Hopf algebroids which splits as an $A$-coring sequence, that is, there exists a morphism $\gamma\colon V \to U$ of $A$-corings such that $\pi \circ \gamma = \id_\vauu$. Assume furthermore that $U_{\ract}$ is projective as an $A$-module and that
\[
\gamma(1_\vauu) = 1_\uhhu, \qquad \gamma(a \blact v) = a \blact \gamma(v), \qquad \gamma(v \bract a) = \gamma(v) \bract a, \qquad \forall v \in V, a \in A. 
\]
Let $B$ be the left Hopf kernel of $\pi$.
Then there exists a $B$-valued Hopf $2$-cocycle $\gs\colon \due U \blact {} \otimes_\Aopp U_\ract \to B$ such that $B$ becomes a $\gs$-twisted left $V$-module with respect to the weak left $V$-action $v \rightslice b = \gamma(v)_+\, b\, \gamma(v)_-$ for all $v \in V$, $b \in B$, and there is an isomorphism
$$
\Phi\colon B \hash{\gs} V \to U, \quad b \hash{} v \mapsto b\,\gamma(v)
$$
of $\Ae$-rings and of right $V$-comodule algebras.
\end{theorem*}

A geometric motivation for the extension of \ref{item:B} and \ref{item:C} from Lie algebras to Lie-Rinehart algebras is that the above two situations in the Lie algebroid case correspond, respectively, to flat or curved connections on a transitive Lie algebroid (for instance, on the Atiyah algebroid of a principal bundle). 
A corollary of our result is that an invariant Ehresmann connection on a principal bundle provides a factorisation of the associative algebra generated by the invariant vector fields on the total space as a smash or crossed product of the algebra generated by the vertical ones and the algebra of differential operators on the base manifold.
The extension of the result \ref{item:A} to Lie-Rinehart algebras is motivated geometrically by its application to foliations. 

In addition, we provide the following alternative description of the (crossed) product $\hash{\gs}$ as the product $\times_\tau$ (which involves a Lie cocycle $\tau$ instead of a Hopf cocycle $\sigma$).

\begin{theorem*}[Theorem \ref{rain&fog}]
Let $\iota \colon A \to R$ be an $A$-algebra and $(A, \mf{h})$ a Lie-Rinehart algebra such that both $R$ and $\mf{h}$ are projective as left $A$-modules. Then for an $A$-ring $S$ the following are equivalent:
\begin{enumerate}
\item
$S \simeq R \times_\tau \cU_\ahha(\mf{h})$ in the sense of Definition \ref{def:GaloisDescent} where $\tau$ denotes a Lie cocycle.
\item
  $S \simeq R  \hash{\gs}  \cU_\ahha(\mf{h})$ in the sense of Definition \ref{GT4} where $\sigma$ denotes a Hopf cocycle. \qedhere
  \end{enumerate}
\end{theorem*}

This latter result extends the equivalent descriptions of crossed product factorisations in \cite{Mon:CPOHAAEA} to the realm of Lie-Rinehart algebras.

\subsection*{Outline}
The plan
of the paper is as follows: In \S\ref{sec:Preliminaries}, we summarise the main notions and results concerning Lie-Rinehart algebras and their universal enveloping algebras, with the intention of keeping the presentation as self-contained as possible.  Both \S\ref{ssec:sym}, where we exhibit a Lie-Rinehart analogue of the well-known symmetrisation map, and \S\ref{ssec:modalgs}, where we prove our first decomposition theorem (Proposition \ref{prop:iso3}), are of particular relevance for their novelty, whereas \S\ref{sec:BCM} is entirely devoted to proving our main theorem, {\em i.e.}, Theorem \ref{thm:sigmatwisted}. In \S\ref{sec:weakcase}, we apply this theorem to the cases of interest for the study of Lie-Rinehart algebras: the crossed product decomposition result of Theorem \ref{thm:mainthmA} and the smash product decomposition result of Theorem \ref{thm:mainthmC}. Finally, we conclude by concrete geometric applications of our achievements in \S\ref{geomex}. Appendix \ref{sec:tech} contains the technical details of some proofs omitted from the main body, to avoid making the discussion unnecessarily heavy.

\addtocontents{toc}{\protect\setcounter{tocdepth}{2}}

\section{Preliminaries and first results}
\label{sec:Preliminaries}

We work over a ground field $\K$ of characteristic $0$. All vector spaces and algebras are assumed to be over $\K$. The unadorned tensor product $\otimes$ stands for $\otimes_{\K}$. Identity morphisms $\id_\vauu$ are often denoted simply by $V$. 
Lie algebras will be denoted by Gothic letters.
Unless stated otherwise, an algebra $A$ is always assumed to be associative and unital.
The Lie algebra associated to $A$, denoted by $\Lie{A}$ (or simply $A$ again), 
is the Lie algebra with $A$ itself as underlying $\K$-vector space and bracket given by the commutator $[a,b] = ab-ba$ for all $a,b \in A$. If $A$ is noncommutative, $\Aop$ denotes its opposite algebra and $\Ae  \coloneqq A \otimes \Aop$ its enveloping algebra.

An {\em $A$-ring} is a $\K$-algebra $R$ together with a morphism of $\K$-algebras $\phi \colon A \to R$, not necessarily landing in the center. This endows $R$ with the $A$-bimodule structure $a \cdot r \cdot b \coloneqq \phi(a)\,r\,\phi(b)$ for all $a,b \in A$, $r \in R$. Equivalently, an $A$-ring is a monoid in the monoidal category $(\Bim{\ahha}{\ahha},\ot_\ahha,A)$ of $A$-bimodules. An {\em $A$-algebra} over a commutative algebra $A$ is an $A$-ring $(R,\phi)$ such that the image of $\phi$ is in the center $Z(R)$ of $R$ or, equivalently, a monoid in the monoidal category $(\Lmod{\ahha},\ot_\ahha,A)$ of $A$-modules.

Dually, an $A$-coring is a comonoid in the monoidal category $(\Bim{\ahha}{\ahha},\ot_\ahha,A)$, that is, an $A$-bimodule $C$ together with a comultiplication $\Delta_{\cehhe} \colon C \to C \ot_\ahha C$ and a counit $\varepsilon_{\cehhe} \colon C \to A$ such that
\[
(\Delta_{\cehhe} \ot_\ahha C) \circ \Delta_{\cehhe} = (C \ot_\ahha \Delta_{\cehhe}) \circ \Delta_{\cehhe}, \qquad \text{and} \qquad (\varepsilon_{\cehhe} \ot_\ahha C)\circ \Delta_{\cehhe} = \id_{\cehhe} = (C \ot_\ahha \varepsilon_{\cehhe}) \circ \Delta_{\cehhe}.
\]
Similarly, an {\em $A$-coalgebra} over a commutative algebra $A$ is a comonoid in $(\Lmod{\ahha},\ot_\ahha,A)$. For a coalgebra or a coring, we usually adopt the Heyneman-Sweedler notation
\[\Delta_{\cehhe}(c) = c_{(1)} \otimes_\ahha  c_{(2)}\] 
(summation understood) to explicitly describe the comultiplication. When there is no risk of confusion, we often write simply $\Delta$ and $\varepsilon$ instead of $\Delta_{\cehhe}$ and $\varepsilon_{\cehhe}$.


\subsection{Lie-Rinehart algebras}\label{sec:LieRin}
Let us fix a commutative algebra $A$ over a field $\K$. 
A \emph{Lie-Rinehart algebra} over $A$ (so baptised in \cite{Huebschmann} to honour Rinehart who studied them in \cite{Rinehart}) is a Lie algebra $\mf{g}$ endowed with a (left) $A$-module structure $A\otimes \mf{g}\to \mf{g}, \ a \otimes X \mapsto a \cdot X,$ and with a Lie algebra morphism $\omega_{\mf{g}}\colon\mf{g}\to \Der_{\,\K}(A)$ 
such that
\[
\omega_{\mf{g}} ( a\cdot X) = a\cdot \omega_{\mf{g}} ( X) \qquad \text{and}\qquad [ X,a\cdot Y] = a \cdot [ X,Y] +\big(\omega_{\mf{g}} (X)(a)\big) \cdot Y
\]
for all $a \in A$ and $X,Y\in \mf{g}$. The Lie algebra morphism $\omega_{\mf{g}}$ with the above property is called the \emph{anchor} of the Lie-Rinehart algebra. 
We often write $\omega$ instead of $\omega_{\mf{g}}$ and $X(a)$ instead of $\omega_{\mf{g}}(X)(a)$ for $X \in \mf{g}$ and $a \in A$, if this does not create confusion.

A \emph{morphism of Lie-Rinehart algebras} over the same base $A$ is a Lie algebra morphism $f\colon\mf{g}\to \mf{g}^{\prime }$ which is left $A$-linear and such that $\omega^{\prime }\circ f=\omega $. 

A Lie-Rinehart algebra typically will be written as $(A,\mf{g},\omega)$. We may also write $(A,\mf{g})$ or just $\mf{g}$ when $A$ or $\omega$ are clear from the context or irrelevant, and likewise we may simply write $f\colon \mf{g} \to \mf{g}'$ to mean a morphism of Lie-Rinehart algebras.

\begin{example}[Lie algebroid]
The smooth global sections of a Lie algebroid over a real smooth manifold $M$ naturally form a Lie-Rinehart algebra over $C^\infty(M)$ (see, {\em e.g.},~\cite[p.\ 101]{Mackenzie}, where Lie-Rinehart algebras are called \emph{Lie pseudoalgebras}). In particular, the smooth vector fields on $M$ give rise to the Lie-Rinehart algebra
$\mathfrak{X}(M) \coloneqq\Der_{\R}\big(C^\infty(M)\,\big)\simeq\Gamma(TM)$ whose anchor is the identity.
\end{example}

\begin{definition}[Lie-Rinehart ideal] 
An \textit{$A$-Lie algebra} is a Lie-Rinehart algebra over $A$ with trivial anchor.
A \textit{Lie-Rinehart ideal} $\mathfrak{n}$ of a Lie-Rinehart algebra $\mathfrak{g}$ over $A$ is an $A$-Lie algebra which is a Lie ideal $\mathfrak{n}\subseteq\mathfrak{g}$.
\end{definition}

\begin{example}[Kernel of a Lie-Rinehart algebra morphism]
\label{ex:LRideal}
The kernel $\mathfrak{n}=\ker f$ of a morphism $f\colon\mf{g}\to \mf{h}$ of Lie-Rinehart algebras from $\mathfrak{g}$ to $\mathfrak{h}$ is a Lie-Rinehart ideal $\mathfrak{n}\subseteq\mathfrak{g}$.
\end{example}

From now on, all Lie-Rinehart algebras are assumed to be over $A$. 
By a short exact sequence 
\begin{equation}\label{eq:ses} 
0 \to \mf{n} \xrightarrow{\iota} \mf{g} \xrightarrow{\pi} \mf{h} \to 0
\end{equation}
of Lie-Rinehart algebras we mean a short exact sequence of morphisms of $A$-modules which are also morphisms of Lie algebras and that are compatible with the anchors as above. The latter condition implies, in particular, that the anchor of $\mf{n}$ needs to be zero (that is, $\mf{n}$ is an $A$-Lie algebra). We will often assume, in addition, that our Lie-Rinehart algebras are projective as (left) $A$-modules, which we may refer to as \emph{projective Lie-Rinehart algebras}. This, in particular, implies that $\pi\colon\mf{g} \to \mf{h}$ admits a section $\gamma\colon\mf{h} \to \mf{g}$ as left $A$-linear map.


\subsection{Lie-Rinehart actions, connections, and curvature}
\label{ssec:LRconnection}
Let $(A,\mf{h}, \go)$ be a Lie-Rinehart algebra with anchor $\omega\colon\mf{h} \to \Der_{\,\K}(A)$.

\begin{definition}[Representations of Lie-Rinehart algebras {\cite[p.\ 62]{Huebschmann}}]
\label{def:LRaction}
A \emph{representation} (or \emph{action}) of a Lie-Rinehart algebra $(A,\mf{h},\omega)$ on a left $A$-module $N$ is a Lie algebra morphism
$\rho \colon \mf{h} \to \mathfrak{gl}_{\,\K}(N)$ such that
\begin{equation}\label{eq:LRaction}
    \rho(a\cdot X)(n) = a\cdot \rho(X)(n) \qquad \text{and} \qquad \rho(X)(a\cdot n) = \omega(X)(a)\cdot n + a\cdot \rho(X)(n)
\end{equation}
for all $X \in \mf{h}$, $n\in N$, $a\in A$. Equivalently, $N$ has a left $\mf{h}$-action $\mf{h} \otimes N \to N$, $X \otimes n \mapsto X(n),$ such that
\begin{equation}
  \label{postnummer}
  (a\cdot X)(n) = a \cdot X(n) \qquad \text{and} \qquad X(a\cdot n) = X(a) \cdot n + a \cdot X(n) 
  \end{equation}
hold for all $X \in\mf{h}$, $n \in N$ and $a \in A$. We refer to $N$ as a {\em left $(A,\mf{h})$-module} or
{\em Lie-Rinehart module}.

\end{definition}

For the sake of simplicity, in the future we may omit the $\cdot$ symbol for the actions.
For an $(A,\mf{h})$-module $N$, we call
$\Hom{}{}{\ahha}{}{\textstyle\bigwedge_\ahha^\bullet \! \mf{h}}{N}$
the space of {\em alternating $N$-valued $A$-multilinear} {\em forms}, which yields a cochain complex $\left(\Hom{}{}{\ahha}{}{\textstyle\bigwedge_\ahha^\bullet \! \mf{h}}{N}, d\right )$, where
$$
d \colon \Hom{}{}{\ahha}{}{\textstyle\bigwedge_\ahha^n \! \mf{h}}{N} \to \Hom{}{}{\ahha}{}{\textstyle\bigwedge_\ahha^{n+1} \! \mf{h}}{N}
$$
is the {\em de Rham-Chevalley-Eilenberg-Rinehart} differential
\begin{small}
\begin{equation}
\label{fassbinder}
\begin{split}
d \alpha(X^1, \ldots, X^{n+1}) & \coloneqq \textstyle\sum\limits^{n+1}
_{i=1} (-1)^{i-1} X^i\pig(\alpha(X^1, \ldots,  \hat{X}^i, \ldots, X^{n+1})\pig) \\
&\quad + \textstyle\sum\limits_{i < j} (-1)^{i+j} \alpha\big([X^i, X^j], X^1, \ldots, \hat{X}^i, \ldots, \hat{X}^j, \ldots,  X^{n+1}\big),
\end{split}
\end{equation}
\end{small}%
and where, as usual, $\hat X^i$ means omission of $X^i\in\mathfrak{h}$, see \cite[Eq.~(4.3)]{Rinehart}.

More generally, if we have an assignment $\mf{h} \otimes N \to N$
which is not necessarily a left $\mf{h}$-action but which still obeys \eqref{eq:LRaction}, or equivalently a map $\rho\colon \mf{h} \to \mathfrak{gl}_{\,\K}(N)$ which is not necessarily a Lie algebra morphism but still fulfilling \eqref{postnummer}, then we call this a {\em left $(A,\mf{h})$-connection} on $N$, following \cite[p.~70f.]{Huebschmann}. A Lie-Rinehart module structure would then be a {\em flat} connection. We rewrite this by using the adjoint
$$
\nabla\colon N \to \mathrm{Hom}_\ahha (\mf{h},N), \qquad n \mapsto \big\{X \mapsto \rho(X)(n)\big\},
$$
of $\rho$ and hence the properties \eqref{eq:LRaction} turn into the customary formul\ae
\begin{equation}
  \label{izvestiya}
\nabla_{aX} \,n = a\nabla_X\, n, \qquad
\nabla_X(an) = X(a)\, n + a \,\nabla_X \,n,
\end{equation}
for all $X \in \mf{h}$, $n \in N$ and $a \in A$, well-known in differential geometry for a linear connection on a vector bundle.
In such a case, the formula \eqref{fassbinder} for $d$ still makes sense if changing the notation from $X(\cdot)$ to $\nabla_X(\cdot)$ in the first summand of the right hand side. The only difference is, in general, that $d^2 \neq 0$, whence it is not a differential any more. To indicate this situation of a non-flat connection, we denote
the left-hand side of \eqref{fassbinder} by $D\alpha$ (instead of $d\alpha$)
and call this the {\em exterior covariant  derivative} on $\Hom{}{}{\ahha}{}{\textstyle\bigwedge^\bullet_\ahha  \!\mf{h}}{N}$. Even if $\big(\Hom{}{}{\ahha}{}{\textstyle\bigwedge^\bullet_\ahha  \!\mf{h}}{N}, D\big)$ does not define a cochain complex for a non-flat connection, we will nevertheless refer to elements in $\Hom{}{}{\ahha}{}{\textstyle\bigwedge^\bullet_\ahha  \!\mf{h}}{N}$  as {\em cochains} and as {\em cocycles} to those elements $\ga \in \Hom{}{}{\ahha}{}{\textstyle\bigwedge^\bullet_\ahha  \!\mf{h}}{N}$ that fulfil $D\ga = 0$.

The {\em curvature} $\gO\colon \bigwedge^{2}_\ahha  \!\mf{h} \to \mathfrak{gl}_{\,\K}(N)$ is the adjoint of
$
D^2\colon N \to \Hom{}{}{\ahha}{}{\textstyle\bigwedge^{2}_\ahha  \!\mf{h}}{N},
$
which leads to the familiar expression
\begin{equation}
  \label{prawda}
\gO(X,Y) = \nabla_X \nabla_Y - \nabla_Y \nabla_X - \nabla_{[X,Y]}, 
\end{equation}
for $X,Y \in \mf{h}$, and the connection is flat as above if and only if its curvature vanishes.


\subsection{Curved semi-direct sum of Lie-Rinehart algebras}\label{ssec:curvedcase}

Assume now that $(A,\mf{n}, 0)$ is a Lie-Rinehart algebra with trivial anchor, that is, $\mf{n}$ is an $A$-Lie algebra. Furthermore, assume that there is a left $(A,\mf{h})$-connection $\nabla \colon \mf{n} \to \Hom{}{}{\ahha}{}{\mf{h}}{\mf{n}}$ on the $A$-module $\mf{n}$ by means of a map $\rho\colon \mf{h} \to \Der_{\,\K}(\mf{n}) \subset \mathfrak{gl}_{\,\K}({\mf{n}})$ 
and assume that the curvature $\gO$ of $\nabla$ is given by
\begin{equation}
  \label{beatport}
\gO(X,Y) = [\tau(X,Y), -]_{\mf{n}}, 
\end{equation}
where $\tau \in \Hom{}{}{\ahha}{}{\textstyle\bigwedge^2_\ahha  \mf{h}}{\mf{n}}$ is a $2$-form which is closed with respect to the exterior covariant derivative, that is, $D \tau = 0$.

\begin{proposition}\label{prop:curvedsdp}
If the aforementioned assumptions are met with respect to a Lie-Rinehart algebra $(A,\mf{h}, \go_{\mf{h}})$ and a Lie-Rinehart algebra $(A,\mf{n}, 0)$ with trivial anchor, then the $A$-module $\mf{n} \oplus \mf{h}$ can be made into a Lie-Rinehart algebra with anchor
$$
\omega
\colon\mf{n} \oplus \mf{h} \to \Der_{\,\K}(A), \quad (n,X) \mapsto \omega_{\mf{h}}(X), 
$$
and Lie bracket
\begin{equation}
\label{baratti&milano}
\big[(n, X), (m, Y) \big]_\tau \coloneqq \pig([n, m] + \nabla_X m - \nabla_Y n + \tau(X,Y), [X,Y] \pig)
\end{equation}
for $n, m \in \mf{n}$ and $X, Y \in \mf{h}$.
\end{proposition}

The proof is straightforward and left to the reader.
For the case of Lie algebroids, see \cite[Thm.~3.20]{Mackenziebook} or \cite[Thm.~7.3.7]{Mackenzie}.
Taking $\tau = 0$, we obtain:

\begin{corollary}
\label{cor:sdp}
Assume that we have a left action (a flat left connection) of a Lie-Rinehart algebra $(A,\mf{h})$ on an $A$-Lie algebra $(A,\mf{n})$, that is, assume a Lie algebra morphism 
$\rho \colon \mf{h} \to \Der_{\,\K}(\mf{n})$ 
is given
such that Eqs.~\eqref{eq:LRaction} hold.
Then the $A$-module $\mf{n} \oplus \mf{h}$ has the structure of a Lie-Rinehart algebra  over $A$ given by the anchor
\[
\omega_0\left(n,X\right)(a) =  \omega_{\mf{h}}(X)(a)
\]
and the Lie bracket
\[
\left[\left(n,X\right),\left(m,Y\right)\right]_0 = \pig( [n,m] +  \rho(X)(m) - \rho(Y)(n), [X,Y] \pig)
\]
for all $X, Y \in \mf{h}$, $n, m\in \mf{n}$, and $a\in A$.
\end{corollary}

\begin{definition}
\label{def:curvedsdp}
We call the Lie-Rinehart algebra from Proposition \ref{prop:curvedsdp} the {\em curved semi-direct sum} of $(A, \mf{n})$ and $(A, \mf{h})$ and we denote it by $(A, \mf{n} \niplus_\tau \mf{h})$. 
Likewise, the Lie-Rinehart algebra of Corollary \ref{cor:sdp} is called the \emph{semi-direct sum} Lie-Rinehart algebra of $(A, \mf{n})$ and $(A, \mf{h})$ and denote it by $(A, \mf{n} \niplus \mf{h})$. 
\end{definition}

This definition of semi-direct sum coincides with the one provided in \cite[\S2.2]{Casas}. 
The following characterisations of (curved) semi-direct sums should sound familiar. 

\begin{proposition}
\label{prop:semidirect}
Let $(A, \mf{n},0)$ be an $A$-Lie algebra and let $(A, \mf{g},\omega_{\mf{g}})$ as well as $(A, \mf{h},\omega_{\mf{h}})$ be Lie-Rinehart algebras. 
Then the following are equivalent:
\begin{enumerate}[label=(C\arabic*),leftmargin=0.9cm]
    \item\label{item:CP1} There is a connection $\rho\colon\mf{h} \to \Der_{\,\K}(\mf{n})$ of $\mf{h}$ on $\mf{n}$ and an isomorphism 
    $
    \mf{n}\niplus_\tau \mf{h} \simeq \mf{g}
    $ 
    of Lie- Rinehart algebras;
    \item\label{item:CP2} $\mf{n}$ is a Lie-Rinehart ideal in $\mf{g}$ and $\mf{h}$ is an $A$-submodule of $\mf{g}$ such that ${\omega_{\mf{g}}}\,|_{\mf{h}} = \omega_{\mf{h}}$ and $\mf{g} = \mf{n} + \mf{h}$, along with $\mf{n}\cap\mf{h} = 0$ as well as $[X,Y]_{\mf{g}} - [X,Y]_{\mf{h}} \in \mf{n}$ for all $X,Y \in \mf{h}$;
    \item\label{item:CP3} there is a short exact sequence of Lie-Rinehart algebras
    \[ 0 \to \mf{n} \xrightarrow{\iota} \mf{g} \xrightarrow{\pi} \mf{h} \to 0\]
    such that $\pi$ admits a left $A$-linear section.
\end{enumerate}
Furthermore, the following are equivalent as well:
\begin{enumerate}[label=(S\arabic*),leftmargin=0.9cm]
    \item\label{item:SP1} There is an action (a flat connection) $\rho\colon\mf{h} \to \Der_{\,\K}(\mf{n})$ of $\mf{h}$ on $\mf{n}$ and an isomorphism $
    \mf{n}\niplus \mf{h} \simeq \mf{g}$ of Lie-Rinehart algebras;
    \item\label{item:SP2} $\mf{n}$ is an $A$-submodule and a Lie ideal in $\mf{g}$ and $\mf{h}$ is a Lie-Rinehart subalgebra of $\mf{g}$ such that $\mf{g} = \mf{n} + \mf{h}$ and $\mf{n}\cap\mf{h} = 0$;
    \item\label{item:SP3} there is a short exact sequence of Lie-Rinehart algebras
    \[ 0 \to \mf{n} \xrightarrow{\iota} \mf{g} \xrightarrow{\pi} \mf{h} \to 0\]
    such that $\pi$ admits a left $A$-linear section which is also a morphisms of Lie algebras.
\end{enumerate}
In any of the above cases, the section $\gamma$ of $\pi$ satisfies $\omega_{\mf{g}} \circ \gamma = \omega_{\mf{h}}$.
\end{proposition}

\begin{proof}
We briefly sketch the proof and leave the straightforward details to the interested reader.

{\ref{item:CP1} $\Rightarrow$ \ref{item:CP2}}: the assignments
\[\varphi_{\mf{n}}\colon\mf{n} \to \mf{n}\niplus_\tau \mf{h}, \quad n \mapsto (n,0) \qquad \text{and} \qquad \varphi_{\mf{h}}\colon\mf{h} \to \mf{n}\niplus_\tau \mf{h}, \quad X \mapsto (0,X)\]
realise $\mf{n}$ as a Lie-Rinehart ideal and $\mf{h}$ as an $A$-submodule of $\mf{n}\niplus_\tau \mf{h}$ satisfying the given conditions.

{\ref{item:CP2} $\Rightarrow$ \ref{item:CP3}}: the hypotheses entail that $\mf{g} = \mf{n} \oplus \mf{h}$ as $A$-modules. Thus, the desired short exact sequence of Lie-Rinehart algebras is provided by the (split) short exact sequence of $A$-modules
\[
\xymatrix @R=15pt{
0 \ar[r] & \mf{n} \ar[r]^-{i_1} & \mf{n} \oplus \mf{h} \ar[r]^-{\pi_2} & \mf{h} \ar[r] &  0.
}
\]
The only non-trivial check consists in observing that $\omega_{\mf{g}}(n) = 0$ for all $n \in \mf{n}$. In fact, we have that
\[[aX,n]_{\mf{g}} = a[X,n]_{\mf{g}} - \omega_{\mf{g}}(n)(a)X\,\in\mf{g}\] 
for all $X \in \mf{h}$ and $a \in A$, and since $[aX,n]_{\mf{g}}, a[X,n]_{\mf{g}} \in \mf{n}$ (because it is a Lie ideal and an $A$-submodule) and since $\omega_{\mf{g}}(n)(a)X \in \mf{h}$ (because it is an $A$-submodule), the hypotheses $\mf{g} = \mf{n} + \mf{h}$ and $\mf{n}\cap\mf{h} = 0$ imply that $\omega_{\mf{g}}(n) = 0$ for all $n \in \mf{n}$. 

{\ref{item:CP3} $\Rightarrow$ \ref{item:CP1}}:
let $\gamma\colon\mf{h} \to \mf{g}$ be a section of $\pi$ as a left $A$-linear map. If we consider the left $A$-linear map
\[\rho\colon\mf{h} \to \Der_{\,\K}(\mf{n}), \qquad X \mapsto \big[\gamma(X),-\big]_{\mf{g}}\]
and the alternating $A$-bilinear form
\[\tau\colon\mf{h} \times \mf{h} \to \mf{n}, \qquad (X,Y) \mapsto \big[\gamma(X),\gamma(Y)\big]_{\mf{g}} - \gamma\big([X,Y]_{\mf{h}}\big),\]
then the data $(\mf{h},\mf{n},\rho,\tau)$ satisfy the conditions of Proposition \ref{prop:curvedsdp} and hence we may form $\mf{n} \niplus_\tau \mf{h}$. The natural left $A$-linear isomorphism
\[\psi\colon\mf{n} \oplus \mf{h} \to \mf{g}, \qquad (n,X) \mapsto n + \gamma(X),\]
then turns into an isomorphism $\psi \colon \mf{n} \niplus_\tau \mf{h} \to \mf{g}$ of Lie-Rinehart algebras over $A$.

The proof of the equivalence between \ref{item:SP1}, \ref{item:SP2}, and \ref{item:SP3} follows the same lines.
\end{proof}


\subsection{Universal enveloping algebras}
\label{regenradar}
 The \emph{universal enveloping algebra} of a Lie-Rinehart algebra $(A, \mf{h}, \go)$ is a $\K$-algebra $\cU_{\ahha}(\mf{h}) $ endowed with a morphism $\iota_{\ahha}\colon A\rightarrow  \cU_{\ahha}(\mf{h}) $ of $\K$-algebras and a morphism $\iota _{\mf{h}}\colon\mf{h}\rightarrow  \cU_{\ahha}(\mf{h}) $ of Lie algebras over $\K$ such that
\begin{equation}
\label{eq:compLRalg}
\iota_{\mf{h}}(aX) = \iota_\ahha (a) \, \iota_{\mf{h}}(X) 
\quad \text{and} \quad 
\iota_{\mf{h}}(X) \, \iota_\ahha (a) -\iota_\ahha(a) \, \iota_{\mf{h}}(X) =\iota_\ahha \big(\omega(X)(a)\big)
\end{equation}
for all $a\in A$ and $X\in \mf{h}$, which is universal with respect to this property. This means that if $\left( \mathcal{U},\phi_\ahha ,\phi _{\mf{h}}\right) $ is another $\K$-algebra with a morphism $\phi_\ahha \colon A\rightarrow \mathcal{U}$ of $\K$-algebras and a morphism $\phi_{\mf{h}}\colon\mf{h}\rightarrow \mathcal{U}$ of $\K$-Lie algebras such that
\begin{equation}
\label{dumdidum}
\phi_{\mf{h}}(aX) = \phi_\ahha (a) \, \phi_{\mf{h}}(X) \quad \text{and} \quad \phi_{\mf{h}}(X) \, \phi_\ahha ( a) -\phi_\ahha (a) \, \phi _{\mf{h}}( X) =\phi_\ahha \big(\omega(X) (a) \big),
\end{equation}
then there exists a unique $\K$-algebra morphism $\Phi \colon \cU_{\ahha}
\left( \mf{h}\right) \rightarrow \mathcal{U}$ such that $\Phi \circ \iota_\ahha =\phi_\ahha $ and $\Phi \circ \iota_{\mf{h}}=\phi_{\mf{h}}$. 
Informally speaking, the universal enveloping algebra is designed in order to have a bijective correspondence (in fact, an isomorphism of categories) between 
$(A,\mf{h})$-module and $\cU_\ahha(\mf{h})$-module structures on an $A$-module $N$.

\begin{remark}
 \label{enhancement}
The above universal property can be stated more compactly as follows. For an $A$-ring $(R,\phi_\ahha )$, set
\[
\cL_\ahha(R) \coloneqq \pig\{(r,\delta) \in R \times \Der_{\K}(A) ~\big\vert~ [\,r,\phi_\ahha (a)\,] = \phi_\ahha\big(\delta(a)\big) \text{ for all }a \in A \pig\}.
\]
This is a Lie-Rinehart algebra over $A$ with anchor given by the projection on $\Der_{\K}(A)$, and with component-wise bracket and $A$-action. In this setting, $ \cU_\ahha \left( \mf{h}\right) $ is an $A$-ring endowed with a morphism $\jmath_{\mf{h}}\colon\mf{h}\rightarrow  \cL_\ahha \big(\,\cU_\ahha \left( \mf{h}\right)\big) $ of Lie-Rinehart algebras such that if $\left( R,\phi _{\mf{h}}\right) $ is another $A$-ring with a morphism $\phi_{\mf{h}}\colon\mf{h}\rightarrow \cL_\ahha (R)$ of Lie-Rinehart algebras, then there exists a unique $A$-ring morphism $\Phi \colon \cU_\ahha  \left( \mf{h}\right) \rightarrow R$ such that $\cL_\ahha (\Phi) \circ \jmath_{\mf{h}}=\phi_{\mf{h}}$. 
See \cite[\S3]{Saracco2} for more details.
\end{remark}

Apart from the well-known constructions of \cite{Rinehart} and \cite{Huebschmann}, the universal enveloping algebra of a Lie-Rinehart algebra $(A,\mf{h},\omega)$ admits other (equivalent) realisations. For instance, the following one (which is coming from \cite{LaiachiPaolo2}) resembles closely the classical construction of the universal enveloping algebra of a Lie algebra. Set $\eta \colon\mf{h}\to \mf{h}\otimes A, X\mapsto X\otimes 1_\ahha$, and consider the tensor $A$-ring $T_{\ahha}(\mf{h} \otimes A) $ on the $A$-bimodule $\mf{h}\otimes A$, that is, 
\[
T_{\ahha}( \mf{h} \otimes A) \coloneqq A \oplus (\mf{h} \otimes A) \oplus \big((\mf{h} \otimes A) \otimes_\ahha  (\mf{h} \otimes A)\big) \oplus \ldots = \bigoplus_{n \geq 0} (\mf{h} \otimes A)^{\otimes_\ahha  n}.\]
It can be shown that 
$ 
 \cU_\ahha \left( \mf{h}\right) \simeq T_{\ahha}\left( \mf{h}\otimes A\right)\big/\cJ,
$ 
where the two-sided ideal $\cJ$ is given by
\[
\cJ \coloneqq\Bigg\langle \left.
\begin{gathered}
\eta (X) \otimes _{\ahha}\eta (Y) -\eta (Y)
\otimes _{\ahha}\eta (X) -\eta \big([ X,Y] \big), 
\\
 \eta (X)  a -a \eta (X)-\omega (X)(a)
\end{gathered}
~\right|~  \begin{gathered} X,Y\in \mf{h}, \\ a\in A \end{gathered}\Bigg\rangle.
\]
We have an algebra morphism $\iota _{\ahha}\colon A\rightarrow \cU_{\ahha}(\mf{h}), a\mapsto a+\cJ,$ and a Lie algebra map $\iota _{\mf{h}}\colon\mf{h}\rightarrow \cU_{\ahha}(\mf{h}), X\mapsto \eta \left( X\right) +\cJ,$ that satisfy the compatibility condition \eqref{eq:compLRalg}. Since an $A$-Lie algebra $\mf{n}$ is equivalent to a Lie-Rinehart algebra over $A$ with zero anchor, a similar construction holds for $A$-Lie algebras. Observe that in this latter case, the relation $\eta(X)  a + \cJ = a  \eta(X) + \cJ$  for all $a\in A$ and $X \in \mf{n}$ entails that
\[
X \otimes a + \cJ = (X \otimes 1) a + \cJ = \eta(X)  a + \cJ = a  \eta(X) + \cJ = a  (X \otimes 1) + \cJ = (a X) \otimes 1 + \cJ,
\]
and hence the foregoing construction is isomorphic to the ordinary universal enveloping algebra $U_{\ahha}(\mf{n})$ of an $A$-Lie algebra.

In what follows, we will often assume that $\mf{h}$ is a projective left $A$-module so that the Poincar\'e–Birkhoff–Witt (PBW) theorem holds, that is to say, $\gr\big(\,\cU_\ahha(\mf{h})\big) \simeq \cS_\ahha(\mf{h})$ as commutative $A$-algebras (see \cite[\S3]{Rinehart}). As a consequence, we will identify $A$ and $\mf{h}$ with their images in $\cU_\ahha(\mf{h})$, that is to say, we will often omit to write explicitly $\iota_\ahha$ or $\iota_{\mf{h}}$.

\begin{remark}
  The universal enveloping algebra $\cU_\ahha(\mf{h})$ of a Lie-Rinehart algebra $\mf{h}$ is naturally an $A$-coalgebra with respect to the left $A$-module structure $a  u \coloneqq \iota_\ahha(a)\,u$ for all $a \in A$, $u \in \cU_\ahha(\mf{h})$, and with comultiplication and counit uniquely determined by 
  \begin{gather*}
  \label{primitive}
  \Delta(a) = a \ot_\ahha 1, \qquad \Delta\left(X_1\cdots X_n\right) = {\sum_{\substack{t=0,\ldots, n \\ \sigma \in W_{t,n-t}}}} X_{\sigma(1)}\cdots X_{\sigma(t)} \ot_\ahha X_{\sigma(t+1)}\cdots X_{\sigma(n)}, \\
  \varepsilon(a) = a \qquad \text{and} \qquad \varepsilon\left(X_1\cdots X_n\right) = 0
  \end{gather*}
  for all $a \in A$ and $X_1,\ldots,X_n \in \mf{h}$, where $W_{t,n-t}$ are the $(t,n-t)$-shuffles. For instance, 
  \begin{equation}\label{eq:DeltaXY}
  \Delta(X) = X \ot_\ahha 1 + 1 \ot_\ahha X \quad \text{and} \quad \Delta(XY) = XY \ot_\ahha 1 + X \ot_\ahha Y + Y \ot_\ahha X + 1 \ot_\ahha XY
  \end{equation}
  for $X,Y \in \mf{h}$.
\end{remark}

\subsection{A symmetrisation map}
\label{ssec:sym}

Let $(A, \mf{g}, \omega)$ be a Lie-Rinehart algebra which is projective as a left $A$-module and let $\{\chi_i,\varphi_i\mid i \in I\}$ be a projective basis for $\mf{g}$, where $\{\chi_i\mid i \in I\}$ stands for a generating set of $\mf{g}$, while the $\{\varphi_i\mid i \in I\}$ are the corresponding $A$-linear forms in $\mf{g}^*=\Hom{\ahha}{}{}{}{\mf{g}}{A}$. Consider $\mf{g}$ as an $A$-bimodule with right $A$-action induced by the left one. In this way, we may consider the symmetric algebra $\cS_\ahha (\mf{g})$ over $A$, and it is easily seen that for all $k \geq 1$,
\begin{equation}\label{eq:symmetr}
\textsc{S}_k \colon \cS_\ahha ^k(\mf{g}) \to \cU_\ahha(\mf{g})_{\leq k}, \qquad X_1 \cdots X_k \mapsto {\textstyle\frac{1}{k!}}\smashoperator[l]{\sum_{\substack{ j_1,\ldots,j_k \in I \\ \sigma\in\mf{S}_k}} }\varphi_{j_1}(X_1)\cdots\varphi_{j_k}(X_k)\,\chi_{j_{\sigma(1)}}\cdots\chi_{j_{\sigma(k)}}
\end{equation}
is a well-defined morphism of $A$-bimodules, where the $A$-bimodule structure on the codomain is coming from the left $A$-module structure.
For instance, $\textsc{S}_1=\iota_{\mf{g}}\colon \cS^1_\ahha(\mf{g}) = \mf{g} \hookrightarrow  \cU_\ahha(\mf{g})_{\leq 1}$ and 
\begin{equation}\label{eq:case2}
\textsc{S}_2 \colon X \otimes_{\ahha} Y \mapsto {\textstyle{\frac{1}{2}}}\sum_{i,j \in I}\varphi_i(X)\varphi_j(Y)(\chi_i\chi_j + \chi_j\chi_i)
\end{equation} 
for $X, Y \in \mf{g}$.
Together with $\textsc{S}_0\colon \cS^0_\ahha (\mf{g}) = A \xrightarrow{=} A = \cU_\ahha(\mf{g})_{\leq 0}$, the morphisms \eqref{eq:symmetr} induce an $A$-bilinear morphism
\begin{equation}
\label{eq:symmetrisation}
\textsc{S}\colon\cS_\ahha (\mf{g}) \to \cU_\ahha(\mf{g}),
\end{equation}
which plays the r\^ole of the well-known symmetrisation map (if $A = \K$, the base field, this \emph{is} the usual symmetrisation map). We prove now that $\textsc{S}$ is, in fact, an isomorphism of $A$-corings.

\begin{remark}
\label{rem:psiinj}
We can always rewrite the sum in \eqref{eq:symmetr} as
\begin{equation}
\label{eq:arciuffa}
\begin{split}
&   \phantom{=} \sum_{\substack{ j_1,\ldots,j_k \\ \sigma\in\mf{S}_k} }\varphi_{j_1}\left(X_1\right)\cdots\varphi_{j_k}\left(X_k\right)\,\chi_{j_{\sigma(1)}}\cdots\chi_{j_{\sigma(k)}} 
    \\
    &= \sum_{\sigma \in \mf{S}_k}\,\sum_{ j_{\sigma(1)},\ldots,j_{\sigma(k)} }\varphi_{j_{\sigma(1)}}\left(X_{\sigma(1)}\right)\cdots\varphi_{j_{\sigma(k)}}\left(X_{\sigma(k)}\right)\,\chi_{j_{\sigma(1)}}\cdots\chi_{j_{\sigma(k)}} 
    \\
     &= \sum_{\sigma \in \mf{S}_k}\,\sum_{ j_{1},\ldots,j_{k} }\varphi_{j_{1}}\left(X_{\sigma(1)}\right)\cdots\varphi_{j_{k}}\left(X_{\sigma(k)}\right)\,\chi_{j_{1}}\cdots\chi_{j_{k}},
\end{split}
\end{equation}
where we used that $A$ is commutative (hence  $\varphi_{j_1}(X_1)\cdots\varphi_{j_k}(X_k)=\varphi_{j_{\sigma(1)}}(X_{\sigma(1)})\cdots\varphi_{j_{\sigma(k)}}(X_{\sigma(k)})$ for any permutation $\sigma\in\mf{S}_k$) for the first equality and that the indexes $j_m$'s in the summation are dummy for the second one.
\end{remark}

\begin{lemma}\label{lem:symmetrisation}
Let $(A, \mf{g}, \omega)$ be a Lie-Rinehart algebra which is projective as a left $A$-module. The morphism $\textsc{S}\colon\cS_\ahha (\mf{g}) \to \cU_\ahha(\mf{g})$ of \eqref{eq:symmetrisation} is an isomorphism of $A$-bimodules, where the bimodule structure on $\cS_\ahha (\mf{g})$ is the one induced by the $A$-algebra structure and the $A$-bimodule structure on $\cU_\ahha(\mf{g})$ is the one induced by the left $A$-module structure.
\end{lemma}

\begin{proof}
Consider $\cU_\ahha(\mf{g})$ as a $\Z$-filtered left $A$-module via
\[
F_n\big(\cU_\ahha(\mf{g})\big) \coloneqq
\begin{cases}
\cU_\ahha(\mf{g})_{\leq n} & n \geq 0, 
\\
0 & n < 0.
\end{cases}
\]
Analogously, consider $\cS_\ahha (\mf{g})$ as a $\Z$-filtered left $A$-module with filtration induced by the graduation
\[
F_n\big(\cS_\ahha (\mf{g})\big) \coloneqq 
\begin{cases}
\displaystyle \bigoplus_{k=0}^n\cS_\ahha ^k(\mf{g}) & n \geq 0,
\\
0 & n < 0.
\end{cases}
\]
Notice that both filtrations, the one on $\cU_\ahha(\mf{g})$ and the one on $\cS_\ahha (\mf{g})$, are separated (that is to say, the intersection $\bigcap_{n \in \Z}F_n$ is $0$), exhaustive (that is to say, the union $\bigcup_{n \in \Z}F_n$ is the whole $A$-module), and discrete (that is to say, $F_n=0$ for all $n < 0$). In particular, both of them are complete filtered $A$-modules (see \cite[Ch.\ D, \S{I} \& \S II]{Nastasescu-vanOystaeyen}). Moreover, by construction, $\textsc{S}$ is a filtered morphism of filtered left $A$-modules and if we consider the graded associated $\gr(\textsc{S})$, then it satisfies
\[
\cS_\ahha ^k(\mf{g}) \simeq \gr_k\big(\cS_\ahha (\mf{g})\big) \xrightarrow{\gr_k(\textsc{S})} \gr_k\big(\cU_\ahha(\mf{g})\big), \qquad X_1\cdots X_k \mapsto X_1\cdots X_k + \cU_\ahha(\mf{g})_{\leq k-1}
\]
for all $k \geq 0$, and $0$ elsewhere. That is to say, $\gr(\textsc{S})$ is the isomorphism of the PBW theorem. Then $\textsc{S}$ is an isomorphism of $A$-bimodules, in view of \cite[Ch.\ D, Cor.\ III.5]{Nastasescu-vanOystaeyen}. 
\end{proof}

\begin{theorem}\label{thm:isocoring}
Let $(A, \mf{g}, \omega)$ be a Lie-Rinehart algebra which is projective as left $A$-module. 
The morphism $\textsc{S}\colon\cS_\ahha (\mf{g}) \to \cU_\ahha(\mf{g})$ of \eqref{eq:symmetrisation} is an isomorphism of $A$-corings.
\end{theorem}

For the sake of readability, the full proof is postponed to \S\ref{sssec:proof3.3}. Here below, we only explicitly show the comultiplicativity of $\textsc{S}$ on a homogeneous element of degree $2$. The general case is not different. 

\begin{example}
For all $X,Y \in \mf{g}$ we have
\begin{align*}
& \Delta_{\cU}\big(\textsc{S}(XY)\big) \stackrel{\scriptscriptstyle\eqref{eq:case2}}{=} \frac{1}{2}\sum_{i,j} \varphi_i(X)\varphi_j(Y)\Delta_\cU\big(\chi_i\chi_j + \chi_j \chi_i\big) \\
 & \stackrel{\scriptscriptstyle\eqref{eq:DeltaXY}}{=} \frac{1}{2}\sum_{i,j} \varphi_i(X)\varphi_j(Y)\Big(\big(\chi_i\chi_j + \chi_j\chi_i\big)\ot_\ahha 1 + 2\,\chi_i\ot_\ahha \chi_j + 2\,\chi_j \ot_\ahha \chi_i + 1 \ot_\ahha \big(\chi_i\chi_j + \chi_j\chi_i\big)\Big) \\
 & \stackrel{\scriptscriptstyle\eqref{eq:case2}}{=} \textsc{S}(XY) \ot_\ahha 1 +  X \ot_\ahha Y +  Y \ot_\ahha X + 1 \ot_\ahha \textsc{S}(XY) \\
  & \stackrel{\phantom{\eqref{eq:case2}}}{=} \textsc{S}(XY) \ot_\ahha \textsc{S}(1) +  \textsc{S}(X) \ot_\ahha \textsc{S}(Y) +  \textsc{S}(Y) \ot_\ahha \textsc{S}(X) + \textsc{S}(1) \ot_\ahha \textsc{S}(XY) \\
 & \stackrel{\phantom{\eqref{eq:DeltaXY}}}{=} (\textsc{S} \ot_\ahha \textsc{S})\big(\Delta_\cS(XY)\big). \qedhere
\end{align*}
\end{example}

Following Theorem \ref{thm:isocoring}, the forthcoming proposition answers our question \ref{item:A} by providing a Lie-Rinehart analogue of the well-known fact that a universal enveloping algebra is always free over any universal enveloping subalgebra (see \cite[Prop.\  2.2.7]{Dixmier}).

\begin{proposition}\label{prop:anothersplitting}
Let $\mf{h} \subseteq \mf{g}$ be an inclusion of Lie-Rinehart algebras which are projective as left $A$-modules. Suppose that the quotient $A$-module $\mf{g}/\mf{h}$ is projective, too. Then we have an isomorphism
\[\cU_\ahha(\mf{g}) \simeq \cU_\ahha(\mf{h}) \otimes_\ahha  \cS_\ahha(\mf{g}/\mf{h})\]
as left $\cU_\ahha(\mf{h})$-modules. In particular, $\cU_\ahha(\mf{g})$ is projective over $\cU_\ahha(\mf{h})$.
\end{proposition}

\begin{proof}
Denote by $\iota\colon \mf{h} \subseteq \mf{g}$ the inclusion.
By the standing hypotheses, we may fix a left $A$-linear section $\sigma\colon\mf{g}/\mf{h} \to \mf{g}$ of the canonical projection $\pi\colon\mf{g} \to \mf{g}/\mf{h}$ and we have that $\mf{g} \simeq \mf{h} \oplus \mf{g}/\mf{h}$ as left $A$-modules. Notice that 
\[\cS_\ahha(\mf{h} \oplus \mf{g}/\mf{h}) \simeq \cS_\ahha(\mf{h}) \otimes_\ahha  \cS_\ahha(\mf{g}/\mf{h})\]
as $A$-algebras via the unique morphism
\[\cS_\ahha(\mf{h}) \otimes_\ahha  \cS_\ahha(\mf{g}/\mf{h}) \to \cS_\ahha(\mf{h} \oplus \mf{g}/\mf{h})\] induced by the canonical maps $\mf{h} \to \mf{h} \oplus \mf{g}/\mf{h} \leftarrow \mf{g}/\mf{h}$ and the universal property of the coproduct $\cS_\ahha(\mf{h}) \otimes_\ahha  \cS_\ahha(\mf{g}/\mf{h})$ applied to the resulting $A$-algebra morphisms
\[\cS_\ahha(\mf{h}) \to \cS_\ahha(\mf{h} \oplus \mf{g}/\mf{h}) \qquad \text{and} \qquad \cS_\ahha(\mf{g}/\mf{h}) \to \cS_\ahha(\mf{h} \oplus \mf{g}/\mf{h}).\]
Therefore, the $A$-algebra morphisms
\[\cS_\ahha(\iota) \colon \cS_\ahha(\mf{h}) \to \cS_\ahha(\mf{g}) \qquad \text{and} \qquad \cS_\ahha(\sigma) \colon \cS_\ahha(\mf{g}/\mf{h}) \to \cS_\ahha(\mf{g})\]
induce an isomorphism of graded $A$-modules (in fact, $A$-algebras)
\begin{equation}
\label{eq:ancoraunosforzo}
\begin{array}{rcl}
\cS_\ahha(\mf{h}) \otimes_\ahha  \cS_\ahha(\mf{g}/\mf{h}) & \longrightarrow & \cS_\ahha(\mf{g}), 
\\  
U_1 \cdots U_m\otimes_\ahha  \pi(X_1)\cdots \pi(X_n) & \longmapsto & \iota(U_1) \cdots \iota(U_m)\sigma\pi(X_1)\cdots \sigma\pi(X_n)
\end{array}
\end{equation}
with explicit inverse uniquely determined by
\[
\cS_\ahha(\mf{g}) \to \cS_\ahha(\mf{h}) \otimes_\ahha  \cS_\ahha(\mf{g}/\mf{h}), \qquad X \mapsto \big(X-\sigma\pi(X)\big) \otimes_\ahha  1 + 1 \otimes_\ahha  \pi(X).
\]
In this setting, consider the composition
\begin{equation}\label{eq:celapossofare}
\Theta\colon \cU_\ahha(\mf{h}) \otimes_\ahha  \cS_\ahha(\mf{g}/\mf{h}) \xrightarrow{\quad \cU_\ahha(\iota)\otimes_\ahha  \left(\textsc{S}_\mf{g}\circ\cS_\ahha(\sigma)\right) \quad} \cU_\ahha(\mf{g}) \otimes_\ahha  \cU_\ahha(\mf{g}) \to \cU_\ahha(\mf{g}),
\end{equation}
where the right-most arrow is simply the multiplication in $\cU_\ahha(\mf{g})$ and the tensor product over $A$ in the domain $\cU_\ahha(\mf{h}) \otimes_\ahha  \cS_\ahha(\mf{g}/\mf{h})$ is taken with respect to the regular right $A$-module structure $\cU_\ahha(\mf{h})_\ahha$ on $\cU_\ahha(\mf{h})$ and the regular left $A$-module structure on $\cS_\ahha(\mf{g}/\mf{h})$. This is a filtered morphism with respect to the filtration
\[
F_n\big(\cU_\ahha(\mf{h}) \otimes_\ahha  \cS_\ahha(\mf{g}/\mf{h})\big) 
= 
\smashoperator[r]{\sum_{h+k=n}} \, \pig(F_h\big(\cU_\ahha(\mf{h})\big) \otimes_\ahha  F_k\big(\cS_\ahha(\mf{g}/\mf{h})\big)\pig)
\]
on the domain and the primitive filtration on the codomain. Since $\gr_k\left(\cU_\ahha(\mf{h})\right) \simeq \cS_\ahha^k(\mf{h})$ and $\gr_h\left(\cS_\ahha(\mf{g}/\mf{h})\right) = \cS_\ahha^h(\mf{g}/\mf{h})$ are projective $A$-bimodules for all $h,k \geq 0$ (where, as usual, the $A$-bimodule structure is the symmetric one induced by the left $A$-module structure), it follows that the canonical morphism
\begin{eqnarray*}
\gr\big(\cU_\ahha(\mf{h})\big) \otimes_\ahha  \gr\big(\cS_\ahha(\mf{g}/\mf{h})\big) & \longrightarrow & \gr\big(\cU_\ahha(\mf{h}) \otimes_\ahha  \cS_\ahha(\mf{g}/\mf{h})\big), 
\\
\pig(u + F_h\big(\cU_\ahha(\mf{h})\big)\pig) \otimes_\ahha  \pig(s + F_k\big(\cS_\ahha(\mf{g}/\mf{h})\big)\pig) & \longmapsto & u \otimes_\ahha  s + F_{h+k+1}\big(\cU_\ahha(\mf{h}) \otimes_\ahha  \cS_\ahha(\mf{g}/\mf{h})\big)
\end{eqnarray*}
is an isomorphism of graded $A$-bimodules. This is a consequence of, for example, \cite[Thm.\  C.24, p.\ 93]{Majewski}, or one may adapt directly \cite[Lem.\  1.1]{Saracco} or \cite[Appendix B.3]{LaiachiPaolo2}). The graded morphism associated to the composition \eqref{eq:celapossofare} is exactly the isomorphism \eqref{eq:ancoraunosforzo}, up to the last isomorphism. Since the filtrations on $\cU_\ahha(\mf{h}) \otimes_\ahha  \cS_\ahha(\mf{g}/\mf{h})$ and $\cU_\ahha(\mf{g})$ are discrete (that is, there exists $n_0 \in \Z$ such that $F_i = 0$ for all $i \leq n_0$) and exhaustive (the union of all the terms give the entire bimodule), they are separated and complete and hence \cite[Ch.\ D, Cor.\ III.5]{Nastasescu-vanOystaeyen} implies that \eqref{eq:celapossofare} is an isomorphism of $A$-bimodules. It is clearly left $\cU_\ahha(\mf{h})$-linear.
The final claim follows from the fact that $\cS_\ahha(\mf{g}/\mf{h})$ is a projective left $A$-module.
\end{proof}

\begin{example}
For the sake of having a concrete instance of the final argument of the proof of Proposition \ref{prop:anothersplitting}, consider the element $UV \otimes_\ahha  \pi(X)\pi(Y)$ in $\cS_\ahha(\mf{h}) \otimes_\ahha  \cS_\ahha(\mf{g}/\mf{h})$. Then the composition
\[\cS_\ahha(\mf{h}) \otimes_\ahha  \cS_\ahha(\mf{g}/\mf{h}) \simeq \gr\big(\cU_\ahha(\mf{h})\big) \otimes_\ahha  \gr\big(\cS_\ahha(\mf{g}/\mf{h})\big) \simeq \gr\big(\cU_\ahha(\mf{h}) \otimes_\ahha  \cS_\ahha(\mf{g}/\mf{h})\big) \xrightarrow{\gr(\Theta)} \gr\big(\cU_\ahha(\mf{g})\big) \simeq \cS_\ahha(\mf{g})\]
acts on $UV \otimes_\ahha  \pi(X)\pi(Y)$ as 
\begin{align*}
UV \otimes_\ahha  \pi(X)\pi(Y) & \mapsto \pig(UV + F_1\big(\cU_\ahha(\mf{h})\big) \pig) \otimes_\ahha  \pig(\pi(X)\pi(Y) + F_1\big(\cS_\ahha(\mf{g}/\mf{h})\big) \pig) 
\\
 & \mapsto \big(UV  \otimes_\ahha  \pi(X)\pi(Y) \big) + F_3\big(\cU_\ahha(\mf{h}) \otimes_\ahha  \cS_\ahha(\mf{g}/\mf{h})\big) 
 \\
 & \mapsto \frac{1}{2}\iota(U)\iota(V)\pig(\sum_{i,j \in I}\varphi_i(\sigma\pi(X))\varphi_j(\sigma\pi(Y))(\chi_i\chi_j + \chi_j\chi_i)\pig) + F_3\big(\cU_\ahha(\mf{g})\big) \\
 & = \iota(U)\iota(V)\sigma\pi(X)\sigma\pi(Y) + F_3\big(\cU_\ahha(\mf{g})\big) \\
 & \mapsto \iota(U)\iota(V)\sigma\pi(X)\sigma\pi(Y). \qedhere
\end{align*}
\end{example}

\begin{remark}
\label{rem:KISS}
\begin{enumerate}[label=(\alph*),leftmargin=0.6cm]
    \item By using a right-hand analogue of the symmetrisation map $\textsc{S}$, one can prove a variation on Proposition \ref{prop:anothersplitting}, providing an isomorphism  of right $\cU_\ahha(\mf{h})$-modules of the form 
    \[\cU_\ahha(\mf{g}) \simeq \cS_\ahha(\mf{g}/\mf{h}) \otimes_\ahha  \cU_\ahha(\mf{h}),\]
allowing to conclude 
    \begin{equation}\label{eq:isoCalaque}
    \cU_\ahha(\mf{g})\,/\,\cU_\ahha(\mf{g})\,\mf{h} \simeq \cS_\ahha(\mf{g}/\mf{h})
    \end{equation}
as left $A$-modules.
    For details on the isomorphism \eqref{eq:isoCalaque} (and its analogue for universal enveloping algebras of ordinary Lie algebras) and its geometric implications see \cite{Calaque1,Calaque2}.
        \item
    Observe that another $A$-coring isomorphism $\cS_\ahha(\mf{g}) \to \cU_\ahha(\mf{g})$ is provided by the pbw map of \cite{PBW}, but the symmetrisation map $\textsc{S}$ does not coincide with it, in general. \qedhere
\end{enumerate}
\end{remark}

\subsection{Left bialgebroids}
\label{bialgebroids1}
First introduced in {\cite{Takeuchi}}, \emph{left bialgebroids} (also known as \emph{$\times_\ahha$-bialgebras}) are a generalisation of $\K$-bialgebras to bialgebra objects over a noncommutative base ring $A$, consisting of compatible $\Ae$-ring and $A$-coring structures on the same $\K$-vector space $U$. In particular, to define a left bialgebroid $(U, A, \gD, \gve, s, t)$, or $(U,A)$ for short, one starts with a ring homomorphism $s\colon A \to U$ (called the \emph{source}) and a ring anti-homomorphism $t\colon A \to U$ (called the {\em target}) which commute, in the sense that $s(a)t(b) = t( b)s(a)$ for all $a,b \in A$. These induce four commuting $A$-module structures on $U$, denoted by
\begin{equation}
\label{pergolesi}
  a \blact b \lact u \ract c \bract d  \coloneqq t(c)s(b)u\,s(d)t(a)
\end{equation}
  for $u \in U, \, a,b,c,d \in A$, which we abbreviate by
$
\due U {\blact \lact} {\ract \bract}
$, depending on the relevant action(s) in a specific situation. 
For example, the tensor product $U_\ract \otimes_\ahha  \due U \lact {}$ has to be understood as
$$
U_\ract \otimes_\ahha  \due U \lact {}
   \coloneqq U \otimes U/\,
\mathrm{span}\big\{t( a) u \otimes v - u \otimes s(a)v, \ \forall u,v \in U, \forall a \in A  \big\}.
$$
Next, one introduces a counit $\gve\colon U \to A$ subject to
\begin{equation}
\label{beethoven1}
\gve(b \lact u \ract c) =  b\,\gve(u)\,c, \quad
\gve(a \blact u) = \gve(u \bract a), \quad
\gve(uv) =  \gve\big(u \bract \gve(v)\big) = \gve\big(\gve(v) \blact u\big), 
\end{equation}
that is, which is linear only with respect to the ``white'' $A$-actions and defines on $A$ the structure of a $U$-module by means of the ``black'' ones, and which in particular
is not a morphism of algebras, in striking contrast to the bialgebra case. 
Moreover, apart from the multiplication $m_\uhhu$, one also has a coassociative comultiplication 
 $$
 \Delta\colon U \to \due U {} \ract \otimes_\ahha  \due U \lact {}, \ u \mapsto u_{(1)} \otimes_\ahha  u_{(2)},
 $$
 that corestricts to the {\em Sweedler-Takeuchi product}
 \[
U \times_{\scriptscriptstyle A} U   \coloneqq
   \big\{ {\textstyle \sum_i} u_i \otimes  v_i \in U_{\!\ract}  \otimes_\ahha  \!  \due U \lact {} \mid {\textstyle \sum_i} a \blact u_i \otimes v_i = {\textstyle \sum_i} u_i \otimes v_i \bract a,  \ \forall a \in A  \big\} \subseteq \due U {} \ract \otimes_\ahha  \due U \lact {},
\]
 which is an $\Ae$-ring via factor-wise multiplication. The use of the Sweedler-Takeuchi product is that it provides the only way of giving a well-defined sense to the compatibility between product and coproduct in the sense
 of
 \begin{equation}
     \label{beethoven2}
     \gD(uv) = \gD(u)\gD(v), \quad \gD( a \blact b \lact u \ract c \bract d) = \big( b \lact u_{(1)} \bract d\,\big) \otimes_\ahha  \big(a \blact u_{(2)} \ract c\,\big),
 \end{equation}
 for $u, v \in U$ and $a,b,c,d \in A$. Furthermore, usually the coproduct is asked to be counital, that is,
 \begin{equation}
 \label{beethoven3}
 m_\uhhu\circ (s\gve \otimes_\ahha  \id_\uhhu) \circ \gD = \id_\uhhu =  m_\Uopp \circ (\id_\uhhu \otimes_\ahha  t\gve) \circ \gD. 
\end{equation}
If the rings $U$  and $A$ are unital, the coproduct and the counit are asked to be so as well, that is
$\gD(1_\uhhu) = 1_\uhhu \otimes_\ahha  1_\uhhu$ and $ \gve(1_\uhhu) = 1_\ahha$,
(see, for instance, \cite{BoeSzl:HAWBAAIAD, Takeuchi}). A {\em morphism} $\pi\colon (U,A) \to (V,A)$ of left bialgebroids over the same base algebra is a map that commutes with all structure maps in an obvious way, that is
\begin{equation}
\label{krach1}
\begin{gathered}
\pi \circ m_\uhhu = m_\vauu \circ (\pi \otimes_\ahha  \pi), \qquad (\pi \otimes_\ahha  \pi)\circ \gD_\uhhu = \gD_\vauu \circ \pi, \\
\gve_\uhhu = \gve_\vauu \circ \pi, \qquad \pi \circ s_\uhhu = s_\vauu, \qquad  \pi \circ t_\uhhu = t_\vauu.
\end{gathered}
\end{equation}

For the sake of completeness, let us mention that \emph{right bialgebroids} can be defined analogously by switching the r\^oles of black and white actions from \eqref{pergolesi}.

\begin{remark}\label{rem:monoidal}
Given a left bialgebroid $(U,A)$, any left module $\left(M,\cdot\right)$ over the $\K$-algebra $U$ inherits an $A$-bimodule structure via
\begin{equation}\label{eq:restscal}
A \otimes M \otimes A \to M, \qquad a' \ot m \ot a \mapsto a'\lact m \ract a \coloneqq s(a')t(a) m.
\end{equation}
In particular, given two left $U$-modules $M,N$, one may consider their tensor product $M \otimes_\ahha  N$ as $A$-bimodules and this naturally becomes a left $\left(U \tak{\ahha} U\right)$-module via
\[
(U \tak{\ahha} U) \otimes (M \otimes_\ahha  N) \to M \otimes_\ahha  N, \quad \textstyle\big(\sum\limits_i u_i \ot_\ahha v_i\big) \ot (m \ot_\ahha n) \mapsto \sum\limits_i u_i m \ot_\ahha v_i n
\]
and then a left $U$-module via restriction of scalars along the algebra morphism $\Delta\colon U \to U \tak{\ahha} U$. Analogously, the base algebra $A$ is naturally a left $\End{\K}{A}$-module by evaluation and then a left $U$-module by restriction of scalars along the map
\begin{equation}\label{eq:UAact}
U \to \End{\K}{A}, \quad u \mapsto \big\{\,a \mapsto \varepsilon(u \bract a)\,\big\},
\end{equation}
which is an algebra morphism due to the third condition in \eqref{beethoven1}.
The category $\umod$ of left $U$-modules is a monoidal category in the sense of \cite[\S VII.1]{MacLane}, with tensor product $\ot_\ahha$ and unit object $A$.
\end{remark}

A distinguished class of left bialgebroids is that of cocommutative left bialgebroids. A left bialgebroid $U$ with $s = t$ over a commutative base algebra $A$ is said to be a cocommutative left bialgebroid if the comultiplication is cocommutative, that is to say, if $\mathsf{tw} \circ \Delta = \Delta$, where $\mathsf{tw}(x \otimes_\ahha  y) = y \otimes_\ahha  x$ for all $x,y \in U$ (in fact, cocommutativity of $\Delta$ implies that source and target should coincide and hence that $A$ can be assumed to be commutative). 

\subsection{Left Hopf algebroids} 
\label{bialgebroids2}

Generalising Hopf algebras ({\em i.e.}, bialgebras with an antipode or, equivalently, whose Hopf-Galois map is invertible), given a left bialgebroid $(U, A)$ one considers the following map of left $U$-modules:
\begin{equation}
  \label{nochmehrRegen}
\gb
\colon \due U \blact {} \otimes_{\Aopp} U_\ract \to U_\ract  \otimes_\ahha   \due U \lact, \qquad u \otimes_\Aopp v  \mapsto  u_{(1)} \otimes_\ahha  u_{(2)}  v, 
\end{equation}
called the \emph{Hopf-Galois map}, where
$$
\due U \blact {} \otimes_{\Aopp} U_\ract {}
   \coloneqq U \otimes U/\,
\mathrm{span}\big\{ u\, t( a) \otimes v - u \otimes t(a)\,v, \ \forall u,v \in U, \forall a \in A  \big\}.
$$
Then the left bialgebroid $(U,A)$ is called a {\em left Hopf algebroid} (also known as \emph{$\times_\ahha $-Hopf algebra}), or simply {\em left Hopf}, if $\gb$ is invertible (see {\cite[Def.~3.5]{Schauenburg}}). 
By adopting a kind of Sweedler notation for the so-called \emph{translation map} $\beta^{-1}(\cdot\,\otimes_\ahha  1)\colon U\to{}_\blact U \! \otimes_\Aopp \! U_\ract
$\,:
\begin{equation*}
 u_+ \otimes_\Aopp u_-  \coloneqq  \beta^{-1}(u \otimes_\ahha  1)
\end{equation*}
with, as usual, summation understood, one proves that for a left Hopf algebroid
\begin{align}
\label{Sch1}
u_+ \otimes_\Aopp  u_- & \in
 U \times_\Aopp U, & &  \\
\label{Sch2}
u_{+(1)} \otimes_\ahha  u_{+(2)} u_- &= u \otimes_\ahha  1  & & \textrm{in } U_{\!\ract} \! \otimes_\ahha  \! {}_\lact U,  \\
\label{Sch3}
u_{(1)+} \otimes_\Aopp u_{(1)-} u_{(2)}  &= u \otimes_\Aopp  1 & & \textrm{in }  {}_\blact U \! \otimes_\Aopp \! U_\ract,  \\
\label{Sch4}
u_{+(1)} \otimes_\ahha  u_{+(2)} \otimes_\Aopp  u_{-} &= u_{(1)} \otimes_\ahha  u_{(2)+} \otimes_\Aopp u_{(2)-} & & \textrm{in } \due U {} {\ract} \otimes_\ahha  \due U {\lact \blact} {} \otimes_\Aopp \due U {} {\ract} ,  \\
\label{Sch5}
u_+ \otimes_\Aopp  u_{-(1)} \otimes_\ahha  u_{-(2)} &= u_{++} \otimes_\Aopp u_- \otimes_\ahha  u_{+-} & &  \textrm{in } \due U {\blact} {} \otimes_\Aopp\due U {} {\ract} \otimes_\ahha  \due U {\lact} {\ract},  \\
\label{Sch6}
(uv)_+ \otimes_\Aopp  (uv)_- &= u_+v_+ \otimes_\Aopp v_-u_- & & \textrm{in }  {}_\blact U \! \otimes_\Aopp \! U_\ract,
\\
\label{Sch7}
u_+u_- &= s (\varepsilon (u)), & &  \\
\label{Sch8}
\varepsilon(u_-) \blact u_+  &= u, & &  \\
\label{Sch9}
(s (a) t (b))_+ \otimes_\Aopp  (s (a) t (b) )_-
&= s (a) \otimes_\Aopp s (b) & &  \textrm{in }  {}_\blact U \! \otimes_\Aopp \! U_\ract,
\end{align}
hold \cite[Prop.~3.7]{Schauenburg}, 
where in  \eqref{Sch1}  we mean the Takeuchi-Sweedler product
\begin{equation*}
\label{petrarca}
   U \! \times_\Aopp \! U    \coloneqq
   \big\{ {\textstyle \sum_i} u_i \otimes v_i \in {}_\blact U  \otimes_\Aopp  U_{\!\ract} \mid {\textstyle \sum_i} u_i \ract a \otimes v_i = {\textstyle \sum_i} u_i \otimes a \blact v_i, \ \forall a \in A \big\}.
\end{equation*}

A {\em morphism} $\pi\colon (U,A) \to (V,A)$ of left
Hopf algebroids over the same base is a morphism of the underlying left bialgebroids as specified in \eqref{krach1} that additionally fulfils
\begin{equation}
\label{krach2}
\gb_\vauu^{-1} \circ \pi = (\pi \otimes_\Aopp \pi) \circ \beta^{-1}_\uhhu.
\end{equation}
One easily verifies that a bialgebroid morphism between the underlying bialgebroid structure between two left Hopf algebroids is automatically a left Hopf algebroid morphism.

Observe that the canonical map \eqref{nochmehrRegen} is not the only one of this kind: a different choice leads to the notion of {\em right} Hopf algebroids.

\begin{example}[Universal enveloping algebra of a Lie-Rinehart algebra]
\label{ex:UEAbialgd}
It can be proven that the universal enveloping algebra $\cU_\ahha(\mf{h})$ of a Lie-Rinehart algebra $\mf{h}$ is a cocommutative left bialgebroid over $A$ (see \cite[Thm.~3.7]{Xu:QG}) and, in fact, even a left Hopf algebroid (see \cite[Ex.~2]{KowKra:DAPIACT}). 
The structure maps are uniquely determined by
\begin{gather}
\label{mainsomma1}
\Delta(X) = X \otimes_\ahha  1 + 1 \otimes_\ahha  X, \qquad \varepsilon(X) = 0, \\
\label{mainsomma2}
\beta^{-1}(X \otimes_\ahha  1) = X \otimes_\Aopp 1 - 1 \otimes_\Aopp X
\end{gather}
for all $X \in \mf{h}$. In particular,  the action \eqref{eq:UAact} satisfies
\begin{equation}
\varepsilon(Xa) \stackrel{\scriptscriptstyle\eqref{eq:compLRalg}}{=} \varepsilon(aX) + X(a) \stackrel{\scriptscriptstyle\eqref{mainsomma1}}{=} X(a) \label{eq:Xfactor}
\end{equation}
for all $X \in \mf{h}$ and $a \in A$.
\end{example}

\begin{remark}
A left Hopf structure on a left bialgebroid still does not imply the existence of an antipode as in \cite{BoeSzl:HAWBAAIAD}. 
For example, the universal enveloping algebra $\cU_\ahha(\mf{h})$ of a Lie-Rinehart algebra $(A,\mf{h},\omega)$ admits an antipode if $A$ is equipped with a {\em flat right $(A,\mf{h})$-connection} \cite[Prop.~3.12]{KowPos:TCTOHA}, which, however, may not always exist \cite{AnaUli}.

However, in case $(U, A) = (H, \K)$ is a Hopf algebra over a field $\K$, then the invertibility of 
$\gb$ guarantees the existence of the antipode $S$. 
In this case, one has $h_+ \otimes_\K h_- = h_{(1)} \otimes_\K S(h_{(2)})$ for any $h \in H$.
\end{remark}

In view of Remark \ref{rem:monoidal} and Example \ref{ex:UEAbialgd}, it is a well-known fact that the isomorphism of categories between the category of representations of a Lie-Rinehart algebra and the category of modules over its universal enveloping algebra from \S\ref{regenradar} can be enhanced to an isomorphism which also respects the monoidal structure on the two, as summarised in the following standard result.

\begin{theorem}
\label{thm:monoidaliso}
There is a monoidal isomorphism of monoidal categories between the category of representations of the Lie-Rinehart algebra $(A, \mf{h})$, or left $(A, \mf{h})$-modules for short,  and the category of left $\cU_\ahha(\mf{h})$-modules.
\end{theorem}

\subsection{Left module algebras and the smash product}
\label{ssec:modalgs}

Let $(U,A)$ be a left bialgebroid. Recall that a \emph{(left) $U$-module algebra} is a $\K$-algebra $R$ equipped with a left $U$-module structure $U \otimes R \to R, u \otimes r \mapsto u \cdot r,$ such that, with respect to the $A$-actions from \eqref{eq:restscal},
\[
 (r \ract a)r' = r(a \lact r') 
, \qquad u\cdot (rr') = (u_{(1)}\cdot r)(u_{(2)}\cdot r'), \qquad u \cdot 1_\erre  = \varepsilon(u) \lact 1_\erre 
\]
for all $a \in A$, $r,r'\in R$, $u \in U$. 
In particular, $1_\erre \ract a  = a \lact 1_\erre $ for all $a \in A$. 
Equivalently, $R$ is a monoid in the monoidal category of left $U$-modules. 

Given a left module algebra $R$ over a left bialgebroid $(U,A)$, we may consider the \emph{smash product}  
$R \hash{} U$ 
(see, for instance, \cite[\S3.7.1]{GabiHandbook}). This is the $\K$-vector space $R \otimes_{\ahha} U$ together with unit $1_\erre \otimes_{\ahha} 1_\uhhu$ and multiplication
\[
(r \otimes_{\ahha} u)(r' \otimes_{\ahha} u') = r\left(u_{(1)}\cdot r'\right) \otimes_{\ahha} u_{(2)}u'
\]
for all $r,r' \in R$ and all $u,u'\in U$. It is, in fact, an $R$-ring with the $\K$-algebra morphism
\begin{equation}\label{eq:etabeta}
R \to R \hash{} U, \qquad r \mapsto r \otimes_{\ahha} 1_\uhhu .
\end{equation}

\begin{lemma}\label{lem:smashbialgd}
If $(U,A)$ is a cocommutative left bialgebroid (or left Hopf algebroid) and $R$ is a commutative left $U$-module algebra, then $R \hash{} U$ is a cocommutative left bialgebroid (resp.~left Hopf algebroid) over $R$.
\end{lemma}

\begin{proof}
 Observe that, in this setting, $R$ becomes an $A$-algebra with respect to the symmetric $A$-action $a\lact r = s(a)\cdot r = t(a) \cdot r = r \ract a$ for all $a \in A$, $r \in R$. Thus, we do not need to specify which $A$-action we are considering on $R$ when taking tensor products over $A$. Moreover, we already know that $R \hash{} U$ is an $R$-ring with respect to \eqref{eq:etabeta}. In addition, it is straightforward to check that
\begin{gather*}
\Delta'\coloneqq \pig(R \otimes_\ahha  U \xrightarrow{R \otimes_\ahha  \Delta} R \otimes_\ahha  U \otimes_\ahha  U \simeq \left(R \otimes_\ahha  U\pig) \otimes_{\erre} \left(R \otimes_\ahha  U\right)\right) \\
\text{and} \qquad \varepsilon'\coloneqq\pig(R \otimes_\ahha  U \xrightarrow{R \otimes_\ahha  \varepsilon} R \otimes_\ahha  A \simeq R\pig),
\end{gather*}
where $\left(R \otimes_\ahha  U\right) \otimes_{\erre} \left(R \otimes_\ahha  U\right)$ is given by
\begin{equation}\label{eq:tensR}
\frac{\left(R \otimes_\ahha  U\right) \otimes \left(R \otimes_\ahha  U\right)}{\Big\langle (r \otimes_\ahha  u) \otimes (st\otimes_\ahha  v) - (rs \otimes_\ahha  u) \otimes (t\otimes_\ahha  v) ~\Big\vert~  r,s,t \in R,  u,v  \in U\Big\rangle},
\end{equation}
define a cocommutative left bialgebroid structure on $R \hash{} U$.

Let us check explicitly the claim concerning the Hopf algebroid property. For the sake of clarity, let us write $\prescript{}{\erre}{\left(R \otimes_\ahha  U\right)} \otimes_{\erre} \prescript{}{\erre}{\left(R \otimes_\ahha  U\right)}$ for the tensor product in \eqref{eq:tensR} and $\left(R \otimes_\ahha  U\right)_\erre \otimes_{\erre} \prescript{}{\erre}{\left(R \otimes_\ahha  U\right)}$ for the tensor product
\[
\frac{\left(R \otimes_\ahha  U\right) \otimes \left(R \otimes_\ahha  U\right)}{\Big\langle (r \otimes_\ahha  u) \otimes (st\otimes_\ahha  v) - (r(u_{(1)}\cdot s) \otimes_\ahha  u_{(2)}) \otimes (t\otimes_\ahha  v) ~\Big\vert~ r,s,t \in R, u,v \in U\Big\rangle}.
\]
Recall that to be a left Hopf algebroid for a cocommutative bialgebroid $(U,A)$ means that the canonical map $\beta\colon U_{\bract}^{} \otimes_\ahha  \prescript{}{\lact}{U} \to \prescript{}{\lact}{U} \otimes_\ahha  \prescript{}{\lact}{U}$ from \eqref{nochmehrRegen} 
is invertible. Observe that $\beta$ is left $A$-linear with respect to the regular left $A$-action on both domain and codomain, that is
\[\beta(s(a)u \otimes_\ahha  v) = s(a)u_{(1)} \otimes_\ahha  u_{(2)}v,\]
whence $\beta^{-1}$ is left $A$-linear, too (see \eqref{Sch9}). Therefore we may consider
\[R \otimes_\ahha  \left(\prescript{}{\lact}{U} \otimes_\ahha  \prescript{}{\lact}{U}\right) \xrightarrow{R \otimes_\ahha  \beta^{-1}} R \otimes_\ahha  \left(U_{\bract}^{} \otimes_\ahha  \prescript{}{\lact}{U}\right)\]
and (pre)compose it with the isomorphisms
{\small
\begin{gather*}
\prescript{}{\erre}{\left(R \otimes_\ahha  \prescript{}{\lact}{U}\right)} \otimes_\erre \prescript{}{\erre}{\left(R \otimes_\ahha  \prescript{}{\lact}{U}\right)} \xrightarrow{\varphi} R \otimes_\ahha  \left(\prescript{}{\lact}{U} \otimes_\ahha  \prescript{}{\lact}{U}\right), \quad (r \otimes_\ahha  u) \otimes_\erre (s \otimes_\ahha  v) \mapsto rs \otimes_\ahha  u \otimes_\ahha  v, \\
\left(R \otimes_\ahha  \prescript{}{\lact}{U}\right)_{\erre}^{} \otimes_\erre \prescript{}{\erre}{\left(R \otimes_\ahha  \prescript{}{\lact}{U}\right)} \xrightarrow{\psi} R \otimes_\ahha  \left(U_{\bract}^{} \otimes_\ahha  \prescript{}{\lact}{U}\right), \ (r \otimes_\ahha  u) \otimes_\erre (s \otimes_\ahha  v) \mapsto r(u_{(1)} \cdot s) \otimes_\ahha  u_{(2)} \otimes_\ahha  v,
\end{gather*}}%
in order to construct
\begin{equation*}
\begin{array}{rcl}
\prescript{}{\erre}{\left(R \otimes_\ahha  \prescript{}{\lact}{U}\right)} \otimes_\erre \prescript{}{\erre}{\left(R \otimes_\ahha  \prescript{}{\lact}{U}\right)} 
&
\xrightarrow{\tilde{\beta}^{-1} \coloneqq \psi^{-1} \circ (R \otimes_\ahha  \beta^{-1}) \circ \varphi} 
&
\left(R \otimes_\ahha  \prescript{}{\lact}{U}\right)_{\erre}^{} \otimes_\erre \prescript{}{\erre}{\left(R \otimes_\ahha  \prescript{}{\lact}{U}\right)} 
\\
(r \otimes_\ahha  u) \otimes_\erre (s \otimes_\ahha  v) 
&
\xmapsto{\phantom{\tilde{\beta}^{-1} \coloneqq \psi^{-1} \circ (R \otimes_\ahha  \beta^{-1}) \circ \varphi}}
& 
(rs \otimes_\ahha  u_+) \otimes_\erre (1_\erre  \otimes_\ahha  u_-v).
\end{array}
\end{equation*}
To check that $\tilde\beta^{-1}$ is, in fact, the inverse of
\begin{equation*}
\begin{array}{rcl}
\left(R \otimes_\ahha  \prescript{}{\lact}{U}\right)_{\erre}^{} \otimes_\erre \prescript{}{\erre}{\left(R \otimes_\ahha  \prescript{}{\lact}{U}\right)} 
&
\xrightarrow{\ \tilde{\beta}\ } 
& 
\prescript{}{\erre}{\left(R \otimes_\ahha  \prescript{}{\lact}{U}\right)} \otimes_\erre \prescript{}{\erre}{\left(R \otimes_\ahha  \prescript{}{\lact}{U}\right)} 
\\
(r \otimes_\ahha  u) \otimes_\erre (s \otimes_\ahha  v)  
&
\xmapsto{\phantom{\ \tilde{\beta} \ }}
& 
\left(r \otimes_\ahha  u_{(1)}\right) \otimes_\erre \left((u_{(2)}\cdot s) \otimes_\ahha  u_{(3)}v\right),
\end{array}
\end{equation*}
it suffices to check that $\tilde\beta = \varphi^{-1} \circ \left(R \otimes_\ahha  \beta\right) \circ \psi$. Thanks to the cocommutativity of $\Delta$, a direct computation shows that this is indeed the case.
\end{proof}

When the hypotheses of Lemma \ref{lem:smashbialgd} are satisfied, we may call the resulting left bialgebroid (left Hopf algebroid) structure on $R \hash{} U$ a \emph{smash product left bialgebroid (left Hopf algebroid) structure}.

\begin{proposition}
\label{prop:iso3}
Let $(A,\mf{h},\omega)$ be a Lie-Rinehart algebra over the commutative algebra $A$ such that $\mf{h}$ is projective as a left $A$-module. Assume that $R$ is a commutative $\K$-algebra together with a Lie algebra morphism $\rho\colon\mf{h} \to \Der_{\,\K}(R)$ and a $\K$-algebra morphism $\eta\colon A\to R$ such that $\eta$ is also a morphism of $\mf{h}$-representations  (intertwining $\rho$ and $\omega$) and $\rho(a\cdot X) = \eta(a)\rho(X)$ for all $a \in A$, $X \in \mf{h}$.
Then: 
\begin{enumerate}[label=(\alph*),ref={\textit{(\alph*)}},leftmargin=0.7cm]
    \item\label{item:iso3a} The algebra $R$ is a left $\cU_\ahha(\mf{h})$-module algebra, whence $R \hash{} \cU_\ahha(\mf{h})$ admits a structure of smash product left Hopf algebroid over $R$.
    \item\label{item:iso3b} The $R$-module $R \otimes_{\ahha} \mf{h}$ admits a structure of Lie-Rinehart algebra over $R$ with bracket
			\[
			\big[r \otimes_{\ahha} X, r' \otimes_{\ahha} Y\big] = rr' \otimes_{\ahha} [X,Y] - r'\rho(Y)(r) \otimes_{\ahha} X + r\rho(X)(r') \otimes_{\ahha} Y
			\]
			and anchor
			\begin{equation}\label{eq:anchorp}
			\omega_{R} \colon R \otimes_{\ahha} \mf{h} \to \Der_{\,\K}(R), \qquad r \otimes_{\ahha} X \mapsto r\rho(X)
			\end{equation}
			for all $r,r' \in R$, $X,Y \in \mf{h}$. This Lie-Rinehart algebra will be denoted $(R,R\rtimes_\ahha \mathfrak{h})$.
    \item\label{item:iso3c} One has $\cU_\erre (R \rtimes_{\ahha} \mf{h}) \simeq R \!\hash{}\!\, \cU_\ahha(\mf{h})$ as cocommutative left Hopf algebroids over $R$ via the unique morphism induced by
\[\iota_\erre (r) \mapsto r \otimes_{\ahha} 1 \qquad \text{and} \qquad \iota_{R \otimes_{\ahha}\mf{h}}\left(r \otimes_{\ahha} X\right) \mapsto r \otimes_{\ahha} \iota_{\mf{h}}(X).\]
\end{enumerate}
\end{proposition}

\begin{proof}
\ref{item:iso3a}: first of all, observe that asking $\eta\colon A \to R$ to be a morphism of $\mf{h}$-representations means that $\rho(X) \circ \eta = \eta \circ \omega(X)$ for all $X \in \mf{h}$. Since we have that
\begin{align*}
\rho(a\cdot X)(r) & = \eta(a)\rho(X)(r) = a\cdot \rho(X)(r), \\
\rho(X)(a\cdot r) & = \rho(X)\big(\eta(a)r\big) = \rho(X)\big(\eta(a)\big)r + \eta(a)\rho(X)(r) \\
 & = \eta\big(\omega(X)(a)\big)r+a\cdot \rho(X)(r) = \omega(X)(a)\cdot r + a\cdot \rho(X)(r)
\end{align*}
for all $a \in A$, $X \in \mf{h}$, $r \in R$, the left $A$-module $R$ is a representation of the Lie-Rinehart algebra $(A,\mf{h},\omega)$. In particular, since it is an $A$-algebra on which $\mf{h}$ acts by $\K$-linear derivations, $R$ is a monoid in the monoidal category of $(A,\mf{h},\omega)$-representations. In view of Theorem \ref{thm:monoidaliso}, this makes of $R$ an $\cU_\ahha(\mf{h})$-module algebra with action uniquely determined by
\begin{equation}\label{eq:actionAp}
\iota_\ahha (a)\cdot r = \eta(a)r \qquad \text{and} \qquad \iota_{\mf{h}}(X) \cdot r = \rho(X)\left(r\right)
\end{equation}
for all $a \in A$, $X \in \mf{h}$ and $r \in R$.
So we may endow the tensor product $R \otimes_{\ahha} \cU_\ahha(\mf{h})$ with the smash product structure $R\hash{}\cU_\ahha(\mf{h})$.

\ref{item:iso3b}: follows from \cite[Prop.~1.16]{Huebschmann}. 

\ref{item:iso3c}: notice that both canonical morphisms $\iota_\ahha  \colon A \to \cU_\ahha(\mf{h})$ and $\iota_{\mf{h}} \colon \mf{h} \to \cU_\ahha(\mf{h})$ are left $A$-linear morphisms, whence we may consider the extensions of scalars
\[
\begin{gathered}
\left(R \otimes_{\ahha}\iota_\ahha \right) \colon R \to R \otimes_{\ahha}\cU_\ahha(\mf{h}), \quad r \mapsto r \otimes_{\ahha} 1, \qquad \text{and} \\
\left(R \otimes_{\ahha}\iota_{\mf{h}}\right) \colon \RR  \to R \otimes_{\ahha} \cU_\ahha(\mf{h}), \quad r \otimes_{\ahha} X \mapsto r \otimes_{\ahha} \iota_{\mf{h}}(X).
\end{gathered}
\]
It is straightforward to check that they satisfy the relations \eqref{dumdidum}, and that $R \ot_\ahha \iota_{\mf{h}}$ takes values in the space of primitive elements
\[
\mathcal{P}_\erre\left(R \hash{} \cU_\ahha(\mf{h})\right) = \pig\{x \in R \hash{} \cU_\ahha(\mf{h}) ~\pig\vert~  x_{(1)} \ot_\erre x_{(2)} = x \ot_\erre 1 + 1 \ot_\erre x\pig\}
\]
of the cocommutative left bialgebroid $R \hash{} \cU_\ahha(\mf{h})$. Therefore, by the universal property of $\cU_\erre \left(\RR \right)$ as a bialgebroid (see \cite[Thm.~3.1]{MoerdijkLie}), there exists a unique morphism of left bialgebroids
\[J'\colon\cU_\erre \left(\RR \right) \to R \hash{} \cU_\ahha(\mf{h})\]
such that $J'\circ \iota_{\RR } = R \otimes_{\ahha}\iota_{\mf{h}}$. 
To prove that $J'$ is an isomorphism, let us construct an inverse.

In view of \ref{item:iso3b}, the morphism $\jmath\colon\mf{h} \to R \otimes_{\ahha} \mf{h}, X \mapsto 1_\erre  \otimes_{\ahha} X,$ is a morphism of Lie algebras and of left $A$-modules. Consider the morphisms
\[
\xymatrix @C=30pt @R=15pt{
A \ar[r]^-{\iota_\ahha } \ar[d]_-{\eta} & \cU_\ahha(\mf{h}) & \mf{h} \ar[l]_-{\iota_{\mf{h}}} \ar[d]^-{\jmath} \\
 R \ar[r]_-{\iota_\erre }& \cU_\erre \left(\RR \right) &  \RR . \ar[l]^-{\iota_{\RR }}
}
\]
The composition $\iota_\erre  \circ \eta$ is a $\K$-algebra morphism (as composition of $\K$-algebra morphisms). The composition $\iota_{\RR } \circ \jmath$ is a Lie algebra morphism (as composition of Lie algebra morphisms). Furthermore, they satisfy the relations \eqref{dumdidum}.
As a consequence, by the universal property of $\cU_\ahha(\mf{h})$, there exists a unique morphism of $\K$-algebras $J\colon\cU_\ahha(\mf{h}) \to \cU_\erre \left(\RR \right)$ such that
\begin{equation}\label{eq:Jiotas}
J \circ \iota_\ahha  = \iota_\erre  \circ \eta \qquad \text{and} \qquad J \circ \iota_{\mf{h}} = \iota_{\RR } \circ \jmath.
\end{equation}
In particular, $J$ is left $A$-linear, where the $A$-module structure on $\cU_\erre (\RR )$ is given by $\iota_\erre  \circ \eta$, and
\begin{equation}
\begin{aligned}
\pig[J\big(\iota_{\mf{h}}(X)\big),\iota_\erre \left(r\right)\pig] & \stackrel{\scriptscriptstyle\eqref{eq:Jiotas}}{=} \pig[\iota_{\scriptscriptstyle \RR }\big(\jmath(X)\big),\iota_\erre \left(r\right)\pig]  \stackrel{\scriptscriptstyle\eqref{eq:compLRalg}}{=} \iota_\erre \Big(\omega_{\erre}\big(\jmath(X)\big)(r)\Big) \\
& \stackrel{\scriptscriptstyle\eqref{eq:anchorp}}{=} \iota_\erre \big(\rho\left(X\right)\left(r\right)\big) \stackrel{\scriptscriptstyle\eqref{eq:actionAp}}{=} \iota_\erre \big(\iota_{\mf{h}}(X) \cdot r\big)
\end{aligned}\label{eq:chenoia}
\end{equation}
for all $X \in \mf{h}$ and $r \in R$. 
By extension of scalars, there exists a unique morphism of left $R$-modules $J_\erre  \colon R \otimes_{\ahha} \cU_\ahha(\mf{h}) \to \cU_\erre \left(\RR \right)$ such that $J_\erre (1 \otimes_{\ahha} x) = J(x)$ for all $x \in \cU_\ahha(\mf{h})$ and a straightforward check using \eqref{eq:chenoia} and induction on a set of PBW generators of $\cU_\ahha(\mf{h})$ of the form $u = X_1\cdots X_n$ for $X_1,\ldots,X_n \in \mf{h}$ shows that
\[J_\erre  \colon R \hash{} \cU_\ahha(\mf{h}) \to \cU_\erre \left(\RR \right), \qquad r \hash{} u \mapsto \iota_\erre (r)J(u),\]
is a morphism of $R$-rings.

By a direct computation, one may check that $J_\erre  \circ J'\circ \iota_\erre  = \iota_\erre $ and $J_\erre  \circ J'\circ \iota_{\scriptscriptstyle \RR } = \iota_{\scriptscriptstyle \RR }$, whence $J_\erre  \circ J'$ is the identity, by the the universal property of $\cU_\erre \left(\RR \right)$ as an algebra (that is, the one from \S\ref{regenradar}).
On the other hand, consider the composition $J '\circ J_\erre  \colon R \otimes_{\ahha}\cU_\ahha(\mf{h}) \to R \otimes_{\ahha}\cU_\ahha(\mf{h})$. Being it left $R$-linear, it is uniquely determined by $J'\left(J_\erre \left(1 \otimes_{\ahha} u\right)\right) =  J'\left(J(u)\right)$ for all $u \in \cU_\ahha(\mf{h})$.
Now, $J'\circ J$ is a $\K$-algebra morphism (as composition of $\K$-algebra morphisms) which satisfies
\begin{gather*}
\left(J'\circ J\right)\big(\iota_\ahha (a)\big) \stackrel{\scriptscriptstyle\eqref{eq:Jiotas}}{=} J'\Big(\iota_\erre \big(\eta(a)\big)\Big) = (R \otimes_{\ahha} \iota_\ahha )\big(\eta(a) \otimes_{\ahha} 1\big) = \eta(a) \otimes_{\ahha} 1_{\cU_\ahha(\mf{h})} = 1_\erre  \otimes_{\ahha} \iota_\ahha (a) \\
\text{and} \qquad \left(J'\circ J\right)\big(\iota_{\mf{h}}(X)\big) \stackrel{\scriptscriptstyle\eqref{eq:Jiotas}}{=} J'\Big(\iota_{R\otimes_{\ahha}\mf{h}}\big(\jmath(X)\big)\Big) = (R \otimes_{\ahha} \iota_{\mf{h}})\big(1_\erre  \otimes_{\ahha} \iota_{\mf{h}}(X)\big) = 1_\erre  \otimes_{\ahha} \iota_{\mf{h}}(X)
\end{gather*}
for all $a \in A$, $X \in \mf{h}$, we have that $J'\circ J = 1_\erre  \otimes_{\ahha} \id_{\scriptscriptstyle \cU_\ahha(\mf{h})}$ and so $J'\circ J_\erre $ is the identity as well.
\end{proof}

\begin{remark}
Given a commutative $\K$-algebra $R$, we may always consider the Lie-Rinehart algebra $(R,\Der_{\,\K}(R),\id)$. In this setting, giving morphisms $\rho\colon \mf{h} \to \Der_{\,\K}(R)$ and $\eta\colon A \to R$ as in the statement of Proposition \ref{prop:iso3} is equivalent to giving a morphism of Lie-Rinehart algebras $(\eta,\rho) \colon (A,\mf{h},\omega) \to (R,\Der_{\,\K}(R),\id)$ in the sense of \cite[p.~61]{Huebschmann}, which in turn is equivalent to giving an $A$-algebra structure $\eta\colon A \to R$ and a morphism of Lie-Rinehart algebras over $A$
\[
\varrho\colon \mf{h} \to \mathcal{A}_\ahha^\erre\big(\Der_{\,\K}(R),\id\big) \coloneqq \Big\{(D,\delta) \in \Der_{\,\K}(R) \times \Der_{\,\K}(A) ~\Big\vert~ D(\eta(a)) = \eta(\delta(a)) \text{ for all }a \in A\Big\}
\]
as in \cite[Def.~4.2]{Saracco2}.
\end{remark}

\subsection{Right comodule algebras}

Recall that a \emph{right comodule} over a left bialgebroid $(U,A)$ is a pair $(M, \delta_\emme )$, where $M$ is a right $A$-module and $\delta_\emme  \colon M \to M \times_\ahha U \subset M \ot_\ahha U$ is a right $A$-module map, satisfying the usual coassociativity and counitality conditions
\[
(\delta_\emme  \ot_\ahha U) \circ \delta_\emme  = (M \ot_\ahha \Delta)\circ \delta_\emme  \qquad \text{resp.} \qquad (M \ot_\ahha \varepsilon) \circ \delta_\emme  = \id_M.
\]
As usual, we will write $\delta_\emme (m) = m_{[0]} \ot_\ahha m_{[1]}$ for all $m \in M$. Any right comodule $(M,\delta_\emme )$ can be equipped with a left $A$-action
\[
am \coloneqq m_{[0]}\varepsilon(a \blact m_{[1]})=m_{[0]}\varepsilon(m_{[1]} \bract a),
\]
for all $a \in A$, $m \in M$. The category of right comodules over $(U,A)$ is again a monoidal category with monoidal product $\ot_\ahha$. The right $A$-action and $U$-coaction on the tensor product
of two $U$-comodules $M$ and $N$ are, for all $a \in A$ and $m \ot_\ahha n \in M \ot_\ahha N$,
\[
(m \ot_\ahha n) a \coloneqq m \ot_\ahha na \qquad \text{and} \qquad \delta_{M \ot_\ahha N}( m \ot_\ahha n) \coloneqq (m_{[0]} \ot_\ahha n_{[0]}) \ot_\ahha n_{[1]}m_{[1]}.
\]
{\em Right $U$-comodule algebras} are, by definition, monoids in the monoidal category $\comodu$ of right $U$-comodules. That is to say, they are right $U$-comodules $(R,\delta_\erre )$ together with a right $U$-colinear, associative and unital multiplication $\mu_\erre \colon R \ot_\ahha R \to R$ with unit $\eta_\erre \colon A \to R$ or, equivalently, $A$-rings $(R,\eta_\erre )$ together with a coassociative and counital coaction $\delta_\erre \colon R \to R \ot_\ahha U$ such that
\begin{gather*}
  (ar)r' = a(rr'), \qquad (ra)r' = r(ar'), \qquad r(r'a) = (rr')a,
  \\
  \eta_\erre(a) = 1_\erre a = a 1_\erre, \qquad \gd_\erre(1_\erre a) = \gd_\erre(a1_\erre) = 1_\erre \otimes_\ahha  t(a),
  \\
 (rr')_{[0]} \ot_\ahha (rr')_{[1]} = r_{[0]}r'_{[0]} \ot_\ahha r_{[1]}r'_{[1]}, \qquad
   1_{[0]} \ot_\ahha 1_{[1]} = 1_\erre  \ot_\ahha 1_\uhhu. 
\end{gather*}

\begin{example}
Let $U \xrightarrow{\pi}V$ be a morphism of left bialgebroids over $A$. Then $U$ with
\[\delta_\uhhu \coloneqq \left(U \xrightarrow{\gD } U \ot_\ahha U \xrightarrow{U \ot_\ahha \pi} U \ot_\ahha V\right)\]
is a right $V$-comodule algebra.
\end{example}


\section{A generalised Blattner-Cohen-Montgomery theorem for left Hopf algebroids}\label{sec:BCM}

In this section, we extend \cite[Thm.~4.14]{BlattnerCohenMontgomery} to the realm of left Hopf algebroids. In order to do this, we need to recall a couple of notions on measurings, weak actions, and $2$-cocycles as introduced (in this context) in \cite{GabiTomek}.


\subsection{Weak actions, cocycles, and crossed products}

This subsection mainly contains (by now well-established) material from \cite{GabiTomek}, which will be stated quite in detail as explicitly needed in the main theorem in this section.

\begin{definition}[{\cite[Def.~4.1]{GabiTomek}}]
\label{GT1}
A left bialgebroid $(U,A)$ is said to {\em measure} an $A$-ring $\iota\colon A \to B$ if there exists a map $U \otimes B \to B$, $u \otimes b \mapsto u \rightslice b$ which is $A$-bilinear with respect to the $A$-bimodule structure 
$
\due {(U \otimes B)} \lact \ract \coloneqq
\due U \lact \ract \otimes B$, that is,
$$
\big(s(a)\,t(a')\,u\big)\rightslice b=i(a)(u\rightslice b)i(a'),
$$
for all $a,a'\in A$, $u\in U$ and $b\in B$,
and which is such that
\begin{equation}
\label{vendeeglobe}
u\rightslice 1_\behhe =\iota(\gve(u)) \qquad \text{and} \qquad
u\rightslice (b b')=(u_{(1)} \rightslice b)(u_{(2)} \rightslice b')
\end{equation}
for all $u\in U$ and $b,b'\in B$.
\end{definition}

A left $U$-module algebra $B$ as in \S\ref{ssec:modalgs} is measured simply by the left $U$-action on $B$.

\begin{definition}[{\cite[Def.~4.2]{GabiTomek}}]
\label{GT2}
Let $(U,A)$ be a left bialgebroid and
$B$ a $U$-measured $A$-ring.
A {\em $B$-valued $2$-cocycle} on $U$ is a map $\sigma\colon \due U \blact {} \otimes_\Aopp U_\ract \to B$ which is $A$-bilinear with respect to the $A$-bimodule structure 
$\due {(U \otimes_\Aopp U)} \lact \ract  \coloneqq
\due U \lact \ract \otimes_\Aopp U$ 
and 
for all $u,v,w\in U$, $a\in A$
subject to
\begin{enumerate}
\item
\label{item:wa1}
$(u_{(1)} \rightslice \iota(a))\sigma(u_{(2)},v)=\sigma(u, a \lact v)$,
\item
\label{item:wa2} 
$\sigma(1,u)=\iota(\gve(u))=\sigma(u,1)$,
\item
\label{item:wa3} 
$\big(u_{(1)} \rightslice  \sigma(v_{(1)},w_{(1)})\big)
\gs(u_{(2)}, v_{(2)}w_{(2)})\big)
=
\gs(u_{(1)}, v_{(1)})\gs(u_{(2)}v_{(2)}, w).
$
\end{enumerate}
A $U$-measured $A$-ring $B$ is called a {\em $\sigma$-twisted left
$U$-module} if a $B$-valued $2$-cocycle $\sigma$ is given and
\begin{enumerate}
\setcounter{enumi}{3}
\item 
\label{item:wa4} 
$1_\uhhu \rightslice b=b$,
\item 
\label{item:wa5}
$\big(u_{(1)} \rightslice  (v_{(1)} \rightslice b)\big)
\gs(u_{(2)}, v_{(2)})
=
\gs(u_{(1)}, v_{(1)})(u_{(2)}v_{(2)} \rightslice b)
$
\end{enumerate}
for all $u,v\in U$, $b\in B$.
\end{definition}

A measuring as in Definition \ref{GT1} that additionally fulfils \eqref{item:wa4} is called a {\em weak (left) action}.
 Similarly to the bialgebra case,
 the map $\sigma(u,v) \coloneqq \iota(\gve(uv))$ is a (trivial) cocycle, provided that the measuring restricts to the $U$-action on $A$, that is to say, 
 \[u\rightslice \iota(a)= \iota\big(u . a\big) \stackrel{\scriptscriptstyle\eqref{eq:UAact}}{=} \iota\big(\gve(u \bract a)\big)\]
 for $u\in U$ and $a\in A$.
A twisted left $U$-module  corresponding to this trivial cocycle
 $\sigma$
is simply a left $U$-module algebra.

\begin{proposition}[{\cite[Prop.~4.3]{GabiTomek}}]
\label{GT3}
Let $(U,A)$ be a left bialgebroid and
$\iota\colon A\to B$ be an $U$-measured $A$-ring. Then the $\K$-module $B_\bract \otimes_\ahha  \due U \lact {}$ is an associative algebra with unit $1_\behhe \otimes_\ahha  1_\uhhu$ and multiplication
\[
(b\otimes_\ahha  u)\otimes_\K (b'\otimes_\ahha  v)\mapsto
b(u_{(1)} \rightslice b')\sigma(u_{(2)}, v_{(1)})\otimes_\ahha  u_{(3)} v_{(2)}
\]
if and only if $\sigma$ is a cocycle and
$B$ is a $\sigma$-twisted $U$-module.
\end{proposition}

\begin{definition}[Crossed product of an $A$-ring with a left bialgebroid over $A$]
\label{GT4}
The associative algebra resulting from the preceding proposition is called a {\em $\gs$-twisted crossed product} of $B$
with $U$ and is denoted by $B \hash{\gs} U$. 
\end{definition}

For better distinction, we will sometimes write $b \hash{} u$ instead of $b \otimes_\ahha  u$ for an element in $B \hash{\gs} U$.


\subsection{Crossed products from coring split exact sequences}\label{sec:SplitThemAll}

This subsection is an extension of the corresponding Hopf algebraic results in \cite[\S4]{BlattnerCohenMontgomery} to the realm of left Hopf algebroids. Such an extension allows to apply the corresponding results to the universal enveloping algebras of Lie-Rinehart algebras.

\begin{definition}[Left Hopf kernel]
\label{boing}
Let $(U,A)$ and $(V,A)$ be two left Hopf algebroids over the same base algebra $A$ and 
let $U \stackrel{\pi}{\to} V \to 0$ be an exact sequence of left Hopf algebroids. The set
\begin{equation}
\label{mandarino}
B  \coloneqq \big\{u \in U \mid (U \otimes_\ahha  \pi)\left(\gD_\uhhu u\right) = u \otimes_\ahha  1_\vauu \in U_\ract \tak{\ahha} \due V \lact {}\big\} 
\end{equation}
is called the {\em left Hopf kernel} of $\pi$.
\end{definition}

\noindent This definition is a straightforward extension to left Hopf algebroids of the corresponding definition for Hopf algebras in \cite[Def.\  4.12]{BlattnerCohenMontgomery}.

\begin{remark}\label{rem:equalizer}
Categorically, the left Hopf kernel of $\pi\colon U \to V$ can be realised as the equaliser
\begin{equation}\label{eq:ee}
\xymatrix@C=25pt{ B \ar[r]^-{e} & \due U \lact \ract \ar@<+0.7ex>[rr]^-{(U \otimes_{\ahha} \pi)\Delta_\uhhu } \ar@<-0.7ex>[rr]_-{U \otimes_{\ahha} s_\vauu} & & \due U {} \ract \tak{\ahha} \due V \lact {}}
\end{equation}
in the category of left $A$-modules. Thus, $B$ is a left $A$-submodule of $U$ with action induced by $ \due U \lact {}$, but not a right $A$-submodule, in general.
\end{remark}

The following is our analogue of \cite[Lem.~4.13]{BlattnerCohenMontgomery}.

\begin{lemma}
\label{wasserkessel}
Let $(U,A)$ be a left Hopf algebroid as in Definition \ref{boing}.
\begin{enumerate}[leftmargin=0.7cm,label=(\roman*),ref={\it (\roman*)}]
\item\label{item:wki} The left Hopf kernel $B \subseteq U$ is an $A$-subring of $U$ with unit map $s_\uhhu\colon A \to B$. 
The elements of $B$ in $U$ commute with the elements of $t_\uhhu(A)$, that is, 
\begin{equation}
\label{ohyes}
b \ract a = t_\uhhu(a)\, b = b\, t_\uhhu(a) = a \blact b
\end{equation}
for all $b \in B$ and $a \in A$. 

\item\label{item:wkii} Equation  \eqref{ohyes} implies that
\begin{equation}
\label{adjointaction}
\rightslice \colon U \otimes_\K B \to B, \quad u \otimes b \mapsto u_+ b u_-
\end{equation}
gives a well-defined left $U$-action on $B$, which we term {\em left adjoint action}, and which descends to an $\Ae$-balanced action $\due U \blact \bract \otimes_\Ae \due B \lact \bract \to B$. Moreover, it turns $B$ into a $U$-module algebra.
\end{enumerate}
\end{lemma}

\begin{proof} 
\ref{item:wki}:
That $s_\uhhu\colon A \to B$ defines on $B$ the structure of an $A$-ring with well-defined multiplication $m_\behhe\colon B_\bract \otimes_\ahha  \due B \lact {} \to B$ directly follows from the $A$-linearity of the coproduct and the fact that it is an algebra morphism as in \eqref{beethoven2}.
Namely, the following diagram commutes serially
\[
\xymatrix@C=27pt{ 
B \ot B \ar@{.>}[d]_-{m_B'} \ar[r]^-{e\, \ot\, e} & U \ot U \ar[d]_-{m_\uhhu} \ar@<+0.7ex>[rrrr]^-{(U \otimes_{\ahha} \pi)\Delta_\uhhu  \, \ot \, (U \otimes_{\ahha} \pi)\Delta_\uhhu } \ar@<-0.7ex>[rrrr]_-{(U \otimes_{\ahha} s_\vauu) \, \otimes \, (U \otimes_{\ahha} s_\vauu)} & & & & (U \tak{\ahha} V) \ot (U \tak{\ahha} V) \ar[d]^-{m_{{}_{U \tak{\ahha} V}}} \\ 
B \ar[r]^-{e} & U \ar@<+0.7ex>[rrrr]^-{(U \otimes_{\ahha} \pi)\Delta_\uhhu } \ar@<-0.7ex>[rrrr]_-{U \otimes_{\ahha} s_\vauu} & & & & U \tak{\ahha} V}
\]
where $e$ is the map introduced in \eqref{eq:ee}. Hence there exists a unique $m_B'$ which factors through $m_\behhe\colon B_\bract \otimes_\ahha  \due B \lact {} \to B$.
Moreover, concerning the commutativity between $B$ and $t_\uhhu(A)$, by definition of $B$ we have that
\[
b_{(1)} \otimes_\ahha  \pi\left(b_{(2)}\right) = b \otimes_{\ahha} 1_\vauu\]
in $U \otimes_\ahha  V$, for all $b \in B$, and hence also
\[
b_{(1)} \otimes_\ahha  \pi\left(b_{(2)}t_\uhhu(a)\right) = b_{(1)} \otimes_\ahha  \pi\left(b_{(2)}\right)t_\vauu(a) = b \otimes_{\ahha} t_\vauu(a),
\]
by multiplying on the right with $1_\uhhu \otimes_\ahha  t_\vauu(a)$. By applying $\id_\uhhu \otimes_{\ahha} \gve_\vauu$ to both sides of the latter equality and identifying $U \otimes_\ahha  A \simeq U$ by means of $u \otimes_\ahha  a \mapsto t_\uhhu(a) u$, we find that
\[
(t_\uhhu\circ\gve_\vauu\circ\pi)\big(b_{(2)}t_\uhhu(a)\big) b_{(1)} =
(t_\uhhu\circ\gve_\vauu\circ t_\vauu)(a) b = t_\uhhu(a) b 
\]
in $U$, where we used $\gve_\vauu\circ t_\vauu=\id_\ahha$ in the second equality.
Now, counitality \eqref{beethoven3} along with the multiplicativity and $A$-bilinearity of $\Delta_\uhhu $ \eqref{beethoven2} and $\gve_\uhhu = \gve_\vauu \circ \pi$ since
$\pi$ is a morphism of left Hopf algebroids, allow to conclude that
\[
b\,t\uhhu(a) = (t_\uhhu\circ\gve_\uhhu)\left( b_{(2)}t_\uhhu(a)\right)b_{(1)} = (t_\uhhu\circ\gve_\vauu\circ\pi)\left(b_{(2)}t_\uhhu(a)\right)b_{(1)}
= t_\uhhu(a)\,b,
\]
which proves \eqref{ohyes}. By this, the expression \eqref{adjointaction} is well-defined. 

\ref{item:wkii}: That the expression \eqref{adjointaction} indeed gives a left $U$-action now simply follows from \eqref{Sch6}, once ensured that its image lands in $B$ again, which one may explicitly check for all $u \in U$ and $b \in B$:
\begin{equation*}
\begin{split}
(U \otimes_\ahha  \pi)\gD_\uhhu(u_+bu_-) 
& \stackrel{\phantom{(**)}}{=} u_{+(1)} b_{(1)} u_{-(1)} \otimes_\ahha  \pi\big(u_{+(2)} b_{(2)} u_{-(2)} \big)
\\
& \stackrel{\phantom{(**)}}{=} \left(u_{+(1)} \otimes_\ahha  \pi\left(u_{+(2)}\right)\right)\cdot\left(b_{(1)} \otimes_\ahha  \pi\left(b_{+(2)}\right)\right)\cdot\left(u_{-(1)} \otimes_\ahha  \pi\left(u_{-(2)}\right)\right) \\
& \stackrel{\phantom{(**)}}{=} \left(u_{+(1)} \otimes_\ahha  \pi\left(u_{+(2)}\right)\right)\cdot\left(b \otimes_\ahha  1_\vauu\right)\cdot\left(u_{-(1)} \otimes_\ahha  \pi\left(u_{-(2)}\right)\right) \\
& \stackrel{\phantom{(**)}}{=} u_{+(1)}bu_{-(1)} \otimes_\ahha  \pi\left(u_{+(2)}\right) \pi\left(u_{-(2)}\right) \\
& \stackrel{\scriptscriptstyle (*)}{\mathmakebox[\widthof{$\stackrel{(**)}{=}$}]{=}} u_{++(1)} b u_{-} \otimes_\ahha  \pi(u_{++(2)} u_{+-})
\\
& \stackrel{\scriptscriptstyle (**)}{=} u_{+} b u_{-} \otimes_\ahha  1_\vauu,
\end{split}
\end{equation*}
where $(*)$ follows from \eqref{Sch5} and the fact that the assignment 
\[
\due U \blact {} \otimes_\Aopp (\due U {} \ract \otimes_\ahha  \due U \lact \ract) \to \due U {} \ract \otimes_\ahha  \due V \lact {}, \qquad u \otimes_\Aopp v \otimes_\ahha  w \mapsto u_{(1)}bv \otimes_\ahha  \pi\left(u_{(2)}w\right)
\]
is well-defined, while $(**)$ follows from \eqref{Sch2} and the fact that the assignment
\[
(\due U \blact \ract \otimes_\ahha  \due U \lact {}) \otimes_\Aopp \due U {} \ract \to \due U {} \ract \otimes_\ahha  \due V \lact {}, \qquad u \otimes_\ahha  v \otimes_\Aopp w \mapsto ubw \otimes_\ahha  \pi(v)
\]
is well-defined by means of \eqref{ohyes}.
Hence, $u_+bu_- \in B$ if $b \in B$, as claimed.
The second statement in \ref{item:wkii} follows from \eqref{Sch9}, and the very last claim in \ref{item:wkii} is proven as follows. We already know that \eqref{adjointaction} defines an action. Let us verify that it also fulfils the conditions \eqref{vendeeglobe}.
As for the first one, for $u \in U$ one has 
$$
u \rightslice 1_\behhe \stackrel{\scriptscriptstyle\eqref{adjointaction}}{=} u_+  u_- \stackrel{\scriptscriptstyle\eqref{Sch7}}{=} s_\uhhu\big(\gve_\uhhu(u)\big) ,
$$
considering that $1_\behhe = 1_\uhhu$, whereas the second condition in \eqref{vendeeglobe} is proven by
\begin{equation*}
\begin{split}
(u_{(1)} \rightslice b)(u_{(2)} \rightslice b') &  \stackrel{\scriptscriptstyle\eqref{adjointaction}}{=} u_{(1)+} b u_{(1)-}  u_{(2)+}  b' u_{(2)-} \stackrel{\scriptscriptstyle (*)}{=} u_{+(1)+} b u_{+(1)-}  u_{+(2)}  b' u_{-} \\
& \stackrel{\scriptscriptstyle (**)}{=} u_{+} b b' u_{-}  \stackrel{\scriptscriptstyle\eqref{adjointaction}}{=}  u \rightslice (bb')
\end{split}
\end{equation*}
for all $u \in U$ and $b, b' \in B$, where $(*)$ follows from \eqref{Sch4} and the fact that the assignment
\[
(\due U {} \ract \otimes_\ahha  \due U {\substack{\lact}{\blact}} {}) \otimes_\Aopp \due U {} \ract \to U, \qquad u \otimes_\ahha  v \otimes_\Aopp w \mapsto u_{+}bu_{-}vb'w
\]
is well-defined by \eqref{ohyes} again, and $(**)$ follows from \eqref{Sch3} in the form
\[u_{+(1)+} \otimes_\Aopp u_{+(1)-} u_{+(2)} \otimes_\Aopp u_{-} = u_{+} \otimes_\Aopp 1 \otimes_\Aopp u_{-} \ \in \ \due U \blact {} \otimes_\Aopp \due U \blact \ract \otimes_\Aopp \due U {} \ract\]
and from the fact that the assignment
\[\due U \blact {} \otimes_\Aopp \due U \blact \ract \otimes_\Aopp \due U {} \ract \to U, \qquad u \otimes_\Aopp v \otimes_\Aopp w \mapsto ubvb'w,\]
is well-defined.
Hence, \eqref{adjointaction} defines the structure of a left $U$-module algebra on $B$.
\end{proof}

\begin{remark}
Observe that for an arbitrary $U$-bimodule $M$, the above definition \eqref{adjointaction} of an adjoint action is not well-defined as $u_+ \otimes_\Aopp u_- \in \due U \blact {} \otimes_\Aopp U_\ract$, and it is precisely \eqref{ohyes} that makes things work for the left Hopf kernel.
\end{remark}

The following is then an adaption of \cite[Thm.~4.14]{BlattnerCohenMontgomery} to our needs. In view of Proposition \ref{GT3}, the idea is now pretty simple: given a surjective morphism of left Hopf algebroids $\pi\colon U\to V$ over $A$, we exhibit a cocycle with values in the left Hopf kernel $B$ and we show that $B$ is a twisted $\gs$-module with respect to the left adjoint action \eqref{adjointaction}. This allows to form a crossed product $B \hash{\gs} V$, which is then shown to be isomorphic, by an explicit map, to the left Hopf algebroid $U$.  

Recall that a morphism between the underlying corings of two left Hopf algebroids over the same base is a map $\gamma$ that 
fulfils, in particular,
\begin{equation}
\label{vanilla}
\gamma(a \lact v) = a \lact \gamma(v), \qquad \gamma(v \ract a) = \gamma(v) \ract a, \qquad \forall v \in V, \ a \in A,
\end{equation}
and in general
commutes with almost all of the structure maps as in \eqref{krach1} except for the first that speaks with the multiplication. 
In other words, it is in general neither a morphism of rings nor of the Hopf structure in the sense of \eqref{krach2}.

\begin{theorem}
\label{thm:sigmatwisted}
Let $(U,A)$ and $(V,A)$ be two left Hopf algebroids over the same base algebra $A$ and 
let $U \stackrel{\pi}{\to} V \to 0$ be an exact sequence of left Hopf algebroids which splits as an $A$-coring sequence, that is, there exists a morphism $\gamma\colon V \to U$ of $A$-corings such that $\pi \circ \gamma = \id_\vauu$. Assume furthermore that $U_{\ract}$ is projective as a right $A$-module and that
\begin{equation}
\label{vanilly}
\gamma(1_\vauu) = 1_\uhhu, \qquad \gamma(a \blact v) = a \blact \gamma(v), \qquad \gamma(v \bract a) = \gamma(v) \bract a, \qquad \forall v \in V, \ a \in A. 
\end{equation}
Let $B$ be the left Hopf kernel of $\pi$ as in Definition \ref{boing}.
Then
\begin{equation}
\label{chesorpresa}
\gs\colon V \otimes_\Aopp V \to B, \quad v \otimes_\Aopp w \mapsto t_\uhhu \gve_\uhhu\big(\gamma(v_{(2)} w_{(2)} )_+\big) \gamma(v_{(1)})\gamma(w_{(1)})\gamma(v_{(2)} w_{(2)} )_-
\end{equation}
is a cocycle in the sense of Definition \ref{GT2} such that $B$ becomes a $\gs$-twisted left $V$-module with respect to the weak left $V$-action 
\begin{equation}
\label{weakadjoint}
v \rightslice b  \coloneqq \gamma(v)_+ b\, \gamma(v)_-, \qquad v \in V, \ b \in B,
\end{equation}
on $B$
induced by \eqref{adjointaction}. In particular, there is an isomorphism
\begin{equation}\label{eq:Phi}
\Phi\colon B \hash{\gs} V \to U, \quad b \hash{} v \mapsto b\,\gamma(v)
\end{equation}
of $\Ae$-rings and of right $V$-comodule algebras, with inverse given by
\begin{equation}\label{eq:Psi}
\Psi\colon U \to B \hash{\gs} V, \quad u \mapsto  
(t_\uhhu\circ \gve_\uhhu)\pig(\gamma\big(\pi(u_{(2)})\big)_+\pig)
u_{(1)} \gamma\big(\pi(u_{(2)})\big)_- \hash{} \pi(u_{(3)}),
\end{equation}
where $B \hash{\gs} V$ is an $\Ae$-ring via
\[
\Ae \to B \hash{\gs} V, \quad a \ot a' \mapsto s_\uhhu(a)\hash{} t_\vauu(a'),
\]
and a right $V$-comodule algebra via
\[
B \hash{\gs} V \to \left(B \hash{\gs} V\right) \otimes_\ahha  V, \quad b\hash{}v \mapsto \left(b\hash{}v_{(1)} \right) \otimes_\ahha  v_{(2)}.
\]
\end{theorem}

\begin{proof}
To start with, and for reference in this proof, let us write down again what it explicitly means for $\gamma$ to be a morphism of $A$-corings for the underlying left bialgebroids $(U, A, \gD_\uhhu, \gve_\uhhu, s_\uhhu, t_\uhhu)$ and $(V, A, \gD_\vauu, \gve_\vauu, s_\vauu, t_\vauu)$: we have
\begin{equation}
\label{krach3}
\begin{array}{cccc}
(\gamma \otimes _\ahha \gamma) \circ \gD_\uhhu = \gD_\vauu \circ \gamma, \quad & \gve_\vauu = \gve_\uhhu \circ \gamma,  \quad & \gamma \circ s_\uhhu = s_\vauu, \quad &  \gamma \circ t_\uhhu = t_\vauu.
\end{array}
   \end{equation}
Observe that in contrast to $\pi$, the map $\gamma$ in general is neither a morphism of algebras nor of Hopf structures in the sense of \eqref{krach2}.

Next, let us introduce a couple of maps that will be instrumental in proving that \eqref{eq:Phi} and \eqref{eq:Psi} are inverse to each other.
Since both $\pi$ and $\gamma$ are $A$-coring morphisms which are compatible with the black actions (see \eqref{pergolesi}),
\[
\delta_\uhhu  \coloneqq \left(U \otimes_\ahha  \gamma\right)\circ (U \otimes_\ahha  \pi) \circ \Delta_\uhhu,\qquad u\mapsto u_{(1)}\otimes_\ahha \gamma\big(\pi(u_{(2)})\big) 
\]
defines an additional right $U$-coaction on $U$ itself. If $\xi\colon U \to U \times_{\Aopp} U$ denotes the translation map, that is to say, $\xi(u) = \beta^{-1}(u \otimes_\ahha  1)$ for all $u \in U$, and 
\[\alpha\colon(U \times_\ahha U) \times_{\Aopp} U \to U \times_\ahha U \times_{\Aopp} U \qquad \text{and} \qquad \alpha'\colon U \times_\ahha (U \times_{\Aopp} U) \to U \times_\ahha U \times_{\Aopp} U\]
denote the canonical maps associated to the double Takeuchi $\times$-product, then by mimicking the argument used in \cite[proof of Thm.~4.1.1]{CheGavKow:DFOLHA}, we observe that there are two well-defined maps
\[
\begin{array}{rrcl}
\phi \coloneqq (U \times_\ahha \varepsilon_\uhhu) \times_{\Aopp} U \colon  & 
(U \times_\ahha U) \times_{\Aopp} U &\to& U \times_{\Aopp} U, 
\\
\psi \coloneqq \alpha'\circ (U \times_\ahha \xi) \circ \delta_\uhhu  \colon  &  U  & \to & U \times_\ahha U \times_{\Aopp} U,
\end{array}
\]
where by $\gamma(v \bract a) \stackrel{\scriptscriptstyle\eqref{vanilly}}{=} \gamma(v) \bract a$ one deduces that $\delta_\uhhu $ takes values in $U \times_\ahha U$.
Since $U_{\ract}$ is projective as right $A$-module, by \cite[Prop.~1.7]{Takeuchi} both $\alpha$ and $\alpha'$ are isomorphisms and hence we can consider
\[
\varphi \coloneqq \left(\mu_\uhhu \circ \phi \circ \alpha^{-1} \circ \psi\right) \colon  U \to U, \qquad u \mapsto t_\uhhu \gve_\uhhu\big(\gamma\pi(u_{(2)})_+\big)\,
u_{(1)}\, \gamma\pi(u_{(2)})_-,
\]
where $\mu_\uhhu \colon  U \times_{\Aopp} U \to U$ is induced by the multiplication of $U$.
We claim that the map $\varphi$ just introduced maps $U$ to the left Hopf kernel $B$, inducing an $A$-bilinear morphism $\varphi \colon  \due U \lact \ract \to \due B \lact \bract$. In fact, for all $u \in U$ we have
$$
\Delta_\uhhu\big(\varphi(u)\big)
  \stackrel{ \eqref{beethoven2}}{=}
    u_{(1)}\, \gamma\big(\pi(u_{(3)})\big)_{-(1)}
  \otimes_\ahha 
    t_\uhhu \gve_\uhhu\pig(\gamma\big(\pi(u_{(3)})\big)_{+}\pig)\,
  u_{(2)}\,\gamma\left(\pi(u_{(3)})\right)_{-(2)}
$$
and therefore
\begin{small}
\begin{align*}
    &
 (U \otimes_\ahha  \pi) \gD_\uhhu \varphi(u)
  \stackrel{ \eqref{krach1}}{=}
    u_{(1)} \gamma\big(\pi(u_{(3)})\big)_{-(1)}
  \otimes_\ahha 
    t_\vauu \gve_\uhhu\pig(\gamma\big(\pi(u_{(3)})\big)_{+}\pig)\,
  \pi(u_{(2)})\,\pi\left( \gamma\pig(\pi(u_{(3)})\pig)_{-(2)}\right)
  \\
  &
  \stackrel{ \eqref{krach1}}{\mathmakebox[\widthof{$\stackrel{\scriptscriptstyle\eqref{Sch7}, \eqref{mandarino}}{=}$}]{=}}
    u_{(1)} \gamma\big(\pi(u_{(2)})_{(2)}\big)_{-(1)}
  \otimes_\ahha 
    t_\vauu \gve_\uhhu\big(\gamma(\pi(u_{(2)})_{(2)})_{+}\big)
  \pi(u_{(2)})_{(1)}\pi\pig( \gamma\big(\pi(u_{(2)})_{(2)}\big)_{-(2)}\pig)
  \\
  &
  \stackrel{(\pi\gamma = \id_\vauu)}{\mathmakebox[\widthof{$\stackrel{\scriptscriptstyle\eqref{Sch7}, \eqref{mandarino}}{=}$}]{=}}
    u_{(1)} \gamma\big(\pi(u_{(2)})_{(2)}\big)_{-(1)}
  \otimes_\ahha 
    t_\vauu \gve_\uhhu\big(\gamma(\pi(u_{(2)})_{(2)})_{+}\big)
  \pi\pig(\gamma\big(\pi(u_{(2)})_{(1)}\big)\pig)\pi\pig( \gamma\big(\pi(u_{(2)})_{(2)}\big)_{-(2)}\pig)
  \\
  &
  \stackrel{\scriptscriptstyle\eqref{krach3}}{\mathmakebox[\widthof{$\stackrel{\scriptscriptstyle\eqref{Sch7}, \eqref{mandarino}}{=}$}]{=}}
    u_{(1)} \gamma\big(\pi(u_{(2)})\big)_{(2)-(1)}
  \otimes_\ahha 
    t_\vauu \gve_\uhhu\big(\gamma(\pi(u_{(2)}))_{(2)+}\big)
  \pi\pig(\gamma\big(\pi(u_{(2)})\big)_{(1)}\pig)\pi\pig( \gamma\big(\pi(u_{(2)})\big)_{(2)-(2)}\pig)
  \\
    &
  \stackrel{\scriptscriptstyle\eqref{Sch4},(*)}{\mathmakebox[\widthof{$\stackrel{\scriptscriptstyle\eqref{Sch7}, \eqref{mandarino}}{=}$}]{=}}
    u_{(1)} \gamma\big(\pi(u_{(2)})\big)_{-(1)}
  \otimes_\ahha 
  \pi\pig(\gamma\big(\pi(u_{(2)})\big)_{+}\gamma\big(\pi(u_{(2)})\big)_{-(2)}\pig)
  \\
  &
  \stackrel{\scriptscriptstyle\eqref{Sch5},(**)}{\mathmakebox[\widthof{$\stackrel{\scriptscriptstyle\eqref{Sch7}, \eqref{mandarino}}{=}$}]{=}}
      u_{(1)} \gamma\big(\pi(u_{(2)})\big)_{-}
  \otimes_\ahha 
  \pi\pig(\gamma\big(\pi(u_{(2)})\big)_{++} \gamma\big(\pi(u_{(2)})\big)_{+-}\pig)
  \\
  &
  \stackrel{\scriptscriptstyle\eqref{Sch7}, \eqref{mandarino}}{=}
 t_\uhhu \gve_\uhhu \big(\gamma\left(\pi(u_{(2)})\right)_{+}\big)
  u_{(1)} \gamma\big(\pi(u_{(2)})\big)_{-}
  \otimes_\ahha 
  1_\vauu = \varphi(u) \otimes_\ahha  1_\vauu,
\end{align*}
\end{small}
where in $(*)$ we applied \eqref{Sch4} as follows: notice that
\[
U \times_\ahha U \times_{\Aopp} U \xrightarrow{\alpha^{-1}} \big(U \times_\ahha U\big) \times_{\Aopp} U \xrightarrow{(U \times_\ahha \varepsilon_\uhhu) \times_{\Aopp} U} U \times_{\Aopp} U \xrightarrow{\subseteq } U \otimes_{\Aopp} U
\]
is well-defined and left $A$-linear, whence from the identity \eqref{Sch4} we deduce that
\begin{align*}
& u_{(1)} \otimes_\ahha  t_\uhhu\varepsilon_\uhhu\big(\gamma\pi(u_{(2)})_{(2)+}\big)\gamma\pi(u_{(2)})_{(1)} \otimes_{\Aopp} \gamma\pi(u_{(2)})_{(2)-} \\
& \stackrel{\phantom{\scriptscriptstyle\eqref{beethoven3}}}{=} u_{(1)} \otimes_\ahha  t_\uhhu\varepsilon_\uhhu\big(\gamma\pi(u_{(2)})_{+(2)}\big)\gamma\pi(u_{(2)})_{+(1)} \otimes_{\Aopp} \gamma\pi(u_{(2)})_{-} \\
& \stackrel{\scriptscriptstyle\eqref{beethoven3}}{=} u_{(1)} \otimes_\ahha  \gamma\pi(u_{(2)})_{+} \otimes_{\Aopp} \gamma\pi(u_{(2)})_{-}.
\end{align*}
Moreover, since $\gD_\uhhu u \in U \times_\ahha U$ along with the specific $A$-linearities \eqref{krach1} as well as \eqref{vanilly} of $\pi$ resp.\ $\gamma$, we can further deduce that
\begin{align*}
& u_{(1)}\gamma\pi(u_{(2)})_{(2)-(1)} \otimes_\ahha  t_\uhhu\varepsilon_\uhhu\big(\gamma\pi(u_{(2)})_{(2)+}\big)\gamma\pi(u_{(2)})_{(1)} \gamma\pi(u_{(2)})_{(2)-(2)} \\
& =  u_{(1)}\gamma\pi(u_{(2)})_{-(1)} \otimes_\ahha  \gamma\pi(u_{(2)})_{+} \gamma\pi(u_{(2)})_{-(2)},
\end{align*}
and hence $(*)$ follows. Similarly, in $(**)$ we applied \eqref{Sch5} as follows: one has
\begin{align*} 
& u_{(1)} \otimes_\ahha  \gamma\pi(u_{(2)})_{+} \otimes_{\Aopp} \gamma\pi(u_{(2)})_{-(1)} \otimes_{\ahha} \gamma\pi(u_{(2)})_{-(2)} \\
&= u_{(1)} \otimes_\ahha  \gamma\pi(u_{(2)})_{++} \otimes_{\Aopp} \gamma\pi(u_{(2)})_{-} \otimes_{\ahha} \gamma\pi(u_{(2)})_{+-} 
\end{align*}
in $U \times_\ahha \left(U \times_{\Aopp} U \times_\ahha U\right) \simeq U \times_\ahha U \times_{\Aopp} U \times_\ahha U$ because the leftmost $U_{\ract}$ is projective. Finally, by performing
\[
U \times_\ahha U \times_{\Aopp} U \times_\ahha U \to U \times_\ahha U \otimes_{\Aopp} U \times_\ahha U \xrightarrow{\mu_{U \times_\ahha U}} U \times_\ahha U,
\]
we conclude that
\[ 
u_{(1)}\gamma\pi(u_{(2)})_{-(1)} \otimes_\ahha  \gamma\pi(u_{(2)})_{+} \gamma\pi(u_{(2)})_{-(2)} = u_{(1)}\gamma\pi(u_{(2)})_{-} \otimes_\ahha  \gamma\pi(u_{(2)})_{++} \gamma\pi(u_{(2)})_{+-}
\]
which is $(**)$. 

The $\K$-linear maps \eqref{eq:Phi} and \eqref{eq:Psi} can now be seen as a composition of well-defined maps
\[
\Psi \coloneqq \left(U \xrightarrow{\delta_\uhhu } \due U {} \ract \otimes_\ahha  \due U \lact {} \xrightarrow{\varphi~\otimes_\ahha ~\pi} B_\bract \otimes_\ahha  \due V \lact {}\right),
\]
and
\[
\Phi \coloneqq \left(B_\bract \otimes_\ahha  \due V \lact {} \xrightarrow{i~\otimes_\ahha ~\gamma} \due U {} \bract \otimes_\ahha  \due U \lact {} \xrightarrow{\mu_\uhhu} U\right),
\]
where $i \colon  B \to U$ is simply the inclusion. Let us verify that they are inverse to each other, indeed.

Since $\pi$ and $\gamma$ are morphisms of $A$-corings, we have
\begin{align*}
    \Phi\big(\Psi(u)\big) & \stackrel{\phantom{\scriptscriptstyle\eqref{Sch3},(\star)}}{=} t_\uhhu\varepsilon_\uhhu\big(\gamma\pi(u_{(2)})_+\big)u_{(1)} \gamma\pi(u_{(2)})_-\gamma\pi(u_{(3)}) \\
    & \stackrel{\phantom{\scriptscriptstyle\eqref{Sch3},(\star)}}{=} t_\uhhu\varepsilon_\uhhu\big(\gamma\pi(u_{(2)})_{(1)+}\big)u_{(1)} \gamma\pi(u_{(2)})_{(1)-}\gamma\pi(u_{(2)})_{(2)} \\
    & \stackrel{\scriptscriptstyle\eqref{Sch3},(\star)}{=}  t_\uhhu\varepsilon_\uhhu\big(\gamma\pi(u_{(2)})\big)u_{(1)} = u,
\end{align*}
that is, $\Phi\circ\Psi=\id_\uhhu$ as desired, where in $(\star)$ we used \eqref{Sch3} to obtain
\[
u_{(1)} \otimes_\ahha  \gamma\pi(u_{(2)})_{(1)+} \otimes_\Aopp \gamma\pi(u_{(2)})_{(1)-}\gamma\pi(u_{(2)})_{(2)} = u_{(1)} \otimes_\ahha  \gamma\pi(u_{(2)}) \otimes_\Aopp 1_\uhhu 
\]
in $U \times_\ahha (U \otimes_\Aopp U)$. Since $U_{\ract}$ is projective as an $A$-module, $(U \times_\ahha  U) \otimes_\Aopp U \subseteq U \otimes_\ahha  U \otimes_\Aopp U$ and its image coincides with $U \times_\ahha (U \otimes_\Aopp U)$, whence applying
\[
U \times_\ahha (U \otimes_\Aopp U) \simeq (U \times_\ahha  U) \otimes_\Aopp U \xrightarrow{(U \times_\ahha \varepsilon_\uhhu) \otimes_\Aopp U} U \otimes_\Aopp U \xrightarrow{\mu_\uhhu} U
\] 
to the last equality yields $(\star)$.

Secondly, we check that $\Psi\circ\Phi=\id_{\scriptscriptstyle B \otimes_\ahha  V}$.
Observe first that for all $b \in B$ and all $u \in U$ 
\[
\begin{aligned}
\delta_\uhhu (bu) & = b_{(1)}u_{(1)} \otimes_\ahha  \gamma\big(\pi(b_{(2)})\pi(u_{(2)})\big) = bu_{(1)} \otimes_\ahha  \gamma\pi(u_{(2)})
\end{aligned}
\]
holds
by the very definition of the left Hopf kernel $B$ of $\pi$ as in \eqref{mandarino}, along with the unitality of $\pi$.
Therefore, since $B$ commutes with the image of $t_\uhhu$ as in \eqref{ohyes}, the map $\delta_\uhhu $ and hence $\psi$, $\varphi$, $\Psi$, and $\Phi$ are left $B$-linear with respect to the regular $B$-module structure.
Moreover, 
\[
\Psi \circ \gamma = (\varphi \otimes_\ahha  \pi)\circ \delta_\uhhu  \circ \gamma = (\varphi \otimes_\ahha  \pi)\circ (\gamma \otimes_\ahha  \gamma) \circ \Delta_\vauu = (\varphi\gamma \otimes_\ahha  V)\circ \Delta_\vauu
\]
and since
\[
\begin{aligned}
    \varphi \circ \gamma & \stackrel{\phantom{\scriptscriptstyle\eqref{Sch4}}}{=} \mu_\uhhu \circ \phi \circ \alpha^{-1} \circ \psi \circ \gamma \\
    & \stackrel{\phantom{\scriptscriptstyle\eqref{Sch4}}}{=} \mu_\uhhu \circ \left((U \times_\ahha \varepsilon_\uhhu) \times_{\Aopp} U\right) \circ \alpha^{-1} \circ \alpha'\circ (U \times_\ahha \xi) \circ \delta_\uhhu  \circ \gamma \\
    & \stackrel{\phantom{\scriptscriptstyle\eqref{Sch4}}}{=} \mu_\uhhu \circ \left((U \times_\ahha \varepsilon_\uhhu) \times_{\Aopp} U\right) \circ \alpha^{-1} \circ \alpha'\circ (U \times_\ahha \xi) \circ \delta_\uhhu  \circ \gamma \\
    & \stackrel{\scriptscriptstyle\eqref{Sch4}}{=} \mu_\uhhu \circ \left((U \times_\ahha \varepsilon_\uhhu) \times_{\Aopp} U\right) \circ \alpha^{-1} \circ \alpha'\circ (\delta_\uhhu  \times_\Aopp U) \circ \xi \circ \gamma \\
    & \stackrel{\phantom{\scriptscriptstyle\eqref{Sch4}}}{=} \mu_\uhhu \circ \xi \circ \gamma \stackrel{\scriptscriptstyle\eqref{Sch7}}{=} s_\uhhu \circ \varepsilon_\uhhu \circ \gamma  = s_\uhhu \circ \varepsilon_\vauu,
\end{aligned}
\]
it follows that $\Psi \circ \gamma = s_\uhhu \otimes_\ahha  V$. Therefore, for all $b \in B$ and $v \in V$, we see that
\[
\Psi \Phi(b \otimes_\ahha  v) = \Psi(b\gamma(v)) = b\Psi(\gamma(v)) = b\otimes_\ahha  v.
\]
Summing up, $U \simeq B \otimes_\ahha  V$ as $\K$-vector spaces.
The various statements in Theorem \ref{thm:sigmatwisted} now follow by transport of structure along this isomorphism. Indeed, notice that the assignment \eqref{weakadjoint} is obtained as the composition 
\[V \otimes B \xrightarrow{\gamma\otimes B} U \otimes B \xrightarrow{\rightslice} B,\]
where the latter map is \eqref{adjointaction} for the $U$-module algebra structure on $B$ from Lemma \ref{wasserkessel}. Therefore, since by hypothesis $\gamma$ is a morphism of $A$-corings and since it satisfies $\gamma(1_\vauu) = 1_\uhhu $, the above is a weak left $V$-action in the sense of Definitions \ref{GT1} and \ref{GT2}. By transport of structure we can define a multiplication on $B \otimes_\ahha  V$ by
\[(b \otimes_\ahha  v)\cdot (b' \otimes_\ahha  w) \coloneqq \Psi\big(\,b\,\gamma(v)\,b'\gamma(w)\,\big).\]
Also observe that,
for every $b \in B$, the assignment 
\[
\mu_{b} \colon \due U \blact {} \otimes_\Aopp \due U {} \ract \to U,\quad u \otimes_\Aopp u'\mapsto ubu',
\] 
is well-defined once more by \eqref{ohyes}. Consequently, \eqref{Sch3} implies
\[\gamma(v)\,b' = \gamma(v)_{(1)+}b'\gamma(v)_{(1)-}\gamma(v)_{(2)}\]
and hence
\begin{small}
\begin{eqnarray*}
& & (b \otimes_\ahha  v)\cdot (b' \otimes_\ahha  w) = \Psi\big(b\gamma(v)b'\gamma(w)\big)
\\
&=& b\gamma(v)_{(1)+}b'\gamma(v)_{(1)-}\Psi\big(\gamma(v)_{(2)}\gamma(w)\big)\\
&\stackrel{\scriptscriptstyle\eqref{krach3}}{=}& b\gamma(v_{(1)})_{+}b'\gamma(v_{(1)})_{-}\Psi\big(\gamma(v_{(2)})\gamma(w)\big)
\\
&\stackrel{\scriptscriptstyle\eqref{weakadjoint}}{=} &
b\big(v_{(1)} \rightslice b'\big)\Psi\big(\gamma(v_{(2)})\gamma(w)\big) 
\\
& \stackrel{\scriptscriptstyle (\dag)}{=} &
b\big(v_{(1)} \rightslice b'\big)t_\uhhu\varepsilon_\uhhu\big(\gamma\pi(\gamma(v_{(3)})\gamma(w)_{(2)})_{+}\big)\gamma(v_{(2)})\gamma(w)_{(1)}\gamma\pi(\gamma(v_{(3)})\gamma(w)_{(2)})_{-} \otimes_\ahha  \pi\left(\gamma(v_{(4)})\gamma(w)_{(3)}\right) 
\\
& \stackrel{\scriptscriptstyle (\ddag)}{=} &
b(v_{(1)}\rightslice b')t_\uhhu\varepsilon_\uhhu\big(\gamma(v_{(3)}w_{(2)})_{+}\big)\gamma(v_{(2)})\gamma(w_{(1)})\gamma(v_{(3)}w_{(2)})_{-} \otimes_\ahha  v_{(4)}w_{(3)},
\end{eqnarray*}
\end{small}
where in $(\dag)$ we applied \eqref{krach3} as follows: since $\gamma$ is a map of $A$-corings, 
we have
\[
\big((\mu_{b'} \otimes_\ahha  \Psi) \circ (\xi \otimes_\ahha  U)\big)\big(\gamma(v)_{(1)} \otimes_\ahha  \gamma(v)_{(2)}\gamma(w)\big) = \big((\mu_{b'} \otimes_\ahha  \Psi) \circ (\xi \otimes_\ahha  U)\big)\big(\gamma(v_{(1)}) \otimes_\ahha  \gamma(v_{(2)})\gamma(w)\big),
\]
which entails $(\dag)$. In $(\ddag)$ we applied the fact that $\pi$ is a morphism of $\Ae$-rings, $\gamma$ is a morphism of $A$-corings and $\pi \circ \gamma = \id_\vauu$.
Finally, defining
\[
\sigma \colon  V \times V \to B, \qquad (v,w) \mapsto \left(B \otimes_\ahha  \varepsilon_\vauu\right)\pig(\big(1_\behhe  \otimes_\ahha  v\big)\cdot \big(1_\behhe  \otimes_\ahha  w\big)\pig),
\]
one obtains the expression \eqref{chesorpresa} that factors through the tensor product $\due V \blact {} \otimes_\Aopp \due V {} \ract$ because of $\Delta\big(t(a)\big) = 1 \otimes_{\ahha} t(a)$ for all $a \in A$.
We conclude that $\gs$  is a Hopf cocycle by Proposition \ref{GT3}.
\end{proof}


\section{Smash and crossed product decomposition of universal enveloping algebras}\label{sec:weakcase}

In this section, we will see how to use the just obtained Theorem \ref{thm:sigmatwisted} in the context of various crossed product decompositions arising from short exact sequences of Lie-Rinehart algebras.


\subsection{Short exact sequences of Lie-Rinehart algebras and left Hopf kernels}\label{ssec:sesHopfker}

Assume that we have a short exact sequence 
\[
0 \to \mf{n} \xrightarrow{\iota} \mf{g} \xrightarrow{\pi} \mf{h} \to 0
\]
of Lie-Rinehart algebras over $A$ which are projective as (left) $A$-modules. 
By functoriality of $\cU_\ahha (-)\colon \mathsf{LieRin}_\ahha  \to \mathsf{Bialgd}_\ahha $, the maps $\iota$ and $\pi$ induce morphisms of left Hopf algebroids
\begin{equation}\label{eq:I}
I\colon U_{\ahha}(\mf{n}) \to \cU_\ahha(\mf{g}) \qquad \text{and} \qquad \Pi\colon \cU_\ahha(\mf{g}) \to \cU_\ahha(\mf{h}). 
\end{equation}
Notice that, by a direct check, for all $a \in A$ and $X \in \mf{n}$
\[
(\Pi\circ I)(\iota_{\mf{n}}(X)) = \Pi\left(\iota_{\mf{g}}\left(\iota(X)\right)\right) = \iota_{\mf{h}}\left(\pi\left(\iota(X)\right)\right) = 0 \qquad
\text{and}\qquad
(\Pi\circ I)(a) = \Pi\left(a\right) = a
\]
because they are morphisms of $A$-rings.  Therefore, by the universal property of $U_{\ahha}(\mf{n})$,
we obtain $\Pi \circ I = \varepsilon$. 
Moreover, since $\pi$ is surjective and $\cU_\ahha(\mf{h})$ is generated by $A$ and $\mf{h}$ as an algebra, $\Pi$ is surjective as well.

\begin{lemma}\label{lem:allinjnew}
Being $\mf{n}$ and $\mf{g}$ projective as left $A$-modules, the map $I$ is an injective morphism.
\end{lemma}

\begin{proof}
The statement follows from the injectivity of $\iota$ and of $\iota_{\mf{g}}$, along with a variation of the Heyneman-Radford Theorem (see \cite[Thm.~4.13]{Saracco} or \cite[Lem.~A.1]{MoerdijkLie}).

As in \cite[\S2]{MoerdijkLie}, since $\mf{n}$ is projective, $\overline{U_{\ahha}(\mf{n})}  \coloneqq \ker(\varepsilon)$ is a graded projective cocomplete non-counital cocommutative coalgebra. Its natural filtration (whose $n$-th term is composed by all elements that can be written as a product of at most $n$ elements of $\mf{n}$) coincides with its primitive filtration and it is composed by projective $A$-modules. In particular, in the notation of \cite{MoerdijkLie}, $\overline{\cU_{\ahha}(\mf{n})}_1 = \mf{n}$ and $\overline{\cU_\ahha(\mf{g})}_1 = \mf{g}$. Now, Lemma A.1 in {\em op.\ cit.}~states that if $\overline{I}\colon \overline{\cU_{\ahha}(\mf{n})} \to \overline{\cU_\ahha(\mf{g})}$ is injective when (co)restricted to $\overline{I}_1\colon \overline{\cU_{\ahha}(\mf{n})}_1 \to \overline{\cU_\ahha(\mf{g})}_1$, then it is injective. Here,  $\overline{I}_1 =  \iota$, which is injective, indeed. To conclude, injectivity of $I$ follows from the fact that $\cU_{\ahha}(\mf{n}) = A \oplus \overline{\cU_{\ahha}(\mf{n})}$ and $I$ is already injective on $A$ because it is a morphism of $A$-rings.
\end{proof}

The following proposition provides us with an effective way to check if an element in $\cU_\ahha(\mf{g})$ belongs to $U_{\ahha}(\mf{n})$ or not.

\begin{proposition}
\label{prop:neverendingprop}
Via the morphism $I\colon U_{\ahha}(\mf{n}) \to \cU_\ahha(\mf{g})$ from \eqref{eq:I}, we have that 
\[
U_{\ahha}(\mf{n}) \simeq B_{\ahha}(\mf{g})  \coloneqq \pig\{x \in \cU_\ahha(\mf{g}) ~\pig|~ \big(\cU_\ahha(\mf{g}) \otimes_{\ahha} \Pi\big)\left(\Delta_{\cU_\ahha(\mf{g})}(x)\right) = x \otimes_{\ahha} 1 \pig\}.
\]
In particular, up to the canonical morphism $I$, we may identify $U_{\ahha}(\mf{n})$ with $B_{\ahha}(\mf{g})$.
\end{proposition}

We divide the proof into a number of steps, reported in the forthcoming lemmata. We first show 
that the morphism $I\colon U_{\ahha}(\mf{n}) \to \cU_\ahha(\mf{g})$ from \eqref{eq:I} lands in $B_{\ahha}(\mf{g})$, and second, that the induced morphism is an isomorphism whenever $\mf{n}$, $\mf{g}$, as well as $\mf{h}$ are free over $A$. Then we prove that for every prime ideal $\mf{p}$ in $A$, we have $U_{\ahha}(\mf{n})_{\mf{p}}\simeq U_{\scriptscriptstyle{A_{\mf{p}}}}(\mf{n}_{\mf{p}})$ and $B_{\ahha}(\mf{g})_{\mf{p}} \simeq B_{\scriptscriptstyle{A_{\mf{p}}}}(\mf{g}_{\mf{p}})$.
Finally, we conclude by recalling that bijectivity for a morphism is a local property.

\begin{lemma}
\label{lem:iso1}
The morphism $I\colon U_{\ahha}(\mf{n}) \to \cU_\ahha(\mf{g})$ from \eqref{eq:I} maps $U_{\ahha}(\mf{n})$ into $B_{\ahha}(\mf{g})$.
\end{lemma}

\begin{proof}
It is straightforward to check that the $\K$-algebra morphisms $\left(\cU_\ahha(\mf{g}) \tak{\ahha} \Pi\right) \circ \Delta_{\cU_\ahha(\mf{g})} \circ I$ and
\[j\colon U_{\ahha}(\mf{n}) \to \cU_\ahha(\mf{g}) \tak{\ahha} \cU_\ahha(\mf{h}), \qquad u \mapsto I(u) \otimes_{\ahha} 1,\]
are well-defined and satisfy
\begin{eqnarray*}
\pig(\big(\cU_\ahha(\mf{g}) \tak{\ahha} \Pi\big) \circ \Delta_{\cU_\ahha(\mf{g})} \circ I\pig)(a) \!\!\!&=\!\!\!& a \otimes_{\ahha} 1 = j(a) \qquad \text{as well as} 
\\
\pig(\big(\cU_\ahha(\mf{g}) \tak{\ahha} \Pi\big) \circ \Delta_{\cU_\ahha(\mf{g})} \circ I\pig)\big(\iota_{\mf{n}}(U)\big) \!\!\!&=\!\!\!& \iota_{\mf{g}}\big(\iota(U)\big) \otimes_{\ahha} 1 + 1 \otimes_{\ahha} \iota_{\mf{h}}\pig(\pi\big(\iota(U)\big)\pig) = j\big(\iota_{\mf{n}}(U)\big)
\end{eqnarray*}
for all $a \in A, U \in \mf{n}$. Therefore, by the uniqueness part of the universal property of $U_{\ahha}(\mf{n})$, $\left(\cU_\ahha(\mf{g}) \tak{\ahha} \Pi\right) \circ \Delta_{\cU_\ahha(\mf{g})} \circ I = j$ and so $I$ lands in $B_{\ahha}(\mf{g})$, as claimed.
\end{proof}

\begin{lemma}
\label{lem:iso2}
If $\mf{n}$, $\mf{g}$, and $\mf{h}$ are free left $A$-modules, then the morphism $U_{\ahha}(\mf{n}) \to B_{\ahha}(\mf{g})$ induced by $I$ from Lemma \ref{lem:iso1} is an isomorphism. In particular, we may identify $U_{\ahha}(\mf{n})$ with $B_{\ahha}(\mf{g})$.
\end{lemma}

\begin{proof}
Assume that $\mf{n}$, $\mf{g}$, and $\mf{h}$ are free left $A$-modules and pick an element
\[
a_0 + \sum_{n = 1}^{N} a_nX_{n,1} \cdots X_{n,n} \in B_{\ahha}(\mf{g}), 
\]
where $X_{n,1}\leq \cdots \leq X_{n,n}$ are elements of an ordered basis of $\mf{g}$ and where $a_i \in A$ for all $i$. The condition of belonging to $B_{\ahha}(\mf{g})$ entails that
\[a_0 \otimes_{\ahha} 1 + \sum_{n = 1}^{N} a_nX_{n,1} \cdots X_{n,n}\otimes_{\ahha} 1 = a_0 \otimes_{\ahha} 1 + \sum_{n = 1}^{N}\sum_{s+t=n} a_nX_{n,i_1} \cdots X_{n,i_t} \otimes_{\ahha} \Pi\left(X_{n,j_1} \cdots X_{n,j_s}\right),\]
where on the right-hand side the second sum involves only shuffles of $X_{n,1}\leq \cdots \leq X_{n,n}$, that is to say, $X_{n,i_1} \leq \cdots \leq X_{n,i_t}$ and $X_{n,j_1} \leq \cdots \leq X_{n,j_s}$ still hold. Therefore,
\[0 = \sum_{n = 1}^{N}\sum_{\substack{ s+t=n \\ s\geq 1}} a_nX_{n,i_1} \cdots X_{n,i_t} \otimes_{\ahha} \pi\left(X_{n,j_1}\right) \cdots \pi\left(X_{n,j_s}\right).\]
By the Poincar\'e-Birkhoff-Witt theorem, the left-hand side tensorands are still elements of a basis of $\cU_\ahha(\mf{g})$ over $A$, whence, in particular,
\[a_N\pi\left(X_{N,k}\right) = 0\]
for all $k = 1, \ldots, N$ because these are the right-hand tensorands of the summands of the form $a_NX_{N,i_1} \cdots X_{N,i_{N-1}} \otimes_{\ahha} \pi\left(X_{N,k}\right)$. In principle, one may have repetitions among the $X_{N,k}$ (as it happens for $X^2\otimes_{\ahha} 1$ with $X \in \mf{n}$, which gives rise to the identity $0 = 2X \otimes_{\ahha} \pi(X) + 1 \otimes_{\ahha} \pi(X)^2$ for instance, from which one deduces that $2\pi(X) = 0$), but since we are working over a field of characteristic $0$, this does not affect the argument. As a consequence, if $a_N\neq 0$, then $\pi(X_{N,k}) = 0$ for all $k = 1,\ldots,N$ (because $\mf{h}$ is free over $A$), and hence $X_{N,k} \in \mf{n}$. In view of the latter, we may consider further the element
\[
a_0 + \sum_{n = 1}^{N} a_nX_{n,1} \cdots X_{n,n} - a_NX_{N,1}\cdots X_{N,N} = a_0 + \sum_{n = 1}^{N-1} a_nX_{n,1} \cdots X_{n,n} \in B_{\ahha}(\mf{g})
\]
and iterate the argument before to conclude that $a_0 + \sum_{n = 1}^{N} a_nX_{n,1} \cdots X_{n,n}$ is, in fact, the image of an element in $U_{\ahha}(\mf{n})$ by means of $I$.
\end{proof}

See \cite[Ex.~4.20]{BlattnerCohenMontgomery} for an indirect proof of the same claim (notice that, even if $\K$ is assumed to be a field therein, their results still hold for $\K$ a commutative ring, provided that the Lie algebras under consideration are free over $\K$).

Summing up, we proved Proposition \ref{prop:neverendingprop} in case the involved Lie-Rinehart algebras are free as left $A$-modules. Let us go back to the general case in which $\mf{n}$, $\mf{g}$, and $\mf{h}$ are projective $A$-modules.

\begin{lemma}
\label{lem:iso4}
Let $\mf{p}$ be a prime ideal in $A$. For any Lie-Rinehart algebra $(A,\mf{h},\omega)$ over $A$, the localisation $\mf{h}_{\mf{p}}$ of the left $A$-module $\mf{h}$ at $\mf{p}$ admits the Lie-Rinehart algebra structure $\left(A_{\mf{p}}, A_{\mf{p}} \rtimes_{\ahha} \mf{h}\right)$ over $A_{\mf{p}}$ introduced in Proposition \ref{prop:iso3}\ref{item:iso3b}, and this construction induces a functor $\ms{LieRin}_\ahha \to \ms{LieRin}_{A_{\mf{p}}}$ given on objects resp.\ morphisms by
\[
(A,\mf{h},\omega) \mapsto (A_{\mf{p}}, \mf{h}_{\mf{p}}, \omega_\mf{p}) \qquad \text{and} \qquad f \mapsto A_{\mf{p}} \otimes_\ahha  f,
\]
where $\omega_{\mf{p}}(X)\left(a/b\right) = \big(\omega(X)(a)b - a\omega(X)(b)\big)/b^2$ for all $a/b \in A_{\mf{p}}$ and $X \in \mf{h}$.
\end{lemma}

\begin{proof}
For $\mf{p}$ a prime ideal in $A$, we can consider the localisations $A_{\mf{p}}$ and $\mf{h}_{\mf{p}}$ of $A$ and $\mf{h}$ at $\mf{p}$, respectively. Recall that 
$\mf{h}_{\mf{p}} \simeq A_{\mf{p}} \otimes_{\ahha} \mf{h}$ 
as left $A_{\mf{p}}$-modules. 
Therefore, the first claim follows directly from point of Proposition \ref{prop:iso3}\ref{item:iso3b} with $R = A_{\mf{p}}$.
We leave it to the reader to verify the functoriality of the construction.
\end{proof}

\begin{lemma}
\label{lem:iso5}
Let $\mf{p}$ be a prime ideal in $A$. Then, $\cU_\ahha(\mf{h})_{\mf{p}}$ 
admits the structure of a smash product via the isomorphism $\cU_\ahha(\mf{h})_{\mf{p}} \simeq A_{\mf{p}} \otimes_{\ahha} \cU_\ahha(\mf{h})$ as well as 
\[
\cU_{A_{\mf{p}}}\left(\mf{h}_{\mf{p}}\right) \simeq  \cU_\ahha(\mf{h})_{\mf{p}}\simeq A_{\mf{p}}\!\hash{}\!\, \cU_\ahha(\mf{h})
\] 
as cocommutative left Hopf algebroids over $A_{\mf{p}}$, where the first isomorphism arises via the unique morphism induced by
\[\iota_{A_{\mf{p}}}\left(\frac{a}{b}\right) \mapsto \frac{\iota_\ahha (a)}{b}  \qquad \text{and} \qquad \iota_{\mf{h}_{\mf{p}}}\left(\frac{X}{b}\right) \mapsto  \frac{\iota_{\mf{h}}(X)}{b}.\]
\end{lemma}

\begin{proof}
In view of Lemma \ref{lem:iso4}, the statement follows from Proposition \ref{prop:iso3} with $R = A_{\mf{p}}$.
\end{proof}

\begin{lemma}
\label{lem:iso6}
Let $\mf{p}$ be a prime ideal in $A$. Then there is a canonical isomorphism
\[B_{A_{\mf{p}}}(\mf{g}_{\mf{p}}) \simeq A_{\mf{p}} \otimes_{\ahha} B_{\ahha}(\mf{g})\]
as left $A_{\mf{p}}$-modules.
\end{lemma}

\begin{proof}
Recall that
 \[
 B_{\ahha}(\mf{g}) = \pig\{x \in \cU_\ahha(\mf{g}) ~\pig|~ \big(\cU_\ahha(\mf{g}) \otimes_{\ahha} \Pi\big)\left(\Delta_{\cU_\ahha(\mf{g})}(x)\right) = x \otimes_{\ahha} 1\pig\}.
 \]
by definition. Therefore, as left $A$-module, $B_{\ahha}(\mf{g})$ is also the equaliser in $\Lmod{A}$ of
 \[
 \xymatrix @C=90pt{
 \cU_\ahha(\mf{g}) \ar@<+.5ex>[r]^-{(\id \, \otimes_{\ahha}\, \Pi) \, \circ \, \Delta} \ar@<-.5ex>[r]_-{\id \, \otimes_{\ahha}\, 1} & \cU_\ahha(\mf{g}) \otimes_{\ahha} \cU_\ahha(\mf{h})
 }
 \]
(since the $\times_\ahha$-product is a submodule of the tensor product over $A$, this is the same as the equaliser from Remark \ref{rem:equalizer}).
 As $A_{\mf{p}}$ is flat over $A$, the rows of the following commutative diagram are equaliser diagrams as well
  \[
  \xymatrix @C=40pt @R=25pt{
 A_{\mf{p}}\otimes_{\ahha}B_{\ahha}(\mf{g}) \ar@{<->}[d] \ar[r] &  A_{\mf{p}}\otimes_{\ahha}\cU_\ahha(\mf{g}) \ar@{<->}[d] \ar@<+.5ex>[rrr]^-{ \id\,\otimes_\ahha \,\left((\id \, \otimes_\ahha  \,\Pi) \, \circ \, \Delta\right)} \ar@<-.5ex>[rrr]_-{ \id \,\otimes_\ahha \,\left(\id \,\otimes_\ahha \, 1\right)} &&&  A_{\mf{p}}\otimes_\ahha \big(\cU_\ahha(\mf{g}) \otimes_{\ahha} \cU_\ahha(\mf{h})\big) \ar@{<->}[d] \\
 B_{\ahha}(\mf{g})_{\mf{p}} \ar[r] & \cU_\ahha(\mf{g})_{\mf{p}} \ar@<+.5ex>[rrr]^-{ (\id \,\otimes_\ahha \, \Pi_{\mf{p}} ) \, \circ \, \Delta_{\mf{p}}} \ar@<-.5ex>[rrr]_-{ \id \,\otimes_{{\scriptscriptstyle A_{\mf{p}}}}\, 1} &&& \cU_\ahha(\mf{g})_{\mf{p}} \otimes_{\scriptscriptstyle A_{\mf{p}}} \cU_\ahha(\mf{h})_{\mf{p}}.
 }
 \]
Furthermore, since the assignment $\mf{h} \mapsto \mf{h}_{\mf{p}}$ by Lemma \ref{lem:iso4} is functorial, there is a short exact sequence of Lie-Rinehart algebras
 \[ 0 \to \mf{n}_{\mf{p}} \xrightarrow{\iota_{\mf{p}}} \mf{g}_{\mf{p}} \xrightarrow{\pi_\mf{p}} \mf{h}_{\mf{p}} \to 0\]
(which still splits as sequence of left $A_{\mf{p}}$-modules),  
whence we may consider the equaliser diagram
 \[
 \xymatrix@C=40pt{
 B_{A_{\mf{p}}}(\mf{g}_{\mf{p}}) \ar[r] & \cU_{\scriptscriptstyle A_{\mf{p}}}(\mf{g}_{\mf{p}}) \ar@<+.5ex>[rr]^-{ (\id  \,\otimes_{{\scriptscriptstyle A_{\mf{p}}}}\, \Pi' ) \, \circ \, \Delta} \ar@<-.5ex>[rr]_-{ \id \, \otimes_{{\scriptscriptstyle A_{\mf{p}}}}\, 1} && \cU_{\scriptscriptstyle A_{\mf{p}}}(\mf{g}_{\mf{p}}) \otimes_{\scriptscriptstyle A_{\mf{p}}} \cU_{\scriptscriptstyle A_{\mf{p}}}(\mf{h}_{\mf{p}})
 }
 \]
 of left $A_{\mf{p}}$-modules. 
Now, the isomorphism $\cU(\mf{g})_{\mf{p}} \simeq \cU_{\scriptscriptstyle A_{\mf{p}}}(\mf{g}_{\mf{p}})$  of left Hopf algebroids over $A_{\mf{p}}$ from Lemma \ref{lem:iso5} implies that the diagram
 \[
 \xymatrix @C=40pt @R=25pt{
 B_{\ahha}(\mf{g})_{\mf{p}} \ar[r] & \cU_\ahha(\mf{g})_{\mf{p}} \ar@{<->}[d]
\ar@<+.5ex>[rrr]^-{ \big(\id \, \otimes_{\ahha} \,\Pi_{\mf{p}}\big)\circ \Delta_{\mf{p}}} \ar@<-.5ex>[rrr]_-{ \id \, \otimes_{{\scriptscriptstyle A_{\mf{p}}}}\, 1} &&& \cU_\ahha(\mf{g})_{\mf{p}} \otimes_{\scriptscriptstyle A_{\mf{p}}} \cU_\ahha(\mf{h})_{\mf{p}} \ar@{<->}[d] \\ B_{A_{\mf{p}}}(\mf{g}_{\mf{p}}) \ar[r] & \cU_{\scriptscriptstyle A_{\mf{p}}}(\mf{g}_{\mf{p}}) \ar@<+.5ex>[rrr]^-{ \big(\id \, \otimes_{{\scriptscriptstyle A_{\mf{p}}}} \,\Pi'\big)\circ \Delta} \ar@<-.5ex>[rrr]_-{ \id \, \otimes_{{A_{\mf{p}}}} \, 1} &&& \cU_{\scriptscriptstyle A_{\mf{p}}}(\mf{g}_{\mf{p}}) \otimes_{\scriptscriptstyle A_{\mf{p}}} \cU_{\scriptscriptstyle A_{\mf{p}}}(\mf{h}_{\mf{p}})
 }
 \]
 commutes, from which we deduce $B_{\ahha}(\mf{g})_{\mf{p}} \simeq B_{A_{\mf{p}}}(\mf{g}_{\mf{p}})$ as claimed. 
\end{proof}

We are now in a position to prove Proposition \ref{prop:neverendingprop}.

\begin{proof}[Proof of Proposition \ref{prop:neverendingprop}]
By Lemma \ref{lem:iso1}, the injective morphism $I \colon U_\ahha (\mf{n}) \to \cU_\ahha (\mf{g})$ lands in $B_{\ahha}(\mf{g})$, inducing therefore an injective morphism $I \colon U_\ahha (\mf{n}) \to B_{\ahha}(\mf{g})$. 
For every prime ideal $\mf{p}$ in $A$, we may consider the isomorphism of cocommutative left Hopf algebroids $J_{\mf{p}}^{\mf{g}} \colon \cU_\ahha(\mf{g})_{\mf{p}} \simeq \cU_{\scriptscriptstyle A_{\mf{p}}}(\mf{g}_{\mf{p}})$ from Lemma \ref{lem:iso5} and we may look at
\[
\xymatrix{
U_\ahha (\mf{n})_{\mf{p}} \ar[r]^-{I_{\mf{p}}} \ar[d]_-{J_{\mf{p}}^{\mf{n}}} & B_{\ahha}(\mf{g})_{\mf{p}} \ar[d]^-{J_{\mf{p}}^{\mf{g}}} \\
U_{\scriptscriptstyle A_{\mf{p}}}(\mf{n}_{\mf{p}}) \ar[r]_-{I'} & B_{\scriptscriptstyle A_\mf{p}}(\mf{g}_{\mf{p}}).
}
\]
In view of Lemma \ref{lem:iso2}, the bottom arrow $I'$ is an isomorphism (the localisation of a projective module is a projective module over the localised ring and every projective module over a local ring is free). By Lemma \ref{lem:iso5}, the left vertical arrow $J_{\mf{p}}^{\mf{n}}$ is an isomorphism, and Lemma \ref{lem:iso6} implies that the right vertical arrow $J_{\mf{p}}^{\mf{g}}$ is an isomorphism as well. Hence, we are left to check that the diagram commutes. Set $J^{\scriptscriptstyle \bullet} \coloneqq J^{\scriptscriptstyle \bullet}_{\mf{p}}\left(1_{\scriptscriptstyle A_\mf{p}} \otimes_{\ahha} -\right)$. Since all the morphisms are left $A_{\mf{p}}$-linear, it is enough to check that
\[
\left(I'\circ J^{\mf{n}}_{\mf{p}}\right)\left(1_{\scriptscriptstyle A_\mf{p}} \otimes_{\ahha} x\right) = \left(I'\circ J^{\mf{n}}\right)\left(x\right) \quad \text{equals} \quad \left(J^{\mf{g}}_{\mf{p}} \circ I_{\mf{p}}\right)\left(1_{\scriptscriptstyle A_\mf{p}} \otimes_{\ahha} x\right) = \left(J^{\mf{g}} \circ I\right)\left(x\right)
\]
for all $x \in U_\ahha(\mf{n})$.
Since all the latter morphisms are morphisms of $\K$-algebras (and of $A$-rings, in fact) and since
\begin{equation*}
\begin{split}
I'\left(J^{\mf{n}}(\iota_{\mf{n}}(U))\right) \stackrel{\scriptscriptstyle\eqref{eq:Jiotas}}{=} I'\left(\iota_{\mf{n}_{\mf{p}}}\left(\jmath_{\mf{n}}(U)\right)\right) 
&= \iota_{\mf{g}_{\mf{p}}}\left(\iota_{\mf{p}}\left(\jmath_{\mf{n}}(U)\right)\right) \\
&= \iota_{\mf{g}_{\mf{p}}}\left(\jmath_{\mf{g}}\left(\iota(U)\right)\right) \stackrel{\scriptscriptstyle\eqref{eq:Jiotas}}{=} J^{\mf{g}}\left(\iota_{\mf{g}}\left(\iota(U)\right)\right) = J^{\mf{g}}\left(I\left(\iota_{\mf{n}}(U)\right)\right)
\end{split}
\end{equation*}
for all $U \in \mf{n}$, the above diagram commutes by the universal property of $U_\ahha(\mf{n})$ and so $I_{\mf{p}}$ is bijective for every prime ideal $\mf{p}$ of $A$. As bijectivity is a local property, we conclude that $I$ is bijective.
\end{proof}

\subsection{Short exact sequences of Lie-Rinehart algebras and the crossed product decomposition}\label{ssec:crossdec}

Assume that $0 \to \mf{n} \xrightarrow{\iota} \mf{g} \xrightarrow{\pi} \mf{h} \to 0$ is a short exact sequence of Lie-Rinehart algebras, which are projective as left $A$-modules. Let $\gamma \colon \mf{h} \to \mf{g}$ be a section of $\pi$ as a left $A$-linear map. We can now consider the following composition
\begin{equation}
\label{eq:Gammadef}
\Gamma \colon \cU_\ahha(\mf{h}) \xrightarrow{\textsc{S}_{\mf{h}}^{-1}} \cS_\ahha (\mf{h}) \xrightarrow{\cS_\ahha (\gamma)} \cS_\ahha (\mf{g}) \xrightarrow{\textsc{S}_{\mf{g}}} \cU_\ahha(\mf{g}),
\end{equation}
where $\textsc{S}_{\mf{h}}$ and $\textsc{S}_{\mf{g}}$ denote the symmetrisation maps \eqref{eq:symmetr} for the Lie-Rinehart algebras $\mf{h}$ and $\mf{g}$, respectively.
The map $\Gamma$ above is a morphism of $A$-corings thanks to Theorem \ref{thm:isocoring} and to functoriality of $\cS_\ahha (-)$ (in fact, $\cS_\ahha (\gamma)$ is a morphism of left bialgebroids over $A$).

\begin{proposition}\label{prop:Gamma}
The $A$-coring morphism $\Gamma$ induced by $\gamma$ is a section of the left $A$-bialgebroid morphism $\Pi \coloneqq \cU_\ahha (\pi)$. Moreover, 
\[
\Gamma(a \bla u \bra b) = a \bla \Gamma(u) \bra b
\]
for all $a,b \in A$, $ u  \in \cU_\ahha(\mf{h})$.
\end{proposition}

\begin{proof}
Fix a dual basis $\{\chi_i,\varphi_i\mid i \in I\}$ of $\mf{g}$ and consider the elements $\theta_i \coloneqq \pi(\chi_i) \in \mf{h}$ and $\phi_i \coloneqq \varphi_i \circ \gamma \in \Hom{}{}{\ahha}{}{\mf{h}}{A}$ for all $i \in I$. For every $X \in \mf{h}$, one has 
\[
\textstyle\sum\limits_{i \in I}\phi_i(X)\theta_i = \sum\limits_{i \in I}\varphi_i\big(\gamma(X)\big)\pi\left(\chi_i\right) = \pi\pig(\sum\limits_{i \in I}\varphi_i\big(\gamma(X)\big)\chi_i\pig) = \pi(\gamma(X)) = X,
\]
whence $\{\theta_i,\phi_i\mid i \in I\}$ forms a dual basis for $\mf{h}$. To prove that $\Pi \circ \Gamma = \id_{\cU_\ahha(\mf{h})}$, we show that $\Pi \circ \Gamma \circ \textsc{S}_{\mf{h}} = \textsc{S}_{\mf{h}}$. Equivalently, we  are going to show that $\Pi \circ \textsc{S}_{\mf{g}} \circ \cS_\ahha (\gamma) = \textsc{S}_{\mf{h}}$. Since all involved maps are (left) $A$-linear, it is enough to check it on a homogeneous element of the form $X_1\cdots X_k$:
\begin{align*}
\big(\Pi \circ \textsc{S}_{\mf{g}} \circ \cS_\ahha  & (\gamma)\big)\left(X_1\cdots X_k\right) = \left(\Pi \circ \textsc{S}_{\mf{g}}\right)\left(\gamma\left(X_1\right)\cdots \gamma\left(X_k\right)\right)  \\
 & \stackrel{\scriptscriptstyle\eqref{eq:symmetr}}{=} \Pi\Big(\frac{1}{k!}\sum_{\substack{ j_1,\ldots,j_k \in I \\ \sigma\in\mf{S}_k} }\varphi_{j_1}\left(\gamma\left(X_1\right)\right)\cdots\varphi_{j_k}\left(\gamma\left(X_k\right)\right)\chi_{j_{\sigma(1)}}\cdots\chi_{j_{\sigma(k)}}\Big) \\
  & \stackrel{\phantom{\eqref{eq:symmetr}}}{=} \frac{1}{k!}\sum_{\substack{ j_1,\ldots,j_k \in I \\ \sigma\in\mf{S}_k} }\phi_{j_1}\left(X_1\right)\cdots\phi_{j_k}\left(X_k\right)\pi\left(\chi_{j_{\sigma(1)}}\right)\cdots\pi\left(\chi_{j_{\sigma(k)}}\right) \\
  & \stackrel{\phantom{\eqref{eq:symmetr}}}{=} \frac{1}{k!}\sum_{\substack{ j_1,\ldots,j_k \in I \\ \sigma\in\mf{S}_k} }\phi_{j_1}\left(X_1\right)\cdots\phi_{j_k}\left(X_k\right)\theta_{j_{\sigma(1)}}\cdots\theta_{j_{\sigma(k)}} \stackrel{\scriptscriptstyle\eqref{eq:symmetr}}{=} \textsc{S}_{\mf{h}} \left(X_1\cdots X_k\right).
\end{align*}
Concerning the last claim, since $a \bla u \bra b = uab$ as elements in $\cU_\ahha(\mf{h})$ for all $a,b \in A$, $u \in \cU_\ahha(\mf{h})$, it would be enough to check that $\Gamma(ua) = \Gamma(u)a$ for $a \in A$, $u \in \cU_\ahha(\mf{h})$. If we denote by $r_a$ the endomorphism (both of $\cU_\ahha(\mf{h})$ and $\cU_\ahha(\mf{g})$) given by right multiplication by $a \in A$, it suffices to check that $\Gamma \circ r_a \circ \textsc{S}_{\mf{h}} = r_a \circ \Gamma \circ \textsc{S}_{\mf{h}}$ or, equivalently, that $\Gamma \circ r_a \circ \textsc{S}_{\mf{h}} = r_a \circ \textsc{S}_{\mf{g}} \circ \cS_\ahha (\gamma)$ on an homogeneous element of the form $X_1\cdots X_k \in \cS_\ahha (\mf{h})$. For the sake of readability, we postpone a detailed proof of it to \S\ref{sssec:prop3.7}.
\end{proof}

Let us provide an example of why $\Gamma \circ r_a \circ \psi_{\mf{h}} = r_a \circ \psi_{\mf{g}} \circ \cS_\ahha (\gamma)$.

\begin{example}
With the same conventions as in the proof of Proposition \ref{prop:Gamma},  we have
{\small
\begin{align*}
& \Gamma(\textsc{S}_{\mf{h}}(XY)a) = \frac{1}{2}\Gamma\pig(\sum_{i,j}\phi_i(X)\phi_j(Y)(\theta_i\theta_j + \theta_j\theta_i)a\pig) \\
& \stackrel{\phantom{\eqref{eq:Gammadef}}}{=} \frac{1}{2}\Gamma\Big(\sum_{i,j}\phi_i(X)\phi_j(Y)\pig(a\big(\theta_i\theta_j + \theta_j\theta_i\big) + 2\omega_{\mf{h}}(\theta_i)(a)\theta_j + 2\omega_{\mf{h}}(\theta_j)(a)\theta_i + \big[\omega_\mathfrak{h}(\theta_i),\omega_\mathfrak{h}(\theta_j)\big](a)\pig)\Big) \\
& \stackrel{\phantom{\eqref{eq:Gammadef}}}{=} a\,\Gamma\big(\textsc{S}_{\mf{h}}(XY)\big) + \omega_{\mf{h}}(X)(a)\Gamma\big(\textsc{S}_{\mf{h}}(Y)\big) + \omega_{\mf{h}}(Y)(a)\Gamma\big(\textsc{S}_{\mf{h}}(X)\big) + \frac{1}{2}\sum_{i,j}\phi_i(X)\phi_j(Y)\big[\omega_{\mf{h}}(\theta_i),\omega_{\mf{h}}(\theta_j)\big](a) \\
& \stackrel{\scriptscriptstyle\eqref{eq:Gammadef}}{=} a\,\textsc{S}_{\mf{g}}\big(\gamma(X)\gamma(Y)\big) + \omega_{\mf{h}}(X)(a)\textsc{S}_{\mf{g}}\big(\gamma(Y)\big) + \omega_{\mf{h}}(Y)(a)\textsc{S}_{\mf{g}}\big(\gamma(X)\big) \, + \\
& \hspace{5cm} + \, \frac{1}{2}\sum_{i,j}\varphi_i\big(\gamma(X)\big)\varphi_j\big(\gamma(Y)\big)\big[\omega_{\mf{h}}\big(\pi(\chi_i)\big),\omega_{\mf{h}}\big(\pi(\chi_j)\big)\big](a) \\
& \stackrel{\phantom{\eqref{eq:Gammadef}}}{=} \frac{1}{2}\sum_{i,j}\varphi_i\big(\gamma(X)\big)\varphi_j\big(\gamma(Y)\big)\Big(a\big(\chi_i\chi_j + \chi_j\chi_i\big) + 2\omega_{\mf{g}}(\chi_i)(a)\chi_j + 2\omega_{\mf{g}}(\chi_j)(a)\chi_i + \big[\omega_{\mf{g}}(\chi_i),\omega_{\mf{g}}(\chi_j)\big](a)\Big) \\
& \stackrel{\phantom{\eqref{eq:Gammadef}}}{=} \frac{1}{2}\sum_{i,j}\varphi_i\big(\gamma(X)\big)\varphi_j\big(\gamma(Y)\big)\Big(\big(\chi_i\chi_j + \chi_j\chi_i\big)a\Big) = \textsc{S}_{\mf{g}}\left(\cS_\ahha (\gamma)(XY)\right)a.
\end{align*}}%
for all $X,Y \in \mf{h}$ and all $a \in A$.
\end{example}

By applying Theorem \ref{thm:sigmatwisted} and Proposition \ref{prop:neverendingprop} we may conclude the following.

\begin{theorem}
\label{thm:mainthmA}
Let $0 \to \mf{n} \xrightarrow{\iota} \mf{g} \xrightarrow{\pi} \mf{h} \to 0$ be a short exact sequence of Lie-Rinehart algebras, which are projective as left $A$-modules. Then we have an isomorphism 
\[
\cU_\ahha(\mf{g}) \simeq U_{\ahha}(\mf{n}) \hash{\sigma} \cU_\ahha(\mf{h})
\]
of $A$-rings and right $\cU_\ahha(\mf{h})$-comodule algebras,
where $\sigma$ is defined as in \eqref{chesorpresa}.
\end{theorem}

\begin{proof}
The sole detail that deserves to be highlighted is that Theorem \ref{thm:sigmatwisted} entails that $\cU_\ahha(\mf{g}) \simeq U_{\ahha}(\mf{n}) \hash{\sigma} \cU_\ahha(\mf{h})$ as $\Ae$-rings, where the structures of rings over $\Ae$ are given by
\begin{equation*}
\begin{array}{rclrcl}
A \ot A &\to& \cU_\ahha(\mf{g}), & a \ot b &\mapsto& \iota_\ahha (ab), 
\\ 
A \ot A &\to& U_\ahha(\mf{n}) \hash{\sigma} \cU_\ahha(\mf{h}), & a \ot b &\mapsto& \iota_\ahha (a) \otimes_\ahha  \iota_\ahha (b) = 1 \ot_\ahha \iota_\ahha (ab) = \iota_\ahha (ab)\ot_\ahha 1.
\end{array}
\end{equation*}
It is then clear that $\cU_\ahha(\mf{g}) \simeq U_{\ahha}(\mf{n}) \hash{\sigma} \cU_\ahha(\mf{h})$ as $A$-rings as well.
\end{proof}

From Proposition \ref{prop:semidirect} we directly deduce:

\begin{corollary}
\label{cor:mainthmB}
If $\mf{g} \simeq \mf{n} \niplus_\tau \mf{h}$ is a curved semi-direct sum of the $A$-Lie algebra $(A,\mf{n})$ and of the Lie-Rinehart algebra $(A,\mf{h},\omega)$, both projective as $A$-modules, then we have an isomorphism 
\[
\cU_\ahha (\mf{n} \niplus_\tau \mf{h}) \simeq U_{\ahha}(\mf{n}) \hash{\sigma} \cU_\ahha(\mf{h})
\]
of $A$-rings and right $\cU_\ahha(\mf{h})$-comodule algebras.
\end{corollary}

Summing up, any left $A$-linear splitting $\gamma \colon  \mf{h} \to \mf{g}$ of a short exact sequence $ 0 \to \mf{n} \xrightarrow{\iota} \mf{g} \xrightarrow{\pi} \mf{h} \to 0$ of Lie-Rinehart algebras over $A$ gives rise to an $A$-coring splitting $\Gamma_\gamma \coloneqq \textsc{S}_{\mf{g}} \circ \cS_\ahha (\gamma) \circ \textsc{S}_{\mf{h}}^{-1}$ of the surjection $\cU_\ahha (\pi) \colon \cU_\ahha(\mf{g}) \to \cU_\ahha(\mf{h})$ such that $\Gamma_\gamma$ is also right $A$-linear and $\Gamma_\gamma(1) = 1$.

\begin{lemma}\label{lem:seccorresp}
Let $\pi \colon \mf{g} \to \mf{h}$ be an epimorphism of Lie-Rinehart algebras over $A$ which are projective as left $A$-modules. Assume to have an $A$-coring splitting $\Gamma$ of the surjection $\cU_\ahha (\pi) \colon \cU_\ahha(\mf{g}) \to \cU_\ahha(\mf{h})$ such that $\Gamma(1) = 1$. Then $\Gamma$ induces a left $A$-linear splitting $\gamma_\Gamma$ of $\pi \colon \mf{g} \to \mf{h}$. Moreover, if $\Gamma = \Gamma_\gamma$ for some $\gamma \colon \mf{h} \to \mf{g}$, then $\gamma_\Gamma = \gamma$.
\end{lemma}

\begin{proof}
Notice that for every $X \in \mf{h}$ we have
\[
\Delta_{\mf{g}}(\Gamma(X)) = (\Gamma \otimes_\ahha  \Gamma)\big(\Delta_{\mf{h}}(X)\big) = \Gamma(X) \otimes_\ahha  1 + 1 \otimes_\ahha  \Gamma(X),
\]
which means that $\Gamma(X) \in \cP_\ahha(\cU_\ahha(\mf{g})) = \mf{g}$ is a primitive element. Therefore, there exists a unique $\gamma_\Gamma \colon \mf{h} \to \mf{g}$ such that
\[
\xymatrix{
\cU_\ahha(\mf{h}) \ar[r]^-{\Gamma} & \cU_\ahha(\mf{g}) \\
\mf{h} \ar[u]^-{\iota_{\mf{h}}} \ar[r]_-{\gamma_\Gamma} & \mf{g} \ar[u]_-{\iota_{\mf{g}}}
}
\]
commutes. Since
\[\iota_{\mf{g}}\big(\gamma_\Gamma(a \cdot X)\big) = \Gamma\big(\iota_{\mf{h}}(a \cdot X)\big) = \iota_\ahha (a)\cdot \Gamma\big(\iota_{\mf{h}}(X)\big) = \iota_\ahha (a) \cdot \iota_{\mf{g}}\big(\gamma(X)\big) = \iota_{\mf{g}}\big(a \cdot \gamma_\Gamma(X)\big)\]
for all $a \in A$, $X \in \mf{h}$, it follows that $\gamma_\Gamma$ is left $A$-linear. Furthermore,
\[\iota_{\mf{h}}\big(\pi(\gamma_\Gamma(X))\big) = \cU_\ahha (\pi)\big(\iota_{\mf{g}}(\gamma_\Gamma(X))\big) = \cU_\ahha (\pi)\big(\Gamma(\iota_{\mf{h}}(X))\big) = \iota_{\mf{h}}(X)\]
for all $X \in \mf{h}$, which implies that $\pi \circ \gamma_\Gamma = \id_{\mf{h}}$.
The final claim follows from observing that $\Gamma_\gamma(X) = \gamma(X)$ for all $X \in \mf{h}$.
\end{proof}

\begin{remark}
\label{rem:notbij}
  There is no reason to expect that the correspondence in Lemma \ref{lem:seccorresp} is a bijective correspondence, since the assignment $\gamma \mapsto \Gamma$ depends on an arbitrary choice of (what physicists loosely call) ``an ordering''. Consider, for instance, the following toy example over $A = \R$, the real numbers. Let $\mf{g}$ be the Heisenberg algebra over $\R$ defined as $\mf{g} = \R q \oplus \R p \oplus \R c$, endowed with the Lie bracket (the so-called ``canonical commutation relations'' in physics)
  \[
  [p,q] =c \qquad \text{and} \qquad [p,c] = 0 = [q,c],
  \]
  and let $\mf{h}$ be the two-dimensional abelian Lie algebra $\R x \oplus \R y$. The linear morphism
  \[
  \pi \colon \mf{g} \to \mf{h}, \qquad 
  \begin{cases} 
  q \mapsto x \\ p \mapsto y \\ c \mapsto 0
  \end{cases}
  \]
  is a surjective Lie algebra morphism with kernel the one-dimensional central Lie subalgebra $\mf{n}=\R c$. In other words, the Heisenberg algebra $\mf{g}$ is the central extension of the abelian two-dimensional algebra $\mf{h}$ by the one-dimensional abelian algebra $\mf{n}$. 
  In this case, 
  \[U(\mf{n}) \simeq \R[C], \qquad U(\mf{h}) \simeq \R[X,Y], \qquad  \text{and} \qquad U(\mf{g}) \simeq \R[Q,C][P;\delta],\] 
  the Ore extension (also known as skew polynomial ring) of the polynomial algebra $\R[Q,C]$ with respect to the derivation $\delta\in\Der\big(\R[Q,C]\big)$,
  uniquely determined by 
  \[
  \delta(Q) = C \qquad \text{and} \qquad \delta(C) = 0. 
  \]
  The Hopf algebra morphism $U(\pi) \colon U(\mf{g}) \to U(\mf{h})$ is uniquely determined by 
  \[\pi(Q) = X, \qquad \pi(P) = Y \qquad \text{and} \qquad \pi(C) = 0.\]
  Now, the $\R$-linear maps
  \[
  \Gamma \colon U(\mf{h}) \to U(\mf{g}), \qquad X^nY^m \mapsto Q^nP^m
  \]
  and
  \[
  \Gamma' \colon U(\mf{h}) \to U(\mf{g}), \qquad X^nY^m \mapsto P^mQ^n
  \]
  are both coalgebra sections of the projection $U(\pi)$ (the first one corresponds to the mapping in \cite[Ex.~4.20]{BlattnerCohenMontgomery} by considering the basis $\{x,y\}$ of $\mf{h}$ ordered by $x < y$ and the second one is simply obtained by considering it ordered by $y < x$). Both of them, when restricted to $\mf{h}$, give rise to the $\R$-linear section
  \[
  \gamma \colon \mf{h} \to \mf{g}, \qquad \begin{cases}x \mapsto q \\ y \mapsto p\end{cases}
  \]
  of $\pi$. However,
  \[
  \Gamma'(XY) = PQ = QP + C \neq QP = \Gamma(XY).
  \] 
  Note that the differential star product 
  \be\label{eq:star}
  f(Q,P)\star g(Q,P) = \pig(\exp\pig(C\frac{\partial}{\partial p}\frac{\partial}{\partial q}\pig)\big(f(Q,P+p)g(Q+q,P)\big)\pig)\pig\vert_{q=p=0}
  \ee
  on $\R[Q,P]$ endows the associated graded space $\gr\,U(\mf{g})=\R[Q,P,C]$ with the structure of an associative algebra isomorphic to $U(\mf{g})$, in formulas
  \[\big(U(\mf{n})\hash{\sigma}U(\mf{h}),\cdot\big) \cong \big(U(\mf{g}),\star\big),\]
  and where the image of $\Gamma$ identifies with the subspace $\R[Q,P]$. The counit is the evaluation at the origin and the coproduct can be realised via the formula $f(Q_1+Q_2,P_1+P_2)=f_{(1)}(Q_1,P_1)f_{(2)}(Q_2,P_2)$, where the isomorphism $\R[Q,P]\otimes\R[Q,P]\simeq\R[Q_1,P_1,Q_2,P_2]$ is understood. 
  This allows to compute the explicit expression for the Hopf cocycle $\sigma \colon U(\mf{h})\otimes U(\mf{h})\to U(\mf{n})$,
  \be\label{Hopf}
  \sigma(f,g)=(f\star g)\mid_{Q=0=P}, \quad \forall f,g\in\R[Q,P].
  \ee
 In practice, the only non-zero values are $\sigma(P^n,Q^n)=n!\,C^n$ for $n\in\mathbb{N}$.
\end{remark}

Recall that the universal enveloping algebra $\cU_\ahha(\mf{g})$ of a Lie-Rinehart algebra $(A,\mf{g},\omega)$ which is projective over $A$ is filtered with respect to the canonical filtration, where $\cU_\ahha(\mf{g})_{-1} \coloneqq 0$, $\cU_\ahha(\mf{g})_{0} \coloneqq A$, and for $p > 0$, $\cU_\ahha(\mf{g})_{p}$ is the left $A$-submodule of $\cU_\ahha(\mf{g})$ generated by $\iota_\mf{g}(\mf{g})^p$ (that is, products of the image of $\mf{g}$ in $\cU_\ahha(\mf{g})$ of length at most $p$).
Using this fact, and in spite of Remark \ref{rem:notbij}, one may prove the following refinement of Lemma \ref{lem:seccorresp}.

\begin{proposition}
\label{prop:seccorresp}
Let $\pi \colon \mf{g} \to \mf{h}$ be an epimorphism of Lie-Rinehart algebras over $A$ which are projective as left $A$-modules. On the set $\Sigma$ of unital $A$-coring splittings $\Gamma \colon \cU_\ahha(\mf{h}) \to \cU_\ahha(\mf{g})$ of the surjection $\cU_\ahha(\pi) \colon \cU_\ahha(\mf{g}) \to \cU_\ahha(\mf{h})$ which are filtered with respect to the canonical filtrations, define the equivalence relation
\[
\Gamma \sim \Gamma' \quad \Leftrightarrow \quad \gr(\Gamma) = \gr(\Gamma'),
\]
where $\gr(\Gamma) \colon \gr\,\cU_\ahha(\mf{h}) \to\gr\,\cU_\ahha(\mf{g})$ denotes the graded map associated to $\Gamma$.
Then there is a bijective correspondence between $\Sigma/{\sim}$ and the set of left $A$-linear splittings $\gamma \colon \mf{h}\to\mf{g}$ of $\pi \colon \mf{g} \to \mf{h}$.
\end{proposition}

\begin{proof}
In Lemma \ref{lem:seccorresp}, we constructed a (non-bijective) correspondence
\begin{equation}
  \begin{array}{rcl}
  \label{eq:bigcorresp}
  \Sigma & \mathrel{\substack{\xrightarrow{\phantom{\quad}} \\[-1pt] \xleftarrow{\phantom{\quad}}}} & \{\gamma \colon \mf{h} \to \mf{g} \mid \pi\circ\gamma = \id_{\mf{h}} \},
  \\[1mm]
  \Gamma & \longmapsto & \gamma_\Gamma,
  \\
  \Gamma_\gamma \coloneqq \textsc{S}_\mf{g} \circ \cS_\ahha(\gamma) \circ \textsc{S}_{\mf{h}}^{-1} & \longmapsfrom & \gamma,
\end{array}
  \end{equation}
because $\Gamma_\gamma$ is filtered since $\textsc{S}_{\mf{h}}$, $\textsc{S}_{\mf{g}}$, and $\cS_\ahha(\gamma)$ are all filtered as well with respect to the canonical filtrations.
Moreover, if $\gamma \colon \mf{h} \to \mf{g}$ is a section of $\pi$, then $\Gamma_\gamma \coloneqq \textsc{S}_{\mf{g}} \circ \cS_\ahha(\gamma) \circ \textsc{S}_{\mf{h}}^{-1}$ satisfies $\Gamma_\gamma \circ \iota_\mf{h} = \iota_{\mf{g}} \circ \gamma$, independently of the projective bases chosen for $\mf{g}$ and $\mf{h}$ to construct $\textsc{S}_{\mf{g}}$ and $\textsc{S}_{\mf{h}}$. Therefore, $\gamma_{\Gamma_\gamma} = \gamma$ (because $\iota_\mf{g}$ is injective by hypothesis).

As already observed in the proof of Lemma \ref{lem:symmetrisation}, $\gr(\textsc{S}_{\mf{g}})$ satisfies
\[
\cS_\ahha ^k(\mf{g}) \simeq \gr_k\big(\cS_\ahha (\mf{g})\big) \xrightarrow{\gr_k(\textsc{S}_{\mf{g}})} \gr_k\big(\cU_\ahha(\mf{g})\big), \qquad X_1\cdots X_k \mapsto X_1\cdots X_k + \cU_\ahha(\mf{g})_{\leq k-1}
\]
for all $k \geq 0$, and $0$ elsewhere. That is to say, $\gr(\textsc{S}_\mf{g})$ is the isomorphism of the PBW theorem, independently of the projective basis chosen for $\mf{g}$ to construct $\textsc{S}_{\mf{g}}$. Since for $\Gamma \in \Sigma$ the composition
\begin{equation}
\label{eq:gluglu}
\mf{h} = \cS_\ahha^1(\mf{h}) \simeq \gr_1\big(\cS_\ahha(\mf{h})\big) \xrightarrow{\gr_1(\textsc{S}_{\mf{h}})} \gr_1\big(\cU_\ahha(\mf{h})\big) \xrightarrow{\gr_1(\Gamma)} \gr_1\big(\cU_\ahha(\mf{g})\big) \xrightarrow{\gr_1(\textsc{S}_{\mf{g}})^{-1}} \gr_1\big(\cS_\ahha(\mf{g})\big) \simeq \cS^1_\ahha(\mf{g}) = \mf{g} 
\end{equation}
maps $X \in \mf{h}$ to $\Gamma(X) \in \mf{g}$, that is to say, it is exactly $\gamma_\Gamma$, it is clear that sections of $\cU_\ahha(\pi)$ with the same graded associated give rise to the same section of $\pi$. Hence, the correspondence \eqref{eq:bigcorresp} factors through the quotient $\Sigma/{\sim}$. Denoting by $[\Gamma]$ the equivalence class of $\Gamma$ in $\Sigma/{\sim}$, it is still true that $\gamma_{[\Gamma_\gamma]} = \gamma$.

Let now $\Gamma$ be a section in $\Sigma$. We want to show that $\gr(\Gamma) = \gr(\Gamma_{\gamma_\Gamma})$ or, equivalently, 
\[
\gr(\textsc{S}_{\mf{g}}^{-1}) \circ \gr(\Gamma) \circ \gr(\textsc{S}_{h}) = \gr\pig(\cS_\ahha(\gamma_\Gamma)\pig) = \cS_\ahha(\gamma_\Gamma)
\]
as maps of $A$-corings from $\cS_\ahha(\mf{h})$ to $\cS_\ahha(\mf{g})$. Set ${\cal S}'\coloneqq \textsc{S}_{\mf{g}}^{-1} \circ \Gamma \circ \textsc{S}_{h}$. Write $\cS_\ahha(\mf{g}) = A \oplus \bar{\cS}_\ahha(\mf{g})$ and $\cS_\ahha(\mf{h}) = A \oplus \bar{\cS}_\ahha(\mf{h})$ as left $A$-modules, where both $\bar{\cS}_\ahha(\mf{g})$ and $\bar{\cS}_\ahha(\mf{h})$ are the kernels of the respective counits, and endow $\bar{\cS}_\ahha(\mf{g})$ and $\bar{\cS}_\ahha(\mf{h})$ with the non-counital comultiplication $\bar{\Delta}(x) \coloneqq \Delta(x) - x \otimes_\ahha  1 - 1 \otimes_\ahha  x$ for $x\in\bar{\cS}_\ahha(\mf{h})$. Denote by $\varpi_{\mf{g}} \colon \bar{\cS}_\ahha(\mf{g}) \to \mf{g}$ the projection on the first component. Then both $\gr({\cal S}')$ and $\cS_\ahha(\gamma_\Gamma)$ restrict to morphisms of non-counital $A$-corings from $\bar{\cS}_\ahha(\mf{h})$ to $\bar{\cS}_\ahha(\mf{g})$ such that
\[
\varpi_\mf{g} \circ \gr({\cal S}') = \gr_1({\cal S}') \circ \varpi_{\mf{h}} \stackrel{\scriptscriptstyle (*)}{=} \gamma_\Gamma \circ \varpi_{\mf{h}} = \cS_\ahha^1(\gamma_\Gamma)\circ \varpi_{\mf{h}} = \varpi_{\mf{g}} \circ \cS_\ahha(\gamma_\Gamma),
\]
where in $(*)$ we used what we observed with \eqref{eq:gluglu}. Therefore, by the uniqueness part of \cite[Prop.~A.2]{MoerdijkLie}, $\gr({\cal S}') = \cS_\ahha(\gamma_\Gamma)$ and so $\left[\Gamma_{\gamma_{[\Gamma]}}\right] = [\Gamma]$, which completes the proof.
\end{proof}

In more prosaic words, Proposition \ref{prop:seccorresp} states that the correspondence between $A$-linear sections of $\pi \colon \mf{g} \to \mf{h}$ and filtered unital $A$-coring sections of $\cU_\ahha(\pi) \colon \cU_\ahha(\mf{g}) \to \cU_\ahha(\mf{h})$ is bijective up to considering the graded associated.

\subsection{Semi-direct sum Lie-Rinehart algebras and the smash product decomposition}

Let $(A,\mf{h},\omega)$ be a Lie-Rinehart algebra and $\mf{n}$ be an $A$-Lie algebra such that both $\mf{h}$ and $\mf{n}$ are projective as $A$-modules. Assume there is an action $\rho\colon \mf{h} \to \Der_{\,\K}(\mf{n})$ as in Definition \ref{def:LRaction} and set $\mf{g} \coloneqq \mf{n} \niplus \mf{h}$; equivalently, assume that the short exact sequence \eqref{eq:ses} splits as a sequence of Lie-Rinehart algebras, in the sense that $\pi$ admits a section $\gamma\colon\mf{h} \to \mf{g}$ which is a morphism of Lie-Rinehart algebras. By functoriality of the universal enveloping algebra construction, $\gamma$ induces a morphism of left Hopf algebroids $\Gamma\colon\cU_\ahha(\mf{h}) \to \cU_\ahha(\mf{g})$ and $\Pi \circ \Gamma = \cU_\ahha (\pi) \circ \cU_\ahha (\gamma) = \cU_\ahha (\pi \circ \gamma) = \id_{\cU_\ahha(\mf{h})}$. Hence, $\Gamma$ is a left Hopf algebroid section of $\Pi$, and, in particular, it is an $A$-coring section. In addition, being a morphism of $A$-rings as well, it satisfies
\[
\Gamma\left(1_{\cU_\ahha(\mf{h})}\right) = 1_{\cU_\ahha(\mf{g})} \qquad \text{and} \qquad \Gamma\left(a \bla u \bra b\right) = \Gamma\big(u\,\iota_\ahha (ab)\big) = \Gamma(u)\,\iota_\ahha (ab) = a \bla \Gamma(u) \bra b
\]
for all $a,b \in A$, $u \in \cU_\ahha(\mf{h})$. Therefore, we can apply Theorem \ref{thm:sigmatwisted} together with Proposition \ref{prop:neverendingprop} to claim that 
\[
U_{\ahha}(\mf{n}) \hash{\gs} \cU_\ahha(\mf{h}) \simeq \cU_\ahha(\mf{g})
\]
by means of $b \hash{} u \mapsto b\,\Gamma(u)$, where $\gs\colon\due {\, \cU_\ahha(\mf{h})} \blact {} \otimes_\Aopp {\cU_\ahha(\mf{h})\,}{}_\ract \to U_{\ahha}(\mf{n})$ is as in \eqref{chesorpresa}.

\begin{proposition}
If a short exact sequence of projective Lie-Rinehart algebras as in \eqref{eq:ses} splits as a sequence of Lie-Rinehart algebras, then the map $\gs$ of \eqref{chesorpresa} is trivial, that is, $\gs\left(u \otimes_\Aopp v \right) = \varepsilon\left(uv\right)$ for all $u,v \in \cU_\ahha(\mf{h})$. Moreover, the weak left $\cU_\ahha(\mf{h})$-action \eqref{weakadjoint} is a proper left action turning $U_{\ahha}(\mf{n})$ into a left $\cU_\ahha(\mf{h})$-module algebra.
\end{proposition}

\begin{proof}
Recall that
\[
\gs\colon \cU_\ahha(\mf{h}) \otimes_\Aopp \cU_\ahha(\mf{h}) \to U_{\ahha}(\mf{n}), \quad
u \otimes_\Aopp v \mapsto \gve \big(\Gamma(u_{(2)} v_{(2)})_+\big) \Gamma(u_{(1)})\Gamma(v_{(1)})\Gamma(u_{(2)} v_{(2)})_-.
\]
As $\Gamma$ is a morphism of left Hopf algebroids, it is multiplicative, comultiplicative, counital, and also a morphism of left Hopf structures, {\em i.e.}, Eqs.~\eqref{krach1} and \eqref{krach2} hold. In particular, if we set $x \coloneqq \Gamma(uv)$, then
\[
\gs\left(u \otimes_\Aopp v\right) = \varepsilon\left(x_{(2)+}\right)x_{(1)}x_{(2)-} \stackrel{\scriptscriptstyle\eqref{Sch4}}{=} \varepsilon\left(x_{+(2)}\right)x_{+(1)}x_{-} = x_{+}x_{-} \stackrel{\scriptscriptstyle\eqref{Sch7}}{=} \varepsilon(x) = \varepsilon(\Gamma(uv)) = \varepsilon\left(uv\right),
\]
where we suppressed the source map, that is, the canonical injection $A 
\to U_{\ahha}(\mf{n})$ in the latter terms.
Furthermore,
\begin{eqnarray*}
u \rightslice \left(v \rightslice b\right) \!\!\!& =\!\!\!& u \rightslice \big(\Gamma(v)_+ b\, \Gamma(v)_-\big) = \Gamma(u)_+\Gamma(v)_+ b\, \Gamma(v)_-\Gamma(u)_- \\
\!\!\!&
 \stackrel{\scriptscriptstyle\eqref{Sch6}}{=}\!\!\! & \Gamma(uv)_+b\,\Gamma(uv)_- = (uv) \rightslice b,
\end{eqnarray*}
and $1 \rightslice b = \Gamma(1)_+b\,\Gamma(1)_- = b$ for all $u,v \in \cU_\ahha(\mf{h})$ and $b \in U_{\ahha}(\mf{n})$, whence $\rightslice$ is a left $\cU_\ahha(\mf{h})$-action. 
To conclude, it is enough to observe that \eqref{vendeeglobe} entails that $\rightslice$ turns $U_{\ahha}(\mf{n})$ into a monoid in the monoidal category of left $\cU_\ahha(\mf{h})$-modules.
\end{proof}

It follows that $U_{\ahha}(\mf{n}) \hash{\gs} \cU_\ahha(\mf{h})$ coincides, in fact, with the smash product 
$U_{\ahha}(\mf{n}) \hash{} \cU_\ahha(\mf{h})$ 
in the sense of \S\ref{ssec:modalgs}.
Summing up, we proved the following.

\begin{theorem}
\label{thm:mainthmC}
If $\mf{g} \simeq \mf{n} \niplus \mf{n}$ is a semi-direct sum of the $A$-Lie algebra $\mf{n}$ and of the Lie-Rinehart algebra $(A,\mf{h},\omega)$, both projective over $A$, then we have an isomorphism 
\[
\cU_\ahha(\mf{g}) \simeq \cU_\ahha (\mf{n} \niplus \mf{h}) \simeq U_{\ahha}(\mf{n}) \hash{} \cU_\ahha(\mf{h})
\]
of $A$-rings and right $\cU_\ahha(\mf{h})$-comodule algebras.
\end{theorem}

 Recall that any Lie-Rinehart splitting $\gamma \colon \mf{h} \to \mf{g}$ of a short exact sequence $ 0 \to \mf{n} \xrightarrow{\iota} \mf{g} \xrightarrow{\pi} \mf{h} \to 0$ of Lie-Rinehart algebras over $A$ gives rise to a splitting $\Gamma \coloneqq \cU_\ahha (\gamma)$ of the surjection $\cU_\ahha (\pi) \colon \cU_\ahha(\mf{g}) \to \cU_\ahha(\mf{h})$ as left Hopf algebroid map. Thus, analogously to what we observed at the end of \S\ref{ssec:crossdec}, we have the following result.
 
\begin{proposition}\label{prop:1-1splittings}
Let $\pi\colon\mf{g} \to \mf{h}$ be an epimorphism of Lie-Rinehart algebras over $A$ which are projective as left $A$-modules. Then there is a bijective correspondence between splittings $\Gamma$ of the surjection $\cU_\ahha (\pi) \colon \cU_\ahha(\mf{g}) \to \cU_\ahha(\mf{h})$ as map of left Hopf algebroids and splittings $\gamma$ of $\pi\colon\mf{g} \to \mf{h}$ as Lie-Rinehart algebra map.
\end{proposition}

\begin{proof}
In view of \cite[Thm.~3.1]{MoerdijkLie}, the functor $\cU_\ahha\colon\mathsf{LieRin}_\ahha \to \mathsf{Bialgd}_\ahha$  induces an equivalence of categories between the full subcategory of Lie-Rinehart algebras over $A$ which are projective as left $A$-modules and the category of cocomplete graded projective $A$-bialgebroids, which entails that we have a bijection
\[
\mathsf{LieRin}_\ahha(\mf{h},\mf{g}) \xrightarrow{\simeq} \ms{Bialgd}_\ahha(\cU_\ahha(\mf{h}),\cU_\ahha(\mf{g})), \qquad \gamma \mapsto \cU_\ahha(\gamma).
\]
By this, it is clear that $\gamma$ is a section of $\pi$ if and only if $\cU_\ahha(\gamma)$ is a section of $\cU_\ahha(\pi)$.
\end{proof}

\subsection{Equivalences of crossed product decompositions}
In this subsection,
in the spirit of \cite[Thm.~2.8]{Mon:CPOHAAEA},
we would like to relate the
$\gs$-twisted crossed product construction from Definition \ref{GT4} to the universal enveloping algebra of the
curved semi-direct sum of Lie-Rinehart algebras from Definition \ref{def:curvedsdp}, or rather to a certain quotient of it. 

More precisely, let  $(A,\mf{h})$ be a Lie-Rinehart algebra and $\iota\colon A \to R$ be an $A$-algebra, seen as a Lie-Rinehart algebra with trivial anchor, denoted $(A,\Lie{R})$. Assume there exists an $A$-module map $\mf{h} \to \Der_{\,\K}(R), \ X \mapsto \nabla_X,$ and a $2$-cochain $\tau \in \mathrm{Hom}_\ahha (\textstyle\bigwedge^2_\ahha  \mf{h}, R)$ subject to $\nabla \tau = 0$, and such that Eq.~\eqref{prawda} for the curvature \eqref{beatport} is fulfilled. Then we may consider the curved semi-direct sum Lie-Rinehart algebra $\Lie{R} \niplus_\tau \mf{h}$ as in Definition \ref{def:curvedsdp} and the canonical inclusion $\jmath\colon\Lie{R} \to \Lie{R} \niplus_\tau \mf{h} , \ r \mapsto (r,0),$ induces an injective $A$-ring morphism $U_{\!\ahha}\left(\Lie{R}\right) \to \cU_\ahha\left(\Lie{R} \niplus_\tau \mf{h} \right)$, which we will treat, by slight abuse of notation, as an inclusion.
  
 \begin{definition}
 \label{def:GaloisDescent}
  Under the standing assumptions, we define
  \begin{equation}
  \label{allergodil}
R \times_\tau \cU_\ahha(\mf{h}) \coloneqq  \cU_\ahha(\Lie{R} \niplus_\tau \mf{h} )/\mathscr{I},
\end{equation}
the quotient of the universal enveloping algebra of the curved semi-direct sum $\Lie{R} \niplus_\tau \mf{h} $ 
by the two-sided ideal $\mathscr{I}$ generated by
\begin{equation}
  \label{brillantschwarz}
\mathscr{S} \coloneqq \{1_{\scriptscriptstyle{U_A (\Lie{R})}} - 1_{{\scriptscriptstyle{R}}}, \quad r \cdot_{\scriptscriptstyle{U_A (\Lie{R})}} r' - r \cdot_{\scriptscriptstyle{R}} r', \quad \forall \ r, r' \in R\}
\end{equation} 
in $\cU_\ahha  (\Lie{R} \niplus_\tau \mf{h} )$.
\end{definition}
The subsequent theorem generalises \cite[Thm.~2.8]{Mon:CPOHAAEA} to the realm of Lie-Rinehart algebras.

    \begin{theorem}
      \label{rain&fog}
      Let $\iota \colon A \to R$ be an $A$-algebra with associated $A$-Lie algebra $(A, \Lie{R})$
      and let $(A, \mf{h})$ be a Lie-Rinehart algebra such that both $R$ and $\mf{h}$ are projective as left $A$-modules.
      Then for an $A$-ring $S$
            the following are equivalent:
\begin{enumerate}[label=(\roman*),ref={\it(\roman*)}]
\item\label{item:moni}
$S \simeq R \times_\tau \cU_\ahha(\mf{h})$ in the sense of Definition \ref{def:GaloisDescent}.
\item\label{item:monii}
  $S \simeq R  \hash{\gs}  \cU_\ahha(\mf{h})$ in the sense of Definition \ref{GT4}.
  \end{enumerate}
  
    \end{theorem}

    \begin{proof}
      \ref{item:moni} $\Rightarrow$ \ref{item:monii}:
        in this implication, we shall strictly follow an argument of the corresponding part in the proof of \cite[Thm.~2.8]{Mon:CPOHAAEA}, which seems to go back to \cite{BorGabRen:PIEALA, McC:ROSLAATGKC}. 

      Assume $S \simeq  R \times_\tau \cU_\ahha(\mf{h})$ as in \eqref{allergodil}, which yields a short exact sequence
      $$
0 \to \Lie{R} \to \Lie{R} \niplus_\tau \mf{h} \to \mf{h} \to 0
      $$
of Lie-Rinehart algebras so that we are in the situation of Theorem \ref{thm:mainthmA}: one obtains the isomorphism
\[
\xymatrix{
\cU_\ahha( \Lie{R} \niplus_\tau \mf{h}) \ar@<+0.5ex>[r]^-{\Psi} & U_\ahha(\Lie{R}) \hash{\sigma} \cU_\ahha(\mf{h}) \ar@<+0.5ex>[l]^-{\Phi}
}
\]
of $A$-rings, where $\gs\colon \cU_\ahha(\mf{h}) \otimes_\Aop \cU_\ahha(\mf{h}) \to U_\ahha(\Lie{R})$ is a Hopf $2$-cocycle, $\Phi$ and $\Psi$ are explicitly given by \eqref{eq:Phi} and \eqref{eq:Psi} in Theorem \ref{thm:sigmatwisted}, respectively, and $U_\ahha(\Lie{R})$ constitutes the left Hopf kernel in the sense of \eqref{mandarino} with the property that $\Psi(r) = r \hash{\gs} 1$ for all $r \in  U_\ahha(\Lie{R})$.

Let us abbreviate $\LL  \coloneqq \Lie{R} \niplus_\tau \mf{h}$.
A closer look at \eqref{eq:Phi} reveals that both $\Phi$ and $\Psi$ are left $U_\ahha(\Lie{R})$-linear morphisms with respect to the ordinary left multiplication in $\cU_\ahha( \LL)$ and the regular left $U_\ahha(\Lie{R})$-module structure on $U_\ahha(\Lie{R}) \hash{\sigma} \cU_\ahha(\mf{h})$, respectively. In particular, we have an isomorphism $\cU_\ahha( \LL) \simeq U_\ahha(\Lie{R}) \otimes_\ahha  \cU_\ahha(\mf{h})$ as left $U_\ahha(\Lie{R})$-modules which entails that $\cU_\ahha( \LL)$ is a projective left $U_\ahha(\Lie{R})$-module since $\mf{h}$ and hence $\cU_\ahha(\mf{h})$ are projective left $A$-modules.

Next, as in \cite[Thm.~2.8]{Mon:CPOHAAEA}, let $\mathscr{I}_0$ be the two-sided ideal in $U_\ahha(\Lie{R})$ generated by the set $\mathscr{S}$ in \eqref{brillantschwarz}. One then clearly has $U_\ahha(\Lie{R})/\mathscr{I}_0 \simeq R$, whence $0 \to \mathscr{I}_0 \xrightarrow{i} U_\ahha(\Lie{R}) \xrightarrow{q} R \to 0$ is a short exact sequence of $U_\ahha(\Lie{R})$-bimodules. Moreover, by some diagram chasing and in view of the Snake Lemma one may check that the following diagram of left $U_\ahha(\Lie{R})$-modules has exact rows and is commutative
\[
\xymatrix{
0 \ar[r] & \mathscr{I}_0 \otimes_{\ahha} \cU_\ahha(\mf{h}) \ar[d]_{\simeq} \ar[r] & U_\ahha(\Lie{R}) \otimes_{\ahha} \cU_\ahha(\mf{h}) \ar[d]^-{\simeq} \ar[r] & R \otimes_{\ahha} \cU_\ahha(\mf{h}) \ar[d]^{\simeq} \ar[r] & 0 \\
0 \ar[r] & \mathscr{I}_0 \otimes_{\scriptscriptstyle U_\ahha(\Lie{R})} \cU_\ahha(\LL) \ar[d]_{\simeq} \ar[r] & U_\ahha(\Lie{R}) \otimes_{\scriptscriptstyle U_\ahha(\Lie{R})} \cU_\ahha(\LL) \ar[d]^-{\simeq} \ar[r] & R \otimes_{\scriptscriptstyle U_\ahha(\Lie{R})} \cU_\ahha(\LL) \ar[d]^{\simeq} \ar[r] & 0 \\
0 \ar[r] & \mathscr{I}_0\,\cU_\ahha(\LL) \ar[r] & \cU_\ahha(\LL) \ar[r] & \cU_\ahha(\LL)/\mathscr{I}_0\,\cU_\ahha(\LL) \ar[r] & 0
}
\]
where the upper vertical arrows are induced by $\Phi$ and the lower central vertical arrow is the canonical isomorphism given by the $U_\ahha(\Lie{R})$-module structure of $\cU_\ahha(\LL)$ ({\em i.e.},  the composition of the central vertical arrows is $\Phi$). That is,
\[
R \otimes_{\ahha} \cU_\ahha(\mf{h}) \simeq \cU_\ahha(\LL)/\mathscr{I}_0\,\cU_\ahha(\LL)
\]
as left $U_\ahha(\Lie{R})$-modules.
Below, we show that $\mathscr{I} = \mathscr{I}_0\,\cU_\ahha(\LL)$. Then $R \otimes_{\ahha} \cU_\ahha(\mf{h})$ inherits a $\K$-algebra structure as quotient of $U_\ahha(\Lie{R}) \hash{\sigma} \cU_\ahha(\mf{h})$ by the ideal $\mathscr{I}_0 \hash{\sigma} \cU_\ahha(\mf{h}) \coloneqq \mathscr{I}_0 \otimes_{\ahha} \cU_\ahha(\mf{h})$. In particular, $\mathscr{I}_0$ becomes closed under the measuring of $\cU_\ahha(\mf{h})$ on $U_\ahha(\Lie{R})$ and hence $R$ becomes a $\cU_\ahha(\mf{h})$-measured $A$-ring. It follows that $R \otimes_{\ahha} \cU_\ahha(\mf{h})$ becomes the $\bar\sigma$-twisted crossed product $R \hash{\,\bar\sigma} \cU_\ahha(\mf{h})$, where $\bar\gs\colon \cU_\ahha(\mf{h}) \otimes_\Aop \cU_\ahha(\mf{h}) \to R$ is $\gs$ under the quotient map, which is still a Hopf $2$-cocycle, and
\[
R \hash{\,\bar\sigma} \cU_\ahha(\mf{h}) \simeq \cU_\ahha(\LL)/\mathscr{I}.
\]
We are left with showing that
$\mathscr{I} = \mathscr{I}_0 \, \cU_\ahha ( \Lie{L})$,
which is done completely analogous to \cite[Thm.~2.8]{McC:ROSLAATGKC}.
One obviously has $\mathscr{I}_0 \subseteq \mathscr{I}$, hence
$\mathscr{I}_0 \, \cU_\ahha ( \Lie{L}) 
\subseteq \mathscr{I} \, \cU_\ahha ( \Lie{L})
= \mathscr{I}$, whereas, on the other side, for all $r,s,t \in R$ and $ X \in \mf{h}$, 
\begin{equation*}
\begin{split}
  (r,X)\big((s,0)(t,0) - (st, 0) \big)
  &= (s,0)(t,0)(r,X) + (s,0)([r,t],0) + (s,0)(\nabla_X t, 0)
  \\
  &
\quad + ([r,s],0)(t,0) + (\nabla_X s, 0)(t,0)
- (st,0)(r,X)
\\
&
\quad
-  (s[r,t],0)
- ([r,s]t,0)
- (\nabla_X (s)t, 0) - (s\nabla_X (t), 0)
\\
  &= \big((s,0)(t,0) - (st,0) \big) (r,X) + \mathscr{S},
\end{split}
\end{equation*}
using \eqref{baratti&milano} and the fact that $\nabla_X \in \Der_{\,\K}(R)$ by construction, and
where we denoted the product in
$\cU_\ahha ( \Lie{L})$ by juxtaposition.
Likewise, one sees that 
$
(r,X) (1_{\scriptscriptstyle{U_A (\Lie{R})}} - 1_{{\scriptscriptstyle{R}}})
=  (1_{\scriptscriptstyle{U_A (\Lie{R})}} - 1_{{\scriptscriptstyle{R}}}) (r,X). 
$
As a consequence,
$
\Lie{L} \, \mathscr{S} \subseteq
\mathscr{S} \, \Lie{L} + \mathscr{S},
$
and therefore
$\cU_\ahha ( \Lie{L})\mathscr{I}_0  \subseteq \mathscr{I}_0 \, \cU_\ahha ( \Lie{L})$, from which we obtain 
$$
\mathscr{I} = \cU_\ahha ( \Lie{L} ) \, \mathscr{S} \, \cU_\ahha ( \Lie{L})
\subseteq
\cU_\ahha ( \Lie{L}) \, \mathscr{I}_0 \, \cU_\ahha ( \Lie{L}) \subseteq
\mathscr{I}_0 \, \cU_\ahha ( \Lie{L})
$$
as $\mathscr{S} \subseteq \mathscr{I}_0$, which ends the proof of this implication.

   \ref{item:monii} $\Rightarrow$ \ref{item:moni}:
   if  $S \simeq R \hash{\gs} \cU_\ahha(\mf{h})$ with respect to an $R$-valued $2$-cocycle $\gs\colon \cU_\ahha(\mf{h}) \otimes_\Aop  \cU_\ahha(\mf{h}) \to R$, then $R$ is in particular a $\gs$-twisted left $\cU_\ahha(\mf{h})$-module in the sense of Definition \ref{GT2} by means of the measuring $\rightslice\colon  \cU_\ahha(\mf{h}) \otimes_\K R \to R$. Restricting this measuring to generators yields a map
   \[
   \nabla\colon \mf{h} \otimes_\K R \to R, \qquad X \otimes_\K r \mapsto \nabla_Xr \coloneqq X \rightslice r,
   \]
   such that $\nabla_X \in \Der_{\,\K}(R)$ for any $X \in \mf{h}$, as follows from \eqref{vendeeglobe} along with \eqref{mainsomma1}. Moreover,  
  it is easy to see that $\nabla$ defines a connection in the sense of Eqs.~\eqref{izvestiya} with curvature
   \begin{equation}
     \label{variazionisuuntemadibach}
\gO(X,Y) = [\tau(X,Y), \,\cdot\, ]_{\Lie{R}},
  \end{equation}
where
$$
\tau(X,Y) \coloneqq  \gs(X,Y) - \gs(Y,X)
$$
defines a (Lie-Rinehart) $2$-cocycle in the sense of Eq.~\eqref{fassbinder},
and
such that the curvature formula \eqref{prawda} is fulfilled: indeed, for all $a \in A$ and $X \in \mf{h}$, we have 
\begin{equation}
\begin{aligned}
X \rightslice \iota(a) & \stackrel{\phantom{\scriptscriptstyle\eqref{vendeeglobe}}}{=} \left(R \otimes_\ahha  \varepsilon\right)\big((1 \otimes_\ahha  X)(\iota(a)\otimes_\ahha  1)\big) = \left(R \otimes_\ahha  \varepsilon\right)\big((1 \otimes_\ahha  X)(1\otimes_\ahha  s(a))\big) \\
 & \stackrel{\scriptscriptstyle\eqref{vendeeglobe}}{=} \iota\big(\varepsilon(X \bract a)\big) \stackrel{\scriptscriptstyle\eqref{mainsomma1}}{=} X(a)
\end{aligned}\label{eq:mimancava}
\end{equation}
and, by using the $A$-bilinearity of the measuring as in Definition \ref{GT1}, along with \eqref{vendeeglobe} and \eqref{mainsomma1},
\begin{equation*}
  \begin{split}
    \nabla_{aX}r & \stackrel{\phantom{\scriptscriptstyle\eqref{mainsomma1}}}{=} (aX) \rightslice r = \iota(a)( X \rightslice r) = a\nabla_Xr,
    \\
    \nabla_X(ar) & \stackrel{\phantom{\scriptscriptstyle\eqref{mainsomma1}}}{=} X \rightslice \big( \iota(a)r\big) \stackrel{\scriptscriptstyle\eqref{vendeeglobe}}{=}  \big(X_{(1)} \rightslice  \iota(a)\big)   \big(X_{(2)} \rightslice r\big)
\\
   & \stackrel{\scriptscriptstyle\eqref{mainsomma1}}{=}  \big(X \rightslice  \iota(a)\big) r  + \iota(a) (X \rightslice r) \stackrel{\scriptscriptstyle\eqref{eq:mimancava}}{=} X(a)r + a \nabla_X r,
  \end{split}
\end{equation*}
which are Eqs.~\eqref{izvestiya} that define a left $(A, \mf{h})$-connection.
  Moreover, observe that
for $X, Y \in \mf{h}$
  from property \eqref{item:wa5} in Definition \ref{GT2} follows, 
  again with Eq.~\eqref{mainsomma1},
  $$
X \rightslice (Y \rightslice r) = \gs(X,Y)r - r\gs(X,Y) + (XY) \rightslice r= [\gs(X,Y),r]_{\Lie{R}} + (XY) \rightslice r, 
  $$
where we used $\gs(1,1) = 1_\erre$ as well. Hence, 
\begin{equation*}
  \begin{split}
 \nabla_X \nabla_Y r - \nabla_Y \nabla_X r - \nabla_{[X,Y]} r
&= 
X \rightslice (Y \rightslice r) - Y \rightslice (X \rightslice r)
- (XY) \rightslice r +  (YX) \rightslice r
\\
\quad 
&=
[\gs(X,Y),r]_{\Lie{R}} - [\gs(Y,X),r]_{\Lie{R}} = [\tau(X,Y),r]_{\Lie{R}}
\end{split}
  \end{equation*}
which proves \eqref{variazionisuuntemadibach}. Finally, that $\tau \in \Hom{}{}{\ahha}{}{\textstyle\bigwedge_\ahha^n \! \mf{h}}{N}$ is a $2$-cocycle, {\em i.e.}, an $A$-bilinear map that fulfils $\nabla \tau = 0$ is seen as follows:  taking once more Eq.~\eqref{mainsomma1} into account,  properties \eqref{item:wa1} and \eqref{item:wa2} in Definition \ref{GT2} guarantee the $A$-bilinearity of $\gs$ when restricted to elements in $\mf{h}$; hence the $A$-bilinearity of $\tau$. Furthermore, applying property \eqref{item:wa3} in Definition \ref{GT2} to three elements $X,Y, Z \in \mf{h}$, yields the familiar cocycle condition
$$
X \rightslice \gs(Y,Z) - \gs(XY, Z) + \gs(X, YZ) = 0,
$$
with the help of which and $\nabla_X r = X \rightslice r$ it is a longish but straightforward check that
\begin{equation*}
  \begin{split}
    \nabla\tau(X,Y,Z) &=
    \nabla_X \tau(Y,Z)
    +
    \nabla_Y \tau(Z,X)
    +
    \nabla_Z \tau(X,Y)
    \\
    &\quad
    -
    \tau([X,Y],Z)
-
\tau([Z,X],Y)
-
\tau([Y,Z],X)
 \end{split}
  \end{equation*}
  vanishes.
Hence, as in Proposition \ref{prop:curvedsdp}, we can build the Lie-Rinehart algebra $(A, \LL ) = (A, \Lie{R} \niplus_\tau \mf{h})$ and  its enveloping algebra $\cU_\ahha  (\LL )$.
Assuming $S= R \hash{\gs} \cU_\ahha(\mf{h})$,
define the following maps
\begin{equation*}
  \label{kindaweird}
  \arraycolsep=1pt\def\arraystretch{1.5}
\begin{array}{rclcrcl}
  \phi_{\LL }\colon \LL  & \to & \Lie{S},   & \qquad & (r, X)   &\mapsto& r \otimes_\ahha  1 + 1 \otimes_\ahha  X, \\
  \phi_{\ahha}\colon A       & \to & S,         & \qquad & a        &\mapsto& \iota(a) \otimes_\ahha  1.
  \end{array}
\end{equation*}
It is a straightforward check
that $\phi_{\LL}$ is a map of $\K$-Lie algebras and $\phi_\ahha$ is a map of $\K$-algebras, and that they fulfil Eqs.~\eqref{dumdidum}.
Hence, as in \S\ref{regenradar}, by the universal property they induce a morphism
$
\Phi\colon \cU_\ahha (\LL ) \to S
$
of $\K$-algebras and, in view of Remark \ref{enhancement}, even $A$-rings.
As obviously $\mathscr{I} \subseteq \ker \Phi$, this descends to a well-defined map
$$
\overline\Phi\colon \cU_\ahha (\LL )/\mathscr{I} \to S.
$$
On the other hand, define a map (of $R$-modules)
$
\Psi\colon S \to  \cU_\ahha (\LL )
$
by setting 
\begin{equation*}
r \otimes_\ahha  1 \mapsto (r,0),
  \qquad
  r \otimes_\ahha  X \mapsto (r,0)(0,X),
\end{equation*}
in degree zero and one, for $r \in R$ and $X \in \mf{h}$, and on general PBW generators recursively by
\begin{equation*}
  \label{orkney2}
r \otimes_\ahha  uX \mapsto \Psi(r \otimes_\ahha  u)\Psi(1 \otimes_\ahha  X) - \Psi\big(r\gs(u_{(1)},X) \otimes_\ahha  u_{(2)} \big)
\end{equation*}
for $u \in \cU_\ahha(\mf{h}), \, X \in \mf{h}, \, r \in R$. It is a simple check ({\em e.g.}, in degree zero) that $\Psi$ is not a morphism of $A$-rings, but it induces one if followed by the projection $\cU_\ahha(\LL) \twoheadrightarrow \cU_\ahha (\LL )/\mathscr{I}$, that is,
$$
\overline{\Psi} \colon S \to  \cU_\ahha (\LL )/\mathscr{I},
$$
is a morphism of $A$-rings by mere construction. That $\overline{\Psi}$ inverts $\overline{\Phi}$ follows by an equally simple induction argument: in degree $1$ this is straightforward, and as for the induction step, we have in degree $n+1$, considering that the coproduct is a degree zero map:
\begin{equation*}
  \begin{split}
    \left(\overline{\Phi} \circ \overline{\Psi}\right)(r \otimes_\ahha  uX)
    &=
    \left(\overline{\Phi} \circ \overline{\Psi}\right)
    (r \otimes_\ahha  u)
    \left(\overline{\Phi} \circ \overline{\Psi}\right)
    (1 \otimes_\ahha  X)
    - 
\left(\overline{\Phi} \circ \overline{\Psi}\right)
    \big(r\gs(u_{(1)},X) \otimes_\ahha  u_{(2)} \big)
    \\
     &=
(r \otimes_\ahha  u)(1 \otimes_\ahha  X)
    - 
r\gs(u_{(1)},X) \otimes_\ahha  u_{(2)}  
    \\
    &= r \otimes_\ahha  uX,
  \end{split}
\end{equation*}
hence $\overline{\Phi} \circ \overline{\Psi} = \id$, which along with the similarly straightforward check of
$\overline{\Psi} \circ \overline{\Phi} = \id$
concludes the proof.
    \end{proof}

\section{Geometric examples}
\label{geomex}

In this section, we present a few geometric instances of our main results in the framework of transformation Lie algebroids, of Atiyah algebroids, and of foliations.
Let us begin by discussing a straightforward algebraic analogue of transformation and Atiyah algebroids in the language of Lie-Rinehart algebras. 
Throughout the section, let us fix a commutative  $\K$-algebra $A$, a Lie algebra $\mf{h}$ and a representation of $\mathfrak{h}$ on $A$ by derivations, {\em i.e.}, a Lie algebra morphism
\begin{equation}\label{representation}
r \colon \mathfrak{h}\to\Der(A), \qquad v\mapsto v^\sharp.
\end{equation}

\subsection{Transformation Lie algebroids}

To start with, we provide the definition of a transformation Lie-Rinehart algebra by finding inspiration from the geometric notion of a transformation Lie algebroid.

\begin{definition}[Transformation Lie-Rinehart algebra]\label{def:transform}
Consider the Lie algebra $\mf{h}$ as a Lie-Rinehart algebra over $\K$ of the form $(\K,\mf{h},0)$. The free $A$-module $A\otimes\mathfrak{h}$ endowed with the structure of a Lie-Rinehart algebra over $A$ as in Proposition \ref{prop:iso3}\ref{item:iso3b}, that is, via the $A$-linear extension
\begin{equation*}
A\otimes\mathfrak{h}\to\Der(A), \qquad f\otimes v\mapsto f\,v^\sharp
\end{equation*}
of the representation $r$ in \eqref{representation}, as anchor,
and via
\begin{equation}\label{actionLRalg}
[f\otimes v,g\otimes w]\, \coloneqq\,(fg)\otimes[v,w]\,+\,\big(f\, v^\sharp(g)\big)\otimes w\,-\,\big(g\, w^\sharp(f)\big)\otimes v,\qquad \forall f,g\in A,\quad\forall v,w\in \mathfrak{h},    
\end{equation}
as Lie bracket, will be denoted $A\rtimes\mathfrak{h}$ and called a \emph{transformation Lie-Rinehart algebra}.
\end{definition}

\begin{example}[Transformation Lie algebroid]\label{ex:TransLieAlgd}
Consider the action of a Lie algebra $\mathfrak{h}$ on a smooth manifold $M$. The corresponding transformation Lie algebroid (see, for example, \cite[Prop.\ 4.1.2]{Mackenziebook-new}) will be denoted $M\rtimes\mathfrak{h}$. It is nothing but the trivial vector bundle $M\times\mathfrak{h}$ endowed with the structure of a Lie algebroid such that the space of its smooth global sections is the transformation Lie-Rinehart algebra $\Gamma(M\rtimes\mathfrak{h})=C^\infty(M)\rtimes\mathfrak{h}$.
\end{example}

As a side remark, let us mention the following corollary of Proposition \ref{prop:iso3} which, despite of its appearance, is not related to our main Theorem \ref{thm:sigmatwisted}.

\begin{proposition}
\label{propArtimesh}
The universal enveloping algebra of a transformation Lie-Rinehart algebra $A\rtimes\mathfrak{h}$ is isomorphic to the smash product
\be
\label{Artimesh}
\mathcal{U}_\ahha(A\rtimes\mathfrak{h}) \simeq A \hash{} U(\mathfrak{h})
\ee
of the commutative algebra $A$ with the universal enveloping algebra
of the Lie algebra $\mathfrak{h}$.
\end{proposition}

\begin{example}[Differential operators on a Lie group]
\label{ex:LieGroup}
The universal enveloping algebra of the Lie algebra $\mathfrak{g}$ of a Lie group $G$ can be defined geometrically as the associative algebra of invariant differential operators on $G$. More precisely, first recall that the tangent bundle of $G$ admits a trivialisation  $TG\simeq G\times\mathfrak{g}$ via the Maurer-Cartan one-form. As the Lie group regularly acts on itself via the (left or right) multiplication, the Lie algebra $\mathfrak{g}$ acts on $G$ and the above trivialisation induces the isomorphism of Lie-Rinehart algebras
\begin{equation}
\label{XGCG}
\mathfrak{X}(G)\simeq C^\infty(G)\rtimes \mathfrak{g},
\end{equation}
which encodes the fact that the Lie algebra of vector fields on $G$ that are invariant under (right, resp.\ left) translations is isomorphic to $\mathfrak{g}$, that is, $\mathfrak{X}(G)^G\simeq\mathfrak{g}$.
On the other hand, for any smooth manifold $M$, there is an isomorphism (see, for example,  \cite[p.~133]{Nistor})
\begin{equation}
\label{UEAXMDM}
\cU_{\scriptscriptstyle C^{\infty}(M)}\big(\mf{X}(M)\big)\simeq{\cal D}(M),
\end{equation} 
as follows from applying \cite[Thm.~3]{Nistor} to \cite[Ex.~1]{Nistor}, 
where $\mf{X}(M)$ denotes the Lie-Rinehart algebra of smooth vector fields and ${\cal D}(M)$ the algebra of differential operators on $M$. Applying this to $M = G$, one obtains
\begin{equation}\label{DGiso}
{\cal D}(G)\simeq \cU_{\scriptscriptstyle C^{\infty}(G)}\big(\mf{X}(G)\big)
\simeq C^\infty(G)\hash{} U(\mathfrak{g}),
\end{equation}
as follows from Proposition \ref{propArtimesh} putting $A=C^\infty(G)$ in \eqref{Artimesh}.
Therefore, we recover the well-known fact that the associative algebra of invariant differential operators on a Lie group is isomorphic to the universal enveloping algebra of its Lie algebra, that is, $\mathcal{D}(G)^G\simeq U(\mathfrak{g})$.
\end{example}

\subsection{Atiyah algebroid of a principal bundle}

Recall that the algebra $A$ is assumed to be commutative and an $\mf{h}$-module via the representation $r$ from \eqref{representation}.

\begin{definition}[Atiyah algebra] 
\label{def:AtiyahAlg}
Let $\mathfrak{h}^\sharp\subset \Der(A)$ denote the image of the representation $r$. The centraliser of
$\mathfrak{h}^\sharp$ inside the Lie algebra $\Der(A)$ will be denoted 
\begin{equation}\label{Ati}
\Der(A)^\mathfrak{h} \coloneqq \{X\in\Der(A) \mid [v^\sharp,X]=0,\,\forall v\in\mathfrak{h}\},
\end{equation}
and called the \emph{Atiyah algebra} of the $\mathfrak{h}$-module $A$.
\end{definition}

\begin{proposition}
Let 
\begin{equation}\label{Ah}
A^\mathfrak{h}\,\coloneqq\,\{f\in A \mid v^\sharp(f)=0,\,\forall v\in\mathfrak{h}\}    
\end{equation}
denote the commutative subalgebra of $A$ spanned by elements that are invariant under the action of the Lie algebra $\mathfrak{h}$.
The commutative subalgebra $A^\mathfrak{h}\subseteq A$ is a module over the Lie subalgebra $\Der(A)^\mathfrak{h}\subseteq\Der(A)$.
The Atiyah algebra is a Lie-Rinehart algebra over $A^\mathfrak{h}$ whose anchor
\begin{equation}\label{restrictionAh}
\rho\colon\Der(A)^\mathfrak{h}\to\Der(A^\mathfrak{h}),\quad X\mapsto X|_{A^\mathfrak{h}},
\end{equation}
is defined by the above restriction. 
\end{proposition}

\begin{proof}
Let $X\in\Der(A)^\mathfrak{h}$ be an $\mathfrak{h}$-invariant derivation on the commutative algebra $A$, that is, $X\in\Der(A)$ with $[v^\sharp,X]=0$ for any $v\in\mathfrak{h}$. 
Let $f\in A^\mathfrak{h}$, that is, $f\in A$ with $v^\sharp(f)=0$ for any $v\in\mathfrak{h}$.
One can explicitly see that $X(f)\in A^\mathfrak{h}$ since
\begin{equation}
v^\sharp\big(X(f)\big)=(v^\sharp X)(f)=(X v^\sharp)(f)=X\big(v^\sharp(f)\big)=0.
\end{equation}
It is then straightforward to check that the commutative subalgebra $A^\mathfrak{h}\subseteq A$ is indeed a module of the Lie subalgebra $\Der(A)^\mathfrak{h}\subseteq\Der(A)$.
One can also explicitly check that $f X\in\Der(A)^\mathfrak{h}$ since
\begin{equation}
[v^\sharp,fX]\,=\,v^\sharp(f)\,X\,+\,f\,[v^\sharp,X]=0.    
\end{equation}
The Atiyah algebra $\Der(A)^\mathfrak{h}$ is therefore an $A^\mathfrak{h}$-module and 
the $A^\mathfrak{h}$-linear map \eqref{restrictionAh} is defined as the restriction to the commutative subalgebra $A^\mathfrak{h}\subseteq A$.
Let $Y\in\Der(A)^\mathfrak{h}$ be another $\mathfrak{h}$-invariant derivation. 
The Leibniz rule holds since
\begin{equation}
[X,fY]\,=\,X(f)\,Y\,+\,f\,[X,Y]
\end{equation}
for any derivations $X, Y$, but $X(f)=\rho(X)(f)$ since $f\in A^\mathfrak{h}$. This proves that the Atiyah algebra $\Der(A)^\mathfrak{h}$ is a Lie-Rinehart algebra over $A^\mathfrak{h}$.
\end{proof}

\begin{example}[Atiyah algebroid of a principal bundle]
\label{principalcase}
Let $P$ be the total space of a principal $H$-bundle over $P/H$; equivalently, the Lie group $H$ acts freely (and properly) on the total space $P$ (from the right, traditionally) and the base $P/H$ is the space of orbits. 
We will assume from now on that the structure group $H$ is connected, so that if we are given a vector bundle $E$ over $P$ with an action of $H$ on $E$ projecting to the given action on $P$, then the equivariance of a global section of this vector bundle $E$ over $P$ under the action of the Lie group $H$ is equivalent to its (infinitesimal) invariance under the corresponding action of the Lie algebra $\mathfrak h$ (this is a direct corollary of \cite[Thm.\ 2.10]{Kosmann-Schwarzbach}). In particular, the commutative subalgebra of $A=C^\infty(P)$ spanned by the $H$-invariant functions on the total space is $A^\mathfrak{h}=C^\infty(P)^\mathfrak{h}$.
It is isomorphic to the commutative algebra of functions on the base space: $C^\infty(P)^\mathfrak{h}\simeq C^\infty(P/H)$. 
The Lie group $H$ acts freely on the tangent bundle $TP$ of the total space, thus one may consider the space of orbits $TP/H$, which is a Lie algebroid over $P/H$, called the \textit{Atiyah algebroid of the principal} $H$-\textit{bundle} $P$. Its global sections are the $H$-invariant vector fields on $P$, which span the Atiyah algebra $\Gamma(TP/H)=\mathfrak{X}(P)^\mathfrak{h}$.
\end{example}

\begin{definition}[Transitive Lie-Rinehart algebra]
A Lie-Rinehart algebra such that its anchor is surjective will be called a \emph{transitive Lie-Rinehart algebra}.
\end{definition}

This choice of terminology is motivated by the fact that the space of global sections of a transitive Lie algebroid (see, for example, \cite[p.~100]{Mackenziebook} or \cite[Def.~3.3.1]{Mackenziebook-new}) over a smooth manifold $M$ is a transitive Lie-Rinehart algebra over $C^\infty(M)$ in the above sense.

\begin{lemma}
Consider a transitive Lie-Rinehart algebra $\mathfrak{g}$ over $A$ with anchor $\rho$. There is a short exact sequence of Lie-Rinehart algebras
\be\label{Atiyahsequencebydef}
0 \to \ker\rho \hookrightarrow \mathfrak{g} \stackrel{\rho}\twoheadrightarrow \Der(A) \to 0,
\ee
where the kernel of the anchor is a Lie-Rinehart ideal (as in Example \ref{ex:LRideal}).
\end{lemma}

\begin{example}[Transitive Lie algebroid]
Remarkably, the kernel of the anchor of a transitive Lie algebroid is a Lie algebra bundle. This property is not obvious (see, for example, \cite[Thm.\  IV.1.4]{Mackenziebook} for an elegant proof). This vector subbundle is called the adjoint algebroid and its typical fibre (a Lie algebra) is called the isotropy algebra. If a transitive algebroid is integrable then it identifies with the Atiyah algebroid $TP/H$ of a principal bundle $P$ with structure group $H$ whose Lie algebra is the isotropy algebra $\mathfrak{h}$. 
\end{example} 

\begin{definition}
Let us set
\begin{equation}
\mathfrak{ann}(A^\mathfrak{h}) \coloneqq \{X\in\Der(A)^\mathfrak{h} \mid X(f)=0,\,\forall f\in A^\mathfrak{h}\},
\end{equation} 
the kernel of $\rho$ from \eqref{restrictionAh}, and call it the {\em annihilator} of the module $A^\mathfrak{h}$ over the Atiyah algebra $\Der(A)^\mathfrak{h}$.
\end{definition}

 With these notations, we always have an exact sequence of Lie-Rinehart algebras
\be
\label{Atiyahsequenceobvious0}
0 \to \mathfrak{ann}(A^\mathfrak{h}) \hookrightarrow \Der(A)^\mathfrak{h} \to \Der(A^\mathfrak{h}).
\ee
If the Atiyah algebra is transitive, then $\Der(A)^\mathfrak{h}$ is a Lie-Rinehart algebra extension of $\Der(A^\mathfrak{h})$ by $\mathfrak{ann}(A^\mathfrak{h})$, that is, there is a
short exact sequence of Lie-Rinehart algebras
\be
\label{Atiyahsequenceobvious}
0 \to \mathfrak{ann}(A^\mathfrak{h}) \hookrightarrow \Der(A)^\mathfrak{h} \twoheadrightarrow \Der(A^\mathfrak{h}) \to 0
\ee
in the sense of Equation \eqref{eq:ses}. 
 By Theorem \ref{thm:mainthmA}, any $A$-linear splitting (which we may call an \emph{invariant Ehresmann connection}, by finding inspiration from the geometric setting, see Example \ref{ex:atiyah} below) of this short exact sequence leads to a factorisation as a $\sigma$-twisted crossed product
\begin{equation}
\label{factorPH}
\cU_{A^\mf{h}}\big(\Der(A)^\mathfrak{h} \big) \simeq {U}_{A^\mf{h}}\big(\mathfrak{ann}(A^\mathfrak{h})\big) \hash{\gs} \cU_{A^\mf{h}}\big(\Der(A^\mathfrak{h})\big).
\end{equation}
In fact, up to the relation from Proposition \ref{prop:seccorresp}, invariant Ehresmann connections on $\Der(A)^\mathfrak{h}$ correspond to decompositions of $\cU_{A^\mf{h}}\big(\Der(A)^\mathfrak{h} \big)$ as in \eqref{factorPH}. Analogously, Lie-Rinehart algebra splittings (which we may call \emph{flat invariant Ehresmann connections}) of \eqref{Atiyahsequenceobvious} correspond to factorisations as smash product
\begin{equation}
\label{factorPH2}
\cU_{A^\mf{h}}\big(\Der(A)^\mathfrak{h} \big) \simeq {U}_{A^\mf{h}}\big(\mathfrak{ann}(A^\mathfrak{h})\big) \hash{} \cU_{A^\mf{h}}\big(\Der(A^\mathfrak{h})\big),
\end{equation}
by Proposition \ref{prop:1-1splittings}, that is, \eqref{factorPH} with respect to a trivial cocycle $\gs$.

\begin{example}[Atiyah sequence of a principal bundle] 
\label{ex:atiyah}
Continuing Example \ref{principalcase} and using the notation introduced in Example \ref{ex:TransLieAlgd}, the transformation Lie algebroid $P\rtimes\mathfrak{h}$ for the free action of $H$ on $P$ is endowed with a fibrewise injective anchor.
The vertical distribution $VP\subset TP$ is canonically isomorphic, as a Lie algebroid over the total space $P$, to the transformation Lie algebroid: $VP\simeq P\rtimes\mathfrak{h}$.
The Lie group $H$ also acts freely on the vertical distribution $VP$ of the total space; thus, one may consider the space of orbits $VP/H$. This gives rise to the short exact sequence 
\begin{equation}
\label{Asequ}
0 \to {VP}/{H} \hookrightarrow {TP}/{H} \twoheadrightarrow T\left({P}/{H}\right) \to 0   
\end{equation}
of Lie algebroids over $P/H$,
which is called the \emph{Atiyah sequence} of the principal $H$-bundle $P$ (see \cite[\S3.2]{Mackenziebook-new}). Let $\mathfrak{V}(P)$ denote the space $\Gamma(VP)$ of vertical vector fields on $P$ endowed with the structure of a Lie-Rinehart algebra over $C^\infty(P)$. The isomorphism $VP\simeq P\rtimes\mathfrak{h}$ of Lie algebroids over $P$ induces the isomorphism
\begin{equation}\label{VPiso}
\mathfrak{V}(P)\simeq C^\infty(P)\rtimes\mathfrak{h}    
\end{equation}
of Lie-Rinehart algebras over $C^\infty(P)$. 
The global sections of the vector bundle $VP/H$ are $H$-invariant vertical vector fields on $P$. The space $\Gamma(VP/H)$ of such global sections is, by construction, the annihilator $\mathfrak{ann}\big(C^\infty(P/H)\big)$ of the module $C^\infty(P)^{\mathfrak{h}}\simeq C^\infty(P/H)$ of the Atiyah algebra. 
The vector bundle $P\times_{Ad}\mathfrak{h}$ associated via the adjoint representation of $H$ on $\mathfrak{h}$ to the principal $H$-bundle $P$ is canonically isomorphic to $VP/H$ as a Lie algebroid over $P/H$. To summarise,
\begin{equation}\label{VP/H}
\mathfrak{ann}\big(C^\infty(P/H)\big)=\Gamma(VP/H)=\mathfrak{V}(P)^{\mf{h}}\simeq\Gamma(P\times_{Ad}\mathfrak{h})=\Gamma(P\rtimes\mathfrak{h})^\mathfrak{h}.
\end{equation}
In particular, \eqref{Atiyahsequenceobvious0} becomes exact as in \eqref{Atiyahsequenceobvious}. More explicitly, the Atiyah sequence \eqref{Asequ} implies the short exact sequence
\begin{equation}
\label{AsequLR}
0 \to \mathfrak{V}(P)^\mathfrak{h}\hookrightarrow \mathfrak{X}(P)^\mathfrak{h} \twoheadrightarrow \mathfrak{X}\left({P}/{H}\right) \to 0   
\end{equation}
of Lie-Rinehart algebras over $C^\infty(P/H)$.
In this setting, invariant Ehresmann connections (\emph{i.e.}, splittings of \eqref{Asequ} or, equivalently, of \eqref{AsequLR}) give rise to decompositions as crossed products
\[
\cU_{\scriptscriptstyle C^{\infty}(P/H)}\big(\mf{X}(P)^{\mf{h}}\big) \simeq U_{\scriptscriptstyle C^{\infty}(P/H)}\big(\mathfrak{V}(P)^{\mf{h}}\big) \hash{\gs} {\cal D}(P/H)
\]
for a certain cocycle $\sigma$ constructed as in Theorem \ref{thm:sigmatwisted}, where $\cD$ has the same meaning as in \eqref{UEAXMDM}. 
In the same way, flat invariant Ehresmann connections give rise to decompositions as smash products
(that is, with respect to a trivial cocycle $\gs$).
In other words, a (flat or curved) invariant Ehresmann connection on a principal bundle provides a factorisation of the associative algebra generated by the invariant vector fields on the total space as a (smash or crossed) product of the algebra generated by the vertical ones and the algebra of differential operators on the base manifold.
We will come back to this topic, in more detail, in \S\ref{ssec:Atiyah2}.
\end{example}

\subsection{Atiyah algebroid of a vector bundle}
\label{ssec:AtiyahVB}

The Atiyah algebroid of the frame bundle of a vector bundle is a principal bundle over the same base with the general linear group as structure group. It is often called the {\em Atiyah algebroid of the vector bundle}.

More precisely, let $E$ be a vector bundle of rank $n$ over a manifold $M$. A frame at a point $m$ of $M$ is a basis of the fibre at $m$. The frame bundle $FE$ of the vector bundle $E$ over $M$ is a principal bundle with structure group $H=GL(n)$, total space $P=FE$, and base $P/H=M$. The corresponding bundle $VP/H$ is sometimes denoted $\mathfrak{gl}(E)$. This is motivated by the isomorphism $VP/H\simeq P\times_{Ad}\mathfrak{h}$, which reads here $\mathfrak{gl}(E)\simeq FE\times_{Ad}\mathfrak{gl}(n)$. In fact, its fibres are isomorphic to the general linear algebra of the fibres of the vector bundle $E$.

The Atiyah algebroid of a vector bundle admits a celebrated equivalent (more algebraic) definition in terms of covariant derivatives, which we will review in what follows. The details about the equivalence between these two definitions are also reported, {\em cf.}\ Diagram \eqref{eq:AalgdVB}.

\begin{definition}[Covariant derivative]
Let $\textsc V$ be an $A$-module.
A pair $(X,\nabla)$ consisting of a derivation $X\in\Der(A)$ of the algebra $A$ and an endomorphism $\nabla$ of the underlying vector space of $\textsc V$ obeying the Leibniz rule
\be\label{forallsigma}
{\nabla}(f \sigma)={ X}(f)\,\sigma + f \,{\nabla}\sigma, \quad \forall f\in{A},\,\,\forall\sigma\in \textsc V,
\ee
can be interpreted as an \emph{infinitesimal automorphism} of the $A$-module $\textsc V$, and will be called
a \emph{covariant derivative} on the $A$-module $\textsc V$ along the derivation ${X}$.
\end{definition}

In \cite[\S1]{Kosmann-SchwarzbachMackenzie} and \cite[p.~72]{Huebschmann}, co\-variant de\-rivatives were called \emph{derivative endomorphisms} and \emph{infinitesimal gauge transformations}, respectively.

\begin{example}[Covariant derivative along a vector field] Consider a vector bundle $E$ over a manifold $M$ and $\textsc{V}=\Gamma(E)$ its vector space of sections. A covariant derivative on the ${\cal C}^\infty(M)$-module 
$\Gamma(E)$ along the derivation ${ X}\in{\mathfrak{X}}(M)$ is a ${\cal C}^\infty(M)$-linear map
\be\label{Kconn}
\nabla_{{X}} \colon \Gamma(E)\to \Gamma(E), \qquad \sigma\mapsto\nabla_{{X}}\sigma,
\ee
obeying the following Leibniz rule:
\be
\label{Leibnitzalong}
\nabla_{{X}} (f \sigma) = X(f)\,\sigma + f\,\nabla_{{X}} \sigma , \qquad\forall f\in{\cal C}^\infty(M), \quad\forall \sigma\in\Gamma(E).
\ee
It will be called a \textit{covariant derivative acting on the vector bundle} $E$
 along the vector field ${X}$. 
\end{example}
 
\begin{definition}[Atiyah algebra of a module]
The Lie algebra $\mathfrak{cder}_{\ahha}(\textsc{V})$ of all covariant derivatives ${\nabla}$ on the $A$-module $\textsc V$ is called the {\em Atiyah algebra of the $A$-module} $\textsc V$.
\end{definition}

\begin{proposition}\label{Atiyvectb}
The Atiyah algebra of the A-module $\textsc V$ is endowed with the structure of a Lie-Rinehart algebra over $A$ by means of the anchor 
\be
\label{surjanchorsigma}
\sigma \colon \mathfrak{cder}_{\ahha}(\textsc V)\to \Der({A}), \qquad (X,\nabla) \mapsto {X},
\ee
which maps a covariant derivative $(X,\nabla)$ along a derivation $X$ onto the latter derivation. 
\end{proposition}

The proof is a routine check. In fact, Proposition \ref{Atiyvectb} is well-known in the Lie algebroid case, see Example \ref{Atiyvectbundle} below.

\begin{definition}[General linear algebra of a module]

The associative algebra $End_\ahha (\textsc{V})$ of all $A$-linear morphisms on the $A$-module $\textsc V$, endowed with the composition as product, is an $A$-algebra.
The corresponding Lie algebra, denoted $\mathfrak{gl}_\ahha (\textsc{V})$, of all $A$-linear morphisms on the $A$-module $\textsc V$, endowed with the commutator as bracket, is an $A$-Lie algebra, which will be called the {\em general linear algebra of the $A$-module} $\textsc V$.
\end{definition}

It is clear that $A$-linear morphisms on the $A$-module $\textsc V$ coincide with covariant derivatives along the trivial derivation $X=0$. 

\begin{example}[Atiyah algebroid of a vector bundle]\label{Atiyvectbundle}
Let $A=C^\infty(M)$ and $\textsc{V}=\Gamma(E)$ be the module of sections of a vector bundle $E$ over a smooth manifold $M$. The Lie-Rinehart algebra $\mathfrak{cder}_{\ahha}(\textsc V)$ defines a transitive Lie algebroid (see, for example, \cite[Thm.\ 1.4]{Kosmann-SchwarzbachMackenzie}) called the \textit{Atiyah algebroid of the vector bundle} $E$. In this case, 
the Atiyah algebra $\mathfrak{cder}_{\ahha}(\textsc{V})$ of the $A$-module $\textsc V$ is the Lie-Rinehart algebra extension of the derivation algebra $\Der({A})$ by the general linear algebra $\mathfrak{gl}_\ahha (\textsc{V})$. In other words, there is a short exact sequence of Lie-Rinehart algebras
\be
\label{shortexactvectLieR}
0\to\mathfrak{gl}_\ahha (\textsc{V})\hookrightarrow\mathfrak{cder}_{\ahha}(\textsc{V})\twoheadrightarrow\Der({A})\to 0,
\ee
which is called the \textit{Atiyah sequence of the vector bundle} $E$.\qedhere
\end{example}

A left $(A,\mf{cder}_\ahha (\textsc{V}))$-connection, in the sense of \S\ref{ssec:LRconnection}, on the $A$-module $\textsc V$ is equivalent to a splitting of $A$-modules
\begin{equation}
\label{Koszul}
\Der({A})\hookrightarrow\mathfrak{cder}_{\ahha}({\textsc V}), \qquad X \mapsto \nabla_X    
\end{equation}
of the short exact sequence \eqref{shortexactvectLieR}.
The connection is flat if and only if the section \eqref{Koszul} is a morphism of Lie-Rinehart algebras.
In the geometric case of a vector bundle, this algebraic point of view reproduces the modern textbook definition (dating back to Koszul in 1950, {\em cf.}\ \cite{Ko}) of a linear connection.
 In view of Theorem \ref{thm:mainthmA}, Proposition \ref{prop:seccorresp}, and Proposition \ref{prop:1-1splittings}, the datum of a linear connection is essentially the same as the datum of a crossed product decomposition
\[
\cU_\ahha\big(\mathfrak{cder}_{\ahha}(\textsc V)\big) \simeq U_\ahha\big(\mathfrak{gl}_\ahha (\textsc{V})\big) \hash{\gs} \cU_\ahha\big(\Der(A)\big)
\]
for some cocycle $\gs$ (modulo the equivalence relation described in Proposition \ref{prop:seccorresp}), while the datum of a flat linear connection is equivalent to the datum of a smash product decomposition.

Furthermore, in view of Theorem \ref{rain&fog}, the quotient of $\cU_\ahha\big(\mathfrak{cder}_{\ahha}(\textsc V)\big)$ by the two-sided ideal $\mathscr{I}$, as in Definition \ref{def:GaloisDescent}, is isomorphic to the crossed product 
\begin{equation}
\mathcal{U}_\ahha\big(\mathfrak{cder}_{\ahha}(\textsc{V})\big)\,/\,\mathscr{I}\simeq End_\ahha(\textsc{V})
\hash{\sigma}
\mathcal{U}_\ahha\big(\Der(A)\big)
\end{equation}
by setting $R=End_\ahha (\textsc{V})$, $\mathfrak{R}=\mathfrak{gl}_\ahha (\textsc{V})$, and $\mathfrak{h}=\Der(A)$.

\begin{remark}
In general, the surjectivity of the anchor \eqref{surjanchorsigma} cannot be given for granted. If $\textsc V$ is an arbitrary $A$-module, then  $\mathfrak{cder}_{\ahha}(\textsc{V})$ coincides with the Lie-Rinehart algebra $\mathrm{DO}(A,\Der(A),\textsc{V})$
introduced in \cite[p.~72]{Huebschmann} (see also \cite[\S4.3]{Saracco2}). In the general case, one may always construct an exact sequence
\begin{equation}\label{eq:surjanchorsigma2}
0\to\mathfrak{gl}_{\ahha}(\textsc{V})\hookrightarrow\mathfrak{cder}_{\ahha}(\textsc{V})\to\Der({A})
\end{equation}
as in \cite[Eq.\ (2.11.8)]{Huebschmann}, and $\textsc V$ is called {\em $\Der(A)$-normal} if the right-most morphism is surjective. Hence $\textsc V$ admits an $(A,\Der(A))$-connection in the sense of \S\ref{ssec:LRconnection} if and only if it is $\Der(A)$-normal, and the short exact sequence splits as $A$-modules, see also \cite[Rem.\ 2.16]{Huebschmann}.
\end{remark}

\subsection{Vector bundle associated to a principal bundle}

In order to show the equivalence between the two constructions of the Atiyah algebroid of a vector bundle, let us discuss the construction of the vector bundle associated to a principal bundle.

As before, consider a commutative $\K$-algebra $A$ which is also an $\mathfrak{h}$-module via the representation $r$ from \eqref{representation}. 

\begin{definition}[Associated module] 
Let $W$ be an $\mathfrak{h}$-module.
The free $A$-module $A\otimes W$ is an $\mathfrak{h}$-module as well and the subspace $(A\otimes W)^\mathfrak{h}$ of $\mathfrak{h}$-invariant elements is an $A^\mathfrak{h}$-module. 
It will be called the {\em $A^\mathfrak{h}$-module associated to the $\mathfrak{h}$-module} $W$. 
\end{definition}

\begin{lemma}
There is an induced representation of the Atiyah algebra $\Der(A)^\mathfrak{h}$ on the associated module $(A\otimes W)^\mathfrak{h}$, that is, a morphism of Lie-Rinehart algebras over $A^\mathfrak{h}$ from
$\Der(A)^\mathfrak{h}$ to the Atiyah algebra $\mathfrak{cder}_{A^\mathfrak{h}}\big((A\otimes W)^\mathfrak{h}\big)$ of the associated module.
\end{lemma}

\begin{proof}
Let $v\in\mathfrak{h}$ be an element of the Lie algebra.
Its action on the free $A$-module $A\otimes W$ is defined by $\hat{v}(a\otimes w)  \coloneqq v^\sharp(a)\otimes w+a\otimes (v\triangleright w)$ for any $a\in A$ and $w\in W$. 

Let $X\in\Der(A)^\mathfrak{h}$, that is, a derivation of $A$ commuting with all elements of $\mathfrak{h}^\sharp$. Its action on the free $A$-module $A\otimes W$ is defined by $\hat{X}(a\otimes w)  \coloneqq X(a)\otimes w$ for any $a\in A$ and $w\in W$. 
One can check explicitly that $\hat{v}$ and $\hat{X}$ in $\mathfrak{gl}(A\otimes W)$ commute for any $v\in\mathfrak{h}$ and any $X\in\Der(A)^\mathfrak{h}$.
Therefore, $X$ preserves the associated module, that is, $\hat{X}$ sends $\mathfrak{h}$-invariant elements to $\mathfrak{h}$-invariant elements, hence
$\hat{X}\in\mathfrak{gl}\big((A\otimes W)^\mathfrak{h}\big)$.
What is more, $\hat{X}\in\mathfrak{cder}_{A^\mathfrak{h}}\big((A\otimes W)^\mathfrak{h}\big)$
since
$$
\hat{X}\big( f\cdot(a\otimes w) \big)
=\hat{X}\big( (f\cdot a)\otimes w \big)
=X(f\cdot a)\otimes w
=\big( X(f)\cdot a +f\cdot X(a) \big)\otimes w
=X(f)\cdot(a\otimes w)+f\cdot \hat{X}(a\otimes w)
$$ 
for any $f,a\in A$ and $w\in W$.
\end{proof}

There is a morphism between the corresponding exact sequences of Lie-Rinehart algebras \eqref{Atiyahsequenceobvious0} and \eqref{eq:surjanchorsigma2}, in the sense that all arrows in the commutative diagram
\begin{equation}
\label{eq:sesss}
\begin{gathered}
\xymatrix @R=15pt{
0 \ar[r] & \mathfrak{ann}(A^\mathfrak{h}) \ar[d] \ar[r] & \Der(A)^\mathfrak{h} \ar[d] \ar[r] & \Der(A^\mathfrak{h}) \ar@{=}[d] \\
0 \ar[r] & \mathfrak{gl}_{A^\mathfrak{h}}\big((A\otimes W)^\mathfrak{h}\big) \ar[r] & \mathfrak{cder}_{A^\mathfrak{h}}\big((A\otimes W)^\mathfrak{h}\big) \ar[r] & \Der(A^\mathfrak{h})
}
\end{gathered}
\end{equation}
are Lie-Rinehart algebra morphisms.

\begin{example}[Principal bundle] 
In case of a principal $H$-bundle $P$ together with a representation of $H$ on a vector space $W$, one can introduce the associated vector bundle $P\times_H W$. This geometric example corresponds to the case $A=C^\infty(P)$ and $(A\otimes W)^\mathfrak{h}=\Gamma(P\times_H W)$. In this setting, recall from Example \ref{ex:atiyah} that the first line of \eqref{eq:sesss} turns out to be a short exact sequence. The above construction reproduces the canonical representation of the Atiyah algebroid $TP/H$ of $P$ on its associated vector bundle $P\times_H W$.
An invariant Ehresmann connection on $P$ is equivalent to a linear splitting of the Atiyah sequence \eqref{Asequ}. Furthermore, there is a morphism from the  Atiyah sequence \eqref{AsequLR} of the principal bundle to the Atiyah sequence \eqref{shortexactvectLieR} of the associated vector bundle.
Accordingly, as is well-known, an invariant Ehresmann connection on the principal bundle induces a linear connection on the associated vector bundle via \eqref{eq:sesss}.
\end{example}

\begin{example}[Vector bundle] Consider a vector bundle $E$ over $M$ with typical fibre $V$.
The general linear algebra of the $C^\infty(M)$-module $\Gamma(E)$ is the space
$$\mathfrak{gl}_{C^\infty(M)}\big(\Gamma(E)\big)=\Gamma\big(\mathfrak{gl}(E)\big)$$ of sections of a Lie algebra bundle $\mathfrak{gl}(E)$ over $M$ with typical fibre $\mathfrak{gl}(V)$. This Lie algebra bundle is isomorphic to the associated bundle to the frame bundle $FE$ for the adjoint representation of the general linear group $GL(V)$ on its Lie algebra $\mathfrak{gl}(V)$, that is, $\mathfrak{gl}(E)\simeq FE\times_{Ad} \mathfrak{gl}(V)$ as at the beginning of this subsection.
The Atiyah algebroid $TFE/GL(V)$ of the frame bundle $FE$ is isomorphic to the Atiyah algebroid of $E$, so these two formulations are equivalent. In fact, the induced representation as above provides an isomorphism between the corresponding Atiyah sequences, because the first terms are isomorphic (since $FE\times_{Ad} \mathfrak{gl}(V)\simeq\mathfrak{gl}(E)$) and the last terms coincide. As a diagram,
\begin{equation}\label{eq:AalgdVB}
\begin{gathered}
\xymatrix @R=15pt{
0 \ar[r] & \Gamma\left(FE\times_{Ad} \mathfrak{gl}(V)\right) \ar[d]_-{\simeq} \ar[r] & \Gamma\left(\frac{TFE}{GL(V)}\right) \ar[d] \ar[r] & \mathfrak{X}(M) \ar[r] \ar@{=}[d] & 0 \\
0 \ar[r] & \mathfrak{gl}_{C^\infty(M)}\big(\Gamma(E)\big) \ar[r] & \mathfrak{cder}_{C^\infty(M)}\big(\Gamma(E)\big) \ar[r] & \mathfrak{X}(M) \ar[r] & 0 
}
\end{gathered}
\end{equation}
as claimed.
\end{example}

\subsection{Foliations}
\label{foli}

An involutive distribution $D$ on $M$ can be thought as a Lie algebroid $D$ over $M$ with injective anchor $\rho\colon D\hookrightarrow TM$. More precisely, the image $\rho(D)\subseteq TM$ of such an injective anchor is an involutive distribution on $M$ (see, for example, \cite[Ex.\ 3.3.5]{Mackenziebook-new}). The generalisation by Sussmann \cite{Sussmann} of the Frobenius Theorem to involutive distributions of locally finite type ensures that this distribution $\rho(D)$ is integrable, {\em i.e.}, it arises from a foliation.
More precisely, there is a one-to-one correspondence between:
\begin{enumerate}
    \item\label{foliation} a foliation of $M$,
    \item\label{invdistr} an involutive distribution on $M$,
    \item\label{Liesubalg} a Lie subalgebroid over $M$ of the tangent bundle $TM$,
    \item\label{LRalginjanch} a finitely generated projective Lie-Rinehart algebra over $C^\infty(M)$ with injective anchor,
    \item\label{subBialg} a graded-projective sub-bialgebroid with a finitely generated space of primitives over $C^\infty(M)$
    of the cocommutative bialgebroid ${\cal D}(M)$ spanned by differential operators on $M$.
\end{enumerate}
Note that other characterisations of foliations exist, {\em e.g.}, in terms of Lie groupoids \cite{Debord, MoerdijkCrainic}.
Let us also mention that, as any vector bundle, the distributions and Lie algebroids above are implicitly assumed to be locally of finite rank. However, they need not be of constant rank. Accordingly, the foliation in $(\ref{foliation})$ is allowed to be mildly singular (see \cite[\S1]{Debord} for a precise statement).

The equivalences $(\ref{foliation})\Leftrightarrow(\ref{invdistr})\Leftrightarrow(\ref{Liesubalg})\Leftrightarrow(\ref{LRalginjanch})$ are immediate. Indeed, the first ones were explained above (see also \cite[Ex.~1.3 \& 1.6]{Debord}) and the last one $(\ref{Liesubalg})\Leftrightarrow(\ref{LRalginjanch})$ follows from the equivalence between the category of Lie algebroids $E$ over $M$ and the category of finitely generated projective Lie-Rinehart algebras $\Gamma(E)$ over $C^\infty(M)$. 
Finally, the equivalence $(\ref{LRalginjanch})\Leftrightarrow (\ref{subBialg})$
is ensured by the equivalence of the categories of projective Lie-Rinehart algebras and of cocomplete graded-projective left bialgebroids, 
{\em cf.}\  the Cartier-Milnor-Moore type theorem of \cite[Thm.~4.1]{MoerdijkLie}.

\vspace{1mm}
The sub-bialgebroid over $C^\infty(M)$  in $(\ref{subBialg})$, corresponding to an involutive distribution $D\subset TM$, will be denoted \begin{equation}
\label{TD}
{\cal T}(D)=\cU_{C^\infty(M)}\pig(\Gamma\big(\rho(D)\big)\pig).
\end{equation}
The elements of the latter algebra are differential operators on $M$ which will be said {\em tangential to the leaves} (in particular, the sections of $D\subset TM$ are by definition vector fields on $M$ everywhere tangent to the leaves of the foliation). 
For any involutive distribution $D\subset TM$, there is a short exact sequence of vector bundles over $M$,
\[
0 \to D \hookrightarrow TM \twoheadrightarrow N \to 0,
\]
where $N$ is the quotient bundle made of the normal bundles of the leaves. Let $SN=\oplus_{k=0}^\infty S^kN$ stand for the bundle of symmetric tensor products of the normal bundle.
Applying Proposition \ref{prop:anothersplitting} in the present setting gives that
\[
{\cal D}(M)\simeq {\cal T}(D)\otimes_{C^\infty(M)} \Gamma(SN),
\]
as a module over the associative algebra ${\cal T}(D)$.

\begin{example}[Vertical foliation of a fibre bundle]
\label{verticalfoliation}
For a fibre bundle $\pi\colon E\twoheadrightarrow M$, the base manifold $M$ can be identified with the space of vertical leaves ({\em i.e.}, the fibres).
The Lie algebroid over the total space $E$ associated to the vertical foliation is the vertical distribution $VE=\ker\pi_*$ on $E$ with the embedding $i\colon VE\hookrightarrow TE$ as anchor.
Note that the vertical distribution can also be seen as a Lie algebra bundle over the base $M$, {\em i.e.}, a Lie algebroid over $M$ with trivial anchor $\pi_*\circ i\colon VE\to TM$.
The Lie-Rinehart algebra $\mathfrak{V}(E)\subset\mathfrak{X}(E)$ over $C^\infty(E)$ is spanned by vector fields on the total space $E$ that are vertical, {\em i.e.}, tangential to the fibres. The sub-bialgebroid over $C^\infty(E)$ corresponding to the vertical distribution will be denoted
\begin{equation}
\label{VD}
{\cal V}(E)=\cU_{C^\infty(E)}\big(\mathfrak{V}(E)\big).
\end{equation}
The pullback $\pi^*\colon C^\infty(M)\hookrightarrow C^\infty(E)$ provides an embedding of the base algebra inside the total algebra, hence the algebra ${\cal D}(E)$ of differential operators on $E$ is a $C^\infty(M)$-ring.
A vector field on $E$, seen as a (first-order) differential operator on $E$, is vertical if and only if it is $C^\infty(M)$-linear. Consequently, the algebra ${\cal V}(E)$ of differential operators tangential to the fibres is a $C^\infty(M)$-algebra.
The pullback vector bundle $\pi^*TM$ fits into the following short exact sequence
\[
0 \to VE \hookrightarrow TE \twoheadrightarrow \pi^*TM \to 0
\]
of vector bundles over $E$, giving rise to the following exact sequence 
\[
0 \to \mf{V}(E) \hookrightarrow \mf{X}(E) \twoheadrightarrow C^{\infty}(E) \otimes_{\scriptscriptstyle C^{\infty}(M)} \mf{X}(M) \to 0
\]
of finitely generated projective $C^{\infty}(E)$-modules. Therefore, in view of Proposition \ref{prop:anothersplitting} and the isomorphism 
\eqref{UEAXMDM}, the algebra $\cD(E)$ of differential operators on the total space $E$ satisfies 
\[
\cD(E) \simeq {\cal V}(E) \otimes_{\scriptscriptstyle C^{\infty}(M)} \Gamma(S\,TM)
\] 
as left ${\cal V}(E)$-modules. 
\end{example}

\begin{example}[Vertical foliation of a principal bundle] Recall from Example \ref{ex:atiyah} that the vertical distribution $VP$ of a principal $H$-bundle with total space $P$ is canonically isomorphic, as a Lie algebroid over $P$, to the transformation Lie algebroid $P\rtimes\mathfrak{h}$. The corresponding isomorphism \eqref{VPiso} together with Proposition \ref{propArtimesh} and Example \ref{verticalfoliation} imply that 
$$
{\cal V}(P)\simeq  C^\infty(P)\hash{} U(\mathfrak{h}).
$$
In other words, the sub-bialgebroid ${\cal V}(P)\subset{\cal D}(P)$ spanned by the  differential operators on $P$ tangential to the fibres is isomorphic to the smash product of the commutative algebra $C^\infty(P)$ with the universal enveloping algebra of $\mathfrak{h}$.
\end{example}

\subsection{More on Atiyah algebroids}
\label{ssec:Atiyah2}\label{invcom}
The Atiyah algebra $\Gamma\left(TP/H\right)\simeq\mathfrak{X}(P)^\mathfrak{h}$ of a principal $H$-bundle $P$ is spanned by the $H$-invariant vector fields on $P$. 
A natural question would be when its universal enveloping algebra is spanned by the $H$-invariant differential operators on $P$, too, that is, when
\begin{equation}\label{UAtiyah}
\cU_{C^\infty(P/H)}\big(\mathfrak{X}(P)^\mathfrak{h}\big)\stackrel{?}{\simeq}{\cal D}(P)^\mathfrak{h}\simeq\cU_{C^\infty(P)}\big(\mathfrak{X}(P)\big)^\mathfrak{h},
\end{equation}
or, more generally, when
\begin{equation}\label{wildguess}
{\cal U}_{A^\mathfrak{h}}\big(\Der(A)^\mathfrak{h}\big)\,\stackrel{?}{\simeq}\,
{\cal U}_{\ahha}\big(\Der(A)\big)^\mathfrak{h}
\end{equation}
for all $A$ and $\mf{h}$. A favourable case is given by a Klein geometry.

\begin{example}[Klein geometry]\label{Kleinex} 
Consider a Klein geometry, {\em i.e.}, a pair made of a Lie group $G$ and a closed Lie subgroup $H\subset G$ such that the coset space $G/H$ is connected. It defines a principal $H$-bundle $G$ over $G/H$.
In this case, the commutative algebra $A=C^\infty(G)$ is the algebra of functions on the Lie group $G$ and the representation \eqref{representation} is given (up to isomorphism) by the inclusion of $\mf{h}$ into $\mathfrak{X}(G)$ as left $G$-invariant vector fields. This representation $r$ is the infinitesimal counterpart of the right action of $H$ on $G$ via multiplication on the right. A function on the space $G/H$ of left cosets $gH$ is equivalent to a function $f$ on $G$ which is right $H$-invariant (that is, $f(gh)=f(g)$ for all $g\in G$ and $h\in H$).
For simplicity ({\em cf.}\ the remark in Example \ref{principalcase}), the Lie group $H$ is assumed to be connected, thus $A^\mathfrak{h}\simeq C^\infty(G/H)$ for the commutative algebra defined in \eqref{Ah}.
The Lie-Rinehart algebra $\Der(A)$ of derivations on $A$ is here the algebra $\mathfrak{X}(G)\simeq C^\infty(G)\rtimes \mathfrak{g}$ of vector fields on the Lie group $G$, see \eqref{XGCG},  where $\mf{g}$ is (up to isomorphism) the Lie algebra of (left or right) invariant vector fields on $G$. In particular, $\mf{der}(A)$ is the free left $A$-module generated (up to isomorphism) by (left or right) invariant vector fields. While $\mf{g}$ is usually realised as the Lie algebra of left invariant vector fields on $G$, for our purposes it is more convenient to use right invariant ones, which are the infinitesimal generators for the left action of $G$ on itself via left multiplication. This induces a left $G$-action on $G/H$; accordingly, there is a non-trivial representation of $\mathfrak{g}$ on $C^\infty(G/H)$, and the corresponding transformation Lie-Rinehart algebra will be denoted $C^\infty(G/H)\rtimes \mathfrak{g}$. The Atiyah algebra $\Der(A)^\mathfrak{h}$ of right $H$-invariant vector fields on $G$ is isomorphic to the transformation Lie-Rinehart algebra of the $\mathfrak h$-module $C^\infty(G/H)$: 
\begin{equation}
\label{CGH}
\mathfrak{X}(G)^\mathfrak{h}\simeq C^\infty(G/H)\rtimes \mathfrak{g},
\end{equation}
as follows from \eqref{XGCG} and 
the fact just mentioned that $\mathfrak{g}$ is spanned by right $G$-invariant vector fields on $G$.  In fact, consider a vector field $X=\sum_i a_i X_i \in \mf{X}(G)$ with $a_i\in C^\infty(G)$ and $X_i \in \mf{g}$ for all $i$. Note that $X \in \mf{X}(G)^{\mf{h}}$ if and only if $v^\sharp(a_i)=0$ for all $v\in\mathfrak{h}$ since
\[
0 = \textstyle\left[v^\sharp,\sum_i a_i X_i\right] = \sum_i a_i \left[v^\sharp,X_i\right] + \sum_i v^\sharp\left(a_i\right) X_i = \sum_i v^\sharp\left(a_i\right) X_i,
\] 
where we used that right $G$-invariant vector fields always commute with left $G$-invariant ones.
Since $\mf{g}$ is realised, as said, as right $G$-invariant vector fields on $G$, the universal enveloping algebra $U(\mathfrak{g})$ is realised as right $G$-invariant differential operators on $G$. The isomorphism \eqref{CGH} and Proposition \ref{Artimesh} imply that the universal enveloping algebra of the Atiyah algebra of a Klein geometry $H\subset G$ is isomorphic to the smash product of the commutative algebra of functions on the coset space $G/H$ with the universal enveloping algebra of $\mathfrak{g}$:
\begin{equation}
{\cal U}_{C^\infty(G/H)}\big(\mathfrak{X}(G)^\mathfrak{h}\big)\simeq C^\infty(G/H)\hash{} U(\mathfrak{g}).
\end{equation}
Moreover, the isomorphism \eqref{DGiso} implies
$$
{\cal D}(G)^\mathfrak{h} \simeq {\cal U}_{C^\infty(G)}\big(\mathfrak{X}(G)\big)^\mathfrak{h}
\simeq C^\infty(G/H)\hash{} U(\mathfrak{g}),
$$
where the last isomorphism follows from the fact that the universal enveloping algebra $U(\mathfrak{g})$ is freely generated by right $G$-invariant differential operators on $G$, which provide a basis for ${\cal U}_{C^\infty(G)}\big(\mathfrak{X}(G)\big)$ as a left $A$-module and which already commute with the images $v^\sharp$ of the elements $v$ from $\mf{h}$. This proves that \eqref{UAtiyah} holds for Klein geometries: 
$$
{\cal U}_{C^\infty(G/H)}\big(\mathfrak{X}(G)^\mathfrak{h}\big)\simeq{\cal D}(G)^\mathfrak{h}\simeq {\cal U}_{C^\infty(G)}\big(\mathfrak{X}(G)\big)^\mathfrak{h}.
$$
In particular, the universal enveloping algebra of the Atiyah algebra of a Klein geometry $H\subset G$ is spanned by $H$-invariant differential operators on $G$.
\end{example}

In general, the isomorphism \eqref{wildguess} holds when $\mf{der}(A)$ is freely generated by elements from $\mf{der}(A)^{\mf{h}}$ as a left $A$-module, for essentially the same reason as in Example \ref{Kleinex}.
In general, we always have a canonical $A^{\mf{h}}$-ring morphism from the left-hand to the right-hand side:

\begin{proposition}\label{prop:wildguess}
Let $\mf{h}$ be a Lie algebra acting by derivations on a commutative algebra $A$ via \eqref{representation}. There always exists a morphism of $A^{\mf{h}}$-rings
\[
\Phi \colon \cU_{A^{\mf{h}}} \left(\mf{der}(A)^{\mf{h}}\right) \to \cU_\ahha\big(\mf{der}(A)\big)^{\mf{h}}
\]
induced by the universal property of $\cU_{A^{\mf{h}}}\left(\mf{der}(A)^{\mf{h}}\right)$.
\end{proposition}

\begin{proof}
For the sake of brevity, set $\mf{d} \coloneqq \mf{der}(A)$ and $U \coloneqq U_\K(\mf{h})$. 
If we extend $r$ to an action of $U$ on $A$ (in fact, $A$ becomes a $U$-module algebra) and if we define an additional right $U$-action on $A$ by restriction of scalars along the counit $\varepsilon$ of $U$, that is
\[
a\cdot u \coloneqq \varepsilon(u)a
\]
for all $a\in A$ and $u \in U$, then
\[
A^{\mf{h}} = Z_\uhhu(A) = \Hom{\uhhu}{}{\uhhu}{}{U}{A} = \left\{a \in A \mid u\cdot a = a \cdot u \text{ for all } u \in U\right\},
\]
where $A^{\mf{h}}$ was defined in \eqref{Ah}.
In a similar way, $\mf{h}$ is represented on the vector space $\mf{d}$ via 
\[
\mf{h } \to \mathfrak{gl}_{\,\K}({\mf{d}}), \qquad v \mapsto [v^\sharp,-],
\]
and hence we have a left $U$-module structure on $\mf{d}$ induced by it. By defining again a right $U$-module structure on $\mf{d}$ by restriction of scalars along $\varepsilon$, we find that
\[\mf{d}^{\mf{h}} = Z_\uhhu(\mf{d}) = \Hom{U}{}{U}{}{U}{\mf{d}} = \big\{ \delta \in \mf{d} \mid u \cdot \delta = \delta \cdot u \text{ for all } u \in U \big\}.
\]
If we now consider the $\K$-algebra morphism $\phi_\ahha  \colon A^{\mf{h}} \subseteq A \xrightarrow{\eta_\ahha } \cU_\ahha (\mf{d})$ and the Lie algebra and left $A^{\mf{h}}$-linear map $\phi_{\mf{d}} \colon \mf{d}^{\mf{h}} \subseteq \mf{d} \xrightarrow{\eta_{\mf{d}}} \cU_\ahha (\mf{d})$, then the universal property of $\cU_{A^{\mf{h}}}\left(\mf{d}^{\mf{h}}\right)$ implies that there exists an $A^{\mf{h}}$-ring morphism
\[
\Phi'\colon \cU_{A^{\mf{h}}}(\mf{d}^{\mf{h}}) \to \cU_\ahha (\mf{d})
\]
extending $\phi_{\mf{d}}$, because obviously
\[
\phi_{\mf{d}}(\delta)\phi_\ahha (a) - \phi_\ahha (a)\phi_{\mf{d}}(\delta) = \eta_{\mf{d}}(\delta)\eta_\ahha (a) - \eta_\ahha (a)\eta_{\mf{d}}(\delta) = \eta_\ahha \left(\omega(\delta)(a)\right) = \phi_\ahha (\delta(a))
\]
for all $a \in A^{\mf{h}}$ and $\delta \in \mf{d}^{\mf{h}}$. 
In addition, the Lie algebra map $\mf{h} \xrightarrow{\rho} \mf{d} \xrightarrow{\eta_{\mf{d}}} \cU_\ahha (\mf{d})$ induces a $\K$-algebra map $R \colon U \to \cU_\ahha (\mf{d})$, which makes $\cU_\ahha (\mf{d})$ into a $U$-ring. By a direct computation,
\[
R(X)\phi_\ahha (a) = \eta_{\mf{d}}(r(X))\eta_\ahha (a) = \eta_\ahha (a)\eta_{\mf{d}}(r(X)) + \eta_\ahha \left(\omega(r(X))(a)\right) = \phi_\ahha (a)R(X)
\]
for all $a \in A^{\mf{h}}$ and all $X \in \mf{h}$, and
\[
R(X)\phi_{\mf{d}}(\delta) = \eta_{\mf{d}}(r(X))\eta_{\mf{d}}(\delta) = \eta_{\mf{d}}(\delta)\eta_{\mf{d}}(r(X)) + \eta_{\mf{d}}\left([r(X),\delta]\right) = \phi_{\mf{d}}(\delta)R(X)
\]
for all $\delta \in \mf{d}^{\mf{h}}$ and all $X \in \mf{h}$, whence $\Phi'$ lands, in fact, in
\[
\cU_\ahha (\mf{d})^{\mf{h}} = \big\{V \in \cU_\ahha (\mf{d}) \mid u \cdot V = R(u)V = VR(u) = V \cdot u \text{ for all } u \in U\big\}.
\]
So, we have an $A^{\mf{h}}$-ring morphism $\Phi \colon \cU_{A^{\mf{h}}}(\mf{d}^{\mf{h}}) \to \cU_\ahha (\mf{d})^{\mf{h}}$ as claimed.
\end{proof}

Note, however, that the above canonical morphism of $A^{\mf{h}}$-ring in Proposition \ref{prop:wildguess} need not be injective in general, not even in the projective case, as explicitly shown by the following counterexample.

\begin{example}[Weyl algebra]
Consider a finite-dimensional $\K$-vector space $V$ of dimension $d$. The symmetric algebra ${\cal S}V$ is isomorphic to the algebra $\K[X^i]$ of polynomials in $d$ variables $X^i$ for $i=1,2,\ldots,d$, with which it will sometimes be identified here.
The Grothendieck algebra ${\cal D}({\cal S}V)\simeq {\cal U}_{{\cal S}V}\big(\Der({\cal S}V)\big)$ of polynomial differential operators on the symmetric algebra is isomorphic to the Weyl algebra $A_d(\K)$.
As an ${\cal S}V$-module, $\Der({\cal S}V)\simeq {\cal S}V\otimes V^*$.
We have 
$$
\text{gr}\,{\cal D}({\cal S}V) \simeq \cS_{{}_{\cS V}}\left(\cS V \otimes V^*\right) 
\simeq {\cal S}(V)\otimes {\cal S}(V^*)
\simeq {\cal S}(V\oplus V^*).
$$
Equivalently, the graded algebra associated to the Weyl algebra is the algebra of polynomial functions on $T^*V\simeq V\oplus V^*$. Similarly, $\text{gr}A_d(\K)\simeq \K[X^i,P_j]$.
Let $\mathfrak{h}$ be the one-dimensional Lie algebra acting on the commutative algebra $\K[X^i]$ via a representation \eqref{representation} whose image $\mathfrak{h}^\sharp\subset\Der(A)$ is spanned by the Euler vector field, reading explicitly
$$
E=\sum\limits_{i=1}^d X^i\frac{\partial}{\partial X^i}.
$$ 
Therefore, the commutative subalgebra $({\cal S}V)^\mathfrak{h}$ as in \eqref{Ah} is isomorphic to $\K$, while the Atiyah algebra $\Der({\cal S}V)^\mathfrak{h}$ as in \eqref{Ati} is isomorphic to the Lie algebra $\mathfrak{gl}_\K(V)$. Moreover, the associative subalgebra ${\cal D}({\cal S}V)^\mathfrak{h}$ is spanned by polynomial differential operators on the symmetric algebra ${\cal S}V$ preserving the rank of tensors. It is isomorphic to the centraliser of the Euler vector field inside the Weyl algebra, which is spanned by sums of differential operators with as many coordinates $X$ as derivatives $\partial_X$. This subalgebra is generated by vector fields of the form $X^i\frac{\partial}{\partial X^j}$. In fact, let $\{T^i{}_j\}$ denote a basis of $\mathfrak{gl}_\K(V)$. The morphism $\Phi\colon U_\K\big(\mathfrak{gl}_\K(V)\big)\to{\cal U}_{{\cal S}V}\big(\Der({\cal S}V)\big)^\mathfrak{h}$ is the associative algebra map extending the Lie algebra map $T^i{}_j\mapsto X^i\frac{\partial}{\partial X^j}$ via the universal property. For $d>1$, this morphism is surjective but not injective. This fact is easier to see at the level of their graded counterparts. On the one hand,
$$
\text{gr}\, U_\K\big(\mathfrak{gl}_\K(V)\big)\simeq 
{\cal S}\big(\mathfrak{gl}_\K(V)\big)\simeq {\cal S}(V\otimes V^*),
$$
which is isomorphic to $\K[Y^i{}_j]$. For instance, in degree two, ${\cal S}^2(V\otimes V^*)$ is of dimension $\frac12(d^2+1)d^2$.
On the other hand,
$$
\text{gr}\,{\cal D}({\cal S}V)^\mathfrak{h}\simeq {\cal S}(V\oplus V^*)^\mathfrak{h}\simeq\bigoplus\limits_{n\in\mathbb N} {\cal S}^n(V)\otimes {\cal S}^n(V^*),
$$
which is isomorphic to the subalgebra of $\K[X^i,P_j]$ spanned by homogeneous polynomials with the same degree in $X$ as in $P$. For instance, in degree two,
${\cal S}^2(V\oplus V^*)^\mathfrak{h}\simeq {\cal S}^2(V)\otimes\,{\cal S}^2(V^*)$ is of dimension $\frac14(d+1)^2d^2$. The map $\text{gr}\,\Phi$ reads explicitly $Y^i{}_j\mapsto X^iP_j$ in terms of $\K[Y^i{}_j]$ and $\K[X^i,P_j]$.
\end{example}


\appendix
\addtocontents{toc}{\protect\setcounter{tocdepth}{1}}

\setcounter{section}{0}

\section{Some technical proofs}\label{sec:tech}
In this subsection, after recalling a few technical details concerning permutations, we collect some of the technical proofs omitted from \S\ref{sec:weakcase}.


\subsection{Useful things about permutations}\label{ssec:perm}
Let us recall the following facts about permutations.

\begin{definition}[Shuffle]
A permutation $\sigma \in \mf{S}_k$ is an $(a,b)$-shuffle (where $a+b=k$) if $\sigma(1)<\sigma(2)<\cdots <\sigma(a)$ and $\sigma(a+1)<\sigma(a+2)<\cdots <\sigma(k)$. Let us denote by $W_{a,b}$ the subset of $(a,b)$-shuffles. Its cardinality is $\binom{k}{a}$.
\end{definition}

\begin{lemma}\label{lem:shuffles}
Consider the assignment
$$
W_{a,b}\times (\mf{S}_a  \times \mf{S}_b) \to \mf{S}_k,\quad (\sigma,\tau) \mapsto \sigma \circ \tau,
$$ where $\sigma\in W_{a,b}$ is an  $(a,b)$-shuffle and $\tau\in \mf{S}_a \times \mf{S}_b$ permutes $\{1,\ldots,a\}$ and $\{a+1,\ldots,k\}$ separately. The above map is a bijection.
\end{lemma}

\begin{proof}
Let $\sigma,\sigma' \in W_{a,b}$ and $\tau,\tau' \in \mf{S}_a \times \mf{S}_b$ and assume that $\sigma\tau = \sigma'\tau'$. This implies that $\sigma = \sigma'\tau'\tau^{-1}$. Therefore
\begin{multline*}
\sigma(1) \stackrel{(*)}{=} \min\{\sigma(1),\ldots,\sigma(a)\} = \min\{\sigma'\tau'\tau^{-1}(1),\ldots,\sigma'\tau'\tau^{-1}(a)\} \\
\stackrel{(**)}{=} \min\{\sigma'(1),\ldots,\sigma'(a)\} \stackrel{(*)}{=} \sigma'(1)
\end{multline*}
where in $(*)$ we used the fact that $\sigma$ and $\sigma'$ are $(a,b)$-shuffles and in $(**)$ we used the fact that both $\tau$ and $\tau'$ just permute $\{1,\ldots,a\}$. Analogously one shows inductively that $\sigma(i) = \sigma'(i)$ for all $1\leq i \leq a$, by considering $\sigma(i) = \min\{\sigma(i),\ldots,\sigma(a)\}$. Moreover,
\begin{multline*}
\sigma(k) = \max\{\sigma(a+1),\ldots,\sigma(k)\} = \max\{\sigma'\tau'\tau^{-1}(a+1),\ldots,\sigma'\tau'\tau^{-1}(k)\} \\
= \max\{\sigma'(a+1),\ldots,\sigma'(k)\} = \sigma'(k)
\end{multline*}
and analogously $\sigma(i) = \sigma'(i)$ for all $a+1\leq i \leq k$. Therefore, the assignment
\[
W_{a,b}\times (\mf{S}_a \times \mf{S}_b) \to \mf{S}_k, \qquad (\sigma,\tau) \mapsto \sigma \circ \tau
\]
is injective between sets of the same cardinality, whence it is bijective.
\end{proof}

\begin{lemma}
The assignment
\begin{equation}\label{eq:permuts}
\{\begin{pmatrix}i & k\end{pmatrix} \in\mf{S}_k \mid 1\leq i\leq k\} \times \mf{S}_{k-1} \mapsto \mf{S}_k, \qquad \big(\begin{pmatrix}i & k\end{pmatrix},\sigma\big) \mapsto \begin{pmatrix}i & k\end{pmatrix}\circ\sigma
\end{equation}
is a bijection.
\end{lemma}

\begin{proof}
If $\begin{pmatrix}i & k\end{pmatrix}\sigma = \begin{pmatrix}j & k\end{pmatrix}\tau$ then $\sigma = \begin{pmatrix}i & k\end{pmatrix}\begin{pmatrix}j & k\end{pmatrix}\tau$ and since $\sigma,\tau \in \mf{S}_{k-1}$, we should have that
\[k = \sigma(k) = \begin{pmatrix}i & k\end{pmatrix}\begin{pmatrix}j & k\end{pmatrix}\tau(k) = \begin{pmatrix}i & k\end{pmatrix}\begin{pmatrix}j & k\end{pmatrix}(k) = \begin{cases} j & i \neq j \\ k & i=j. \end{cases}\]
Therefore, $i =j$ and $\sigma = \tau$. Being \eqref{eq:permuts} an injection between sets of the same cardinality, it is a bijection.
\end{proof}
    
Recall also that for every $\rho \in \mf{S}_k$ we have a bijective correspondence $\{\sigma\rho \mid \sigma\in\mf{S}_k\} \leftrightarrow \mf{S}_k$, because $\mf{S}_k$ is a group.


\subsection{Proof of Theorem \ref{thm:isocoring}}\label{sssec:proof3.3}

By the combinatorics of permutations, we can now provide the details of the proof of Theorem \ref{thm:isocoring}.

\begin{proof}[Proof of Theorem \ref{thm:isocoring}]
For every $k \geq 1$ we have
{\small
\begin{align*}
k! \, \Delta_{\cU} & \big(\textsc{S}(X_1 \cdots X_k)\big) \stackrel{\scriptscriptstyle\eqref{eq:symmetr}}{=} 
\sum_{\substack{ j_1,\ldots,j_k \\ \sigma\in\mf{S}_k} }\varphi_{j_1}(X_1)\cdots\varphi_{j_k}(X_k)\Delta_{\cU}\big(\chi_{j_{\sigma(1)}}\cdots\chi_{j_{\sigma(k)}}\big) 
\\
 & \stackrel{\phantom{\eqref{eq:arciuffa}}}{=} \sum_{\substack{ j_1,\ldots,j_k \\ \sigma\in\mf{S}_k} }\varphi_{j_1}(X_1)\cdots\varphi_{j_k}(X_k)\Delta_{\cU}\big(\chi_{j_{\sigma(1)}}\big)\cdots\Delta_{\cU}\big(\chi_{j_{\sigma(k)}}\big) 
 \\
 & \stackrel{\phantom{\eqref{eq:arciuffa}}}{=} \sum_{\substack{ j_1,\ldots,j_k \\ \sigma\in\mf{S}_k} }\varphi_{j_1}(X_1)\cdots\varphi_{j_k}(X_k)\Bigg(\sum_{\substack{s+t=k \\ \rho\in W_{s,t}}}\chi_{j_{\rho\sigma(1)}}\cdots \chi_{j_{\rho\sigma(t)}} \otimes_{\ahha} \chi_{j_{\rho\sigma(t+1)}}\cdots \chi_{j_{\rho\sigma(k)}}\Bigg) 
 \\
 & \stackrel{\scriptscriptstyle\eqref{eq:arciuffa}}{=} \sum_{ j_1,\ldots,j_k }\sum_{ \sigma\in\mf{S}_k}\sum_{\substack{s+t=k \\ \rho\in W_{s,t}}}\varphi_{j_{1}}\left(X_{\rho\sigma(1)}\right)\cdots\varphi_{j_{k}}\left(X_{\rho\sigma(k)}\right)\left(\chi_{j_{1}}\cdots \chi_{j_{t}} \otimes_{\ahha} \chi_{j_{t+1}}\cdots \chi_{j_{k}}\right),
\end{align*}
}
where $W_{s,t} \subseteq \mf{S}_k$ denotes the set of $(s,t)$-shuffles, while%
\begin{small}
\begin{align*}
\big(\textsc{S} \otimes_{\ahha} & \textsc{S}\big) \left(\Delta_{\cS}\left(X_1 \cdots X_k\right)\right) = \left(\textsc{S} \otimes_{\ahha} \textsc{S}\right)\Bigg(\sum_{\substack{s+t=k \\\alpha\in W_{s,t}}} X_{\alpha(1)}\cdots X_{\alpha(t)} \otimes_{\ahha} X_{\alpha(t+1)} \cdots X_{\alpha(k)}\Bigg) 
\\
 & \begin{aligned} \stackrel{\scriptscriptstyle\eqref{eq:symmetr}}{=}  \sum_{\substack{s+t=k \\\alpha\in W_{s,t}}} & \Bigg(\frac{1}{t!}\sum_{\substack{ j_1,\ldots,j_t \\ \beta\in\mf{S}_t} }\varphi_{j_1}\left(X_{\alpha(1)}\right)\cdots\varphi_{j_t}\left(X_{\alpha(t)}\right)\chi_{j_{\beta(1)}}\cdots\chi_{j_{\beta(t)}}\Bigg) 
 \\ 
 & \otimes_{\ahha} \Bigg( \frac{1}{s!} \sum_{\substack{ j_{t+1},\ldots,j_{k} \\ \gamma\in\mf{S}_s} }\varphi_{j_{t+1}}\left(X_{\alpha(t+1)}\right)\cdots\varphi_{j_{k}}\left(X_{\alpha(k)}\right)\chi_{j_{\gamma(t+1)}}\cdots\chi_{j_{\gamma(k)}}\Bigg) 
 \end{aligned} 
 \\
 & \stackrel{\phantom{\eqref{eq:symmetr}}}{=} \sum_{\substack{s+t=k \\\alpha\in W_{s,t}}} \frac{1}{t!} \frac{1}{s!} \sum_{\substack{ j_1,\ldots,j_k \\ \beta\in\mf{S}_t \\ \gamma\in\mf{S}_s }} \varphi_{j_1}\left(X_{\alpha(1)}\right)\cdots\varphi_{j_t}\left(X_{\alpha(t)}\right)\varphi_{j_{t+1}}\left(X_{\alpha(t+1)}\right)\cdots\varphi_{j_k}\left(X_{\alpha(k)}\right) \ldots \\
 & \hspace{2cm} \ldots\big(\chi_{j_{\beta(1)}}\cdots\chi_{j_{\beta(t)}} \otimes_{\ahha} \chi_{j_{\gamma(t+1)}}\cdots\chi_{j_{\gamma(k)}}\big) 
 \\
 & \stackrel{\scriptscriptstyle\eqref{eq:arciuffa}}{=} \sum_{\substack{s+t=k \\\alpha\in W_{s,t}}} \frac{1}{t!} \frac{1}{s!} \sum_{\substack{ j_1,\ldots,j_k \\ \beta\in\mf{S}_t \\ \gamma\in\mf{S}_s }} \varphi_{j_{1}}\left(X_{\alpha\beta(1)}\right)\cdots\varphi_{j_{t}}\left(X_{\alpha\beta(t)}\right)\varphi_{j_{t+1}}\left(X_{\alpha\gamma(t+1)}\right)\cdots\varphi_{j_{k}}\left(X_{\alpha\gamma(k)}\right)\ldots \\
 & \hspace{2cm} \ldots\left(\chi_{j_{1}}\cdots\chi_{j_{t}} \otimes_{\ahha} \chi_{j_{t+1}}\cdots\chi_{j_{k}}\right).
\end{align*}
\end{small}
If we analyse the summand corresponding to
\[\chi_{j_{1}} \cdots \chi_{j_{t}} \otimes_{\ahha} \chi_{j_{t+1}} \cdots \chi_{j_{k}},\]
it appears with coefficient
\[\frac{1}{k!}\sum_{ \sigma\in\mf{S}_k}\sum_{\substack{s+t=k \\ \rho\in W_{s,t}}}\varphi_{j_{1}}\left(X_{\rho\sigma(1)}\right)\cdots\varphi_{j_{k}}\left(X_{\rho\sigma(k)}\right)\]
in the first sum and with coefficient
\[\sum_{\substack{s+t=k \\\alpha\in W_{s,t} \\ \beta\in\mf{S}_t, \gamma\in\mf{S}_s}} \frac{1}{t!} \frac{1}{s!} \varphi_{j_{1}}\left(X_{\alpha\beta(1)}\right)\cdots\varphi_{j_{t}}\left(X_{\alpha\beta(t)}\right)\varphi_{j_{t+1}}\left(X_{\alpha\gamma(t+1)}\right)\cdots\varphi_{j_{k}}\left(X_{\alpha\gamma(k)}\right)\]
in the second. In view of what is recalled in \S\ref{ssec:perm}, 
\begin{gather*}
    \sum_{\substack{s+t=k \\\alpha\in W_{s,t} \\ \beta\in\mf{S}_t, \gamma\in\mf{S}_s}} \frac{1}{t!} \frac{1}{s!} \varphi_{j_{1}}(X_{\alpha\beta(1)})\cdots\varphi_{j_{t}}(X_{\alpha\beta(t)})\varphi_{j_{t+1}}(X_{\alpha\gamma(t+1)})\cdots\varphi_{j_{k}}(X_{\alpha\gamma(k)}) \\
    = \sum_{\substack{s+t=k \\ \tau\in\mf{S}_k}} \frac{1}{t!} \frac{1}{s!} \varphi_{j_{1}}(X_{\tau(1)})\cdots\varphi_{j_{k}}(X_{\tau(k)}).
\end{gather*}
On the other hand, since for every $\rho \in \mf{S}_k$ we have that $\{\rho\sigma \mid \sigma\in\mf{S}_k\} $ is in bijection with $ \mf{S}_k$, we also have that
\begin{align*}
\frac{1}{k!}\sum_{ \sigma\in\mf{S}_k}\sum_{\substack{s+t=k \\ \rho\in W_{s,t}}}\varphi_{j_{1}}\left(X_{\rho\sigma(1)}\right)\cdots\varphi_{j_{k}}\left(X_{\rho\sigma(k)}\right) 
& = \frac{1}{k!}\sum_{\substack{\sigma\in\mf{S}_k \\ s+t=k }}\left|W_{s,t}\right|\varphi_{j_{1}}\left(X_{\sigma(1)}\right)\cdots\varphi_{j_{k}}\left(X_{\sigma(k)}\right) \\
& = \sum_{\substack{\sigma\in\mf{S}_k \\ s+t=k }}\frac{1}{s!t!}\varphi_{j_{1}}\left(X_{\sigma(1)}\right)\cdots\varphi_{j_{k}}\left(X_{\sigma(k)}\right).
\end{align*}
Therefore $\Delta_{\cU} \circ \textsc{S} = (\textsc{S} \otimes_{\ahha} \textsc{S}) \circ \Delta_{\cS}$. The compatibility with the counit is clear.
\end{proof}


\subsection{Proof of Proposition \ref{prop:Gamma}}
\label{sssec:prop3.7}

The following computational result will be helpful in completing the proof of Proposition \ref{prop:Gamma}.

\begin{lemma}\label{lem:maancheno}
For every $a \in A$ and $X_1,\ldots,X_k \in \mf{h}$, in $\cU_\ahha(\mf{h})$ we have
\begin{equation}\label{eq:maancheno}
X_1\cdots X_ka = \sum_{\substack{t+s = k \\ \sigma \in W_{t,s}}} \omega\left(X_{\sigma(1)}\right)\cdots \omega\left(X_{\sigma(t)}\right)(a)X_{\sigma(t+1)} \cdots X_{\sigma(k)},
\end{equation}
where, as usual, $W_{t,s}$ is the set of $(t,s)$-shuffles as in \S\ref{ssec:perm}.
\end{lemma}

\begin{proof}
We proceed by induction on $k \geq 0$. For $k = 1$, we have
\[
Xa \stackrel{\scriptscriptstyle\eqref{eq:compLRalg}}{=} aX + \omega(X)(a).
\]
Assume then that 
\begin{equation}
\label{eq:maancheno1}
X_1\cdots X_na = \sum_{\substack{t+s = n \\ \sigma \in W_{t,s}}} \omega\left(X_{\sigma(1)}\right)\cdots \omega\left(X_{\sigma(t)}\right)(a)X_{\sigma(t+1)} \cdots X_{\sigma(n)},
\end{equation}
for all $n \leq k-1$ and let us verify it for $n = k$. We have
\begin{align*}
\sum_{\substack{u+v = k \\ \sigma \in W_{u,v}}} & \omega\left(X_{\sigma(1)}\right)\cdots \omega\left(X_{\sigma(u)}\right)(a)X_{\sigma(u+1)} \cdots X_{\sigma(k-1)}X_{\sigma(k)} \\
 & \stackrel{\phantom{\eqref{eq:maancheno1}}}{=} \sum_{\substack{u+v = k \\ \sigma \in W_{u,v} \\ \sigma(k) = k}} \omega\left(X_{\sigma(1)}\right)\cdots \omega\left(X_{\sigma(u)}\right)(a)X_{\sigma(u+1)} \cdots X_{\sigma(k-1)}X_{k}\, + \\ 
& \hspace{1cm} +\, \sum_{\substack{u+v = k \\ \sigma \in W_{u,v} \\ \sigma(k) \neq k}} \omega\left(X_{\sigma(1)}\right)\cdots \omega\left(X_{\sigma(u)}\right)(a)X_{\sigma(u+1)} \cdots X_{\sigma(k-1)}X_{\sigma(k)} \\
 & \stackrel{(*)}{\mathmakebox[\widthof{$\stackrel{\scriptscriptstyle\eqref{eq:maancheno1}}{=}$}]{=}} \sum_{\substack{t+s = k-1 \\ \sigma \in W_{t,s}}} \omega\left(X_{\sigma(1)}\right)\cdots \omega\left(X_{\sigma(t)}\right)(a)X_{\sigma(t+1)} \cdots X_{\sigma(k-1)}X_{k} \, + \\ 
& \hspace{1cm} +\,  \sum_{\substack{u+v = k \\ \sigma \in W_{u,v} \\ \sigma(u) = k}} \omega\left(X_{\sigma(1)}\right)\cdots \omega\left(X_{\sigma(u-1)}\right)\omega\left(X_{k}\right)(a)X_{\sigma(u+1)} \cdots X_{\sigma(k-1)}X_{\sigma(k)} \\
 & \stackrel{\phantom{\eqref{eq:maancheno1}}}{=} \sum_{\substack{t+s = k-1 \\ \sigma \in W_{t,s}}} \omega\left(X_{\sigma(1)}\right)\cdots \omega\left(X_{\sigma(t)}\right)(a)X_{\sigma(t+1)} \cdots X_{\sigma(k-1)}X_{k} \, + \\ 
& \hspace{1cm} +\,  \sum_{\substack{u+v = k \\ \sigma \in W_{u,v} \\ \sigma(u) = k \\ \tau=(u,u+1,\ldots,k)}} \omega\left(X_{\sigma\tau(1)}\right)\cdots \omega\left(X_{\sigma\tau(u-1)}\right)\omega\left(X_{k}\right)(a)X_{\sigma\tau(u)} \cdots X_{\sigma\tau(k-2)}X_{\sigma\tau(k-1)} \\
 & \stackrel{\phantom{\eqref{eq:maancheno1}}}{=} \sum_{\substack{t+s = k-1 \\ \sigma \in W_{t,s}}} \omega\left(X_{\sigma(1)}\right)\cdots \omega\left(X_{\sigma(t)}\right)(a)X_{\sigma(t+1)} \cdots X_{\sigma(k-1)}X_{k} \, + \\ 
& \hspace{1cm} +\,  \sum_{\substack{t+s = k-1 \\ \sigma \in W_{t,s}}} \omega\left(X_{\sigma(1)}\right)\cdots \omega(X_{\sigma(t)})\omega(X_{k})(a)X_{\sigma(t+1)} \cdots X_{\sigma(k-1)} \\
& \stackrel{\scriptscriptstyle\eqref{eq:maancheno1}}{=} X_1\cdots X_{k-1}aX_k + X_1\cdots X_{k-1}\omega(X_k)(a) \\
& \stackrel{\scriptscriptstyle\eqref{eq:compLRalg}}{\mathmakebox[\widthof{$\stackrel{\scriptscriptstyle\eqref{eq:maancheno1}}{=}$}]{=}} X_1\cdots X_ka,
\end{align*}
where to transform the second summand in $(*)$ we used the fact that if $\sigma(k) \neq k$, then $k \neq \sigma(i)$ for all $u+1 \leq i \leq k-1$ because $\sigma(i) < \sigma(k)$ by the shuffle condition, whence $\sigma(u) = k$ by the shuffle condition again.
\end{proof}

We are now ready to provide the details of the missing part in the proof of Proposition \ref{prop:Gamma}.

\begin{lemma}
The $A$-coring section $\Gamma$ of the $A$-bialgebroid morphism $\Pi \coloneqq \cU_\ahha (\pi)$ satisfies
\[\Gamma(a \bla u \bra b) = a \bla \Gamma(u) \bra b\]
for all $a,b \in A$ and $ u  \in \cU_\ahha(\mf{h})$.
\end{lemma}

\begin{proof}
To prove the compatibility between $\Gamma$ and the black actions (see \eqref{pergolesi}), it is enough to check that $\Gamma(ua) = \Gamma(u)a$ for $a \in A$, $u \in \cU_\ahha(\mf{h})$. To this end, denote by $r_a$ the endomorphism (both of $\cU_\ahha(\mf{h})$ and $\cU_\ahha(\mf{g})$) given by right multiplication by $a \in A$. In view of the left $A$-linearity of the maps involved, it suffices to check that $\Gamma \circ r_a \circ \psi_{\mf{h}} = r_a \circ \Gamma \circ \psi_{\mf{h}}$ (or, equivalently, that $\Gamma \circ r_a \circ \psi_{\mf{h}} = r_a \circ \psi_{\mf{g}} \circ \cS_\ahha (\gamma)$) on an homogeneous element of the form $X_1\cdots X_k \in \cS_\ahha (\mf{h})$. This can be done directly as follows:
{\tiny
\begin{align*}
	(\Gamma \circ r_a \circ \psi_{\mf{h}}) & (X_1\cdots X_k) 
		\stackrel{\scriptscriptstyle\eqref{eq:symmetr}}{=} 
	\Gamma\Bigg(\frac{1}{k!}\sum_{\substack{ j_1,\ldots,j_k \in I \\ \sigma\in\mf{S}_k} }\psi_{j_1}(X_1)\cdots\psi_{j_k}(X_k)\theta_{j_{\sigma(1)}}\cdots\theta_{j_{\sigma(k)}}a\Bigg) 
\\
	\stackrel{\scriptscriptstyle\eqref{eq:arciuffa}}{=} \ & 
		\Gamma\Bigg(\frac{1}{k!}\sum_{\substack{ j_1,\ldots,j_k \in I \\ \sigma\in\mf{S}_k} }\psi_{j_1}(X_{\sigma(1)})\cdots\psi_{j_k}(X_{\sigma(k)})\theta_{j_{1}}\cdots\theta_{j_{k}}a\Bigg) 
\\
	\stackrel{\scriptscriptstyle\eqref{eq:maancheno}}{=} \ &
		\Gamma\Bigg(\frac{1}{k!} \sum_{\substack{ j_1,\ldots,j_k \in I \\ \sigma\in\mf{S}_k} }\sum_{\substack{t+s=k \\ \tau \in W_{t,s}}}\psi_{j_1}(X_{\sigma(1)})\cdots\psi_{j_k}(X_{\sigma(k)})\omega_{\mf{h}}(\theta_{j_{\tau(1)}})\cdots\omega_{\mf{h}}(\theta_{j_{\tau(t)}})(a) \theta_{j_{\tau(t+1)}}\cdots\theta_{j_{\tau(k)}}\Bigg) 
\\
		\stackrel{(\bullet)}{=} \ & 
			\frac{1}{k!}\sum_{\substack{ j_1,\ldots,j_k \in I \\ \sigma\in\mf{S}_k} }\sum_{\substack{t+s=k \\ \tau \in W_{t,s}}}\psi_{j_{\tau(1)}}(X_{\sigma\tau(1)})\cdots\psi_{j_{\tau(t)}}(X_{\sigma\tau(t)})\omega_{\mf{h}}(\theta_{j_{\tau(1)}})\cdots\omega_{\mf{h}}(\theta_{j_{\tau(t)}})(a) \ldots \\ 
			& \hspace{1cm} \ldots \Gamma(\psi_{j_{\tau(t+1)}}(X_{\sigma\tau(t+1)})\cdots\psi_{j_{\tau(k)}}(X_{\sigma\tau(k)}) \theta_{j_{\tau(t+1)}}\cdots\theta_{j_{\tau(k)}}) 
\\
		\stackrel{(*)}{=} \ &
			\frac{1}{k!} \sum_{\substack{t+s=k \\ \tau \in W_{t,s}}} \sum_{\substack{ j_1,\ldots,j_k \in I \\ \rho\in\mf{S}_k} } \psi_{j_{\tau(1)}}(X_{\rho(1)})\cdots\psi_{j_{\tau(t)}}(X_{\rho(t)})\omega_{\mf{h}}(\theta_{j_{\tau(1)}})\cdots\omega_{\mf{h}}(\theta_{jf_{\tau(t)}})(a) \ldots \\ 
			& \hspace{1cm} \ldots \Gamma(\psi_{j_{\tau(t+1)}}(X_{\rho(t+1)})\cdots\psi_{j_{\tau(k)}}(X_{\rho(k)}) \theta_{j_{\tau(t+1)}}\cdots\theta_{j_{\tau(k)}}) 
\\
		\stackrel{(\star)}{=} \ & 
			\frac{1}{k!} \sum_{\substack{t+s=k \\ \tau \in W_{t,s}}} \sum_{\substack{ j_1,\ldots,j_k \in I \\ \sigma\in W_{t,s} \\ (\varpi_t,\varpi_s) \in \mf{S}_t \times \mf{S}_s} } \psi_{j_{\tau(1)}}(X_{\sigma\varpi_t(1)})\cdots\psi_{j_{\tau(t)}}(X_{\sigma\varpi_t(t)})\omega_{\mf{h}}(\theta_{j_{\tau(1)}})\cdots\omega_{\mf{h}}(\theta_{j_{\tau(t)}})(a) \ldots \\ 
			& \hspace{1cm} \ldots \Gamma(\psi_{j_{\tau(t+1)}}(X_{\sigma(t+\varpi_s(1))})\cdots\psi_{j_{\tau(k)}}(X_{\sigma(t + \varpi_s(s))}) \theta_{j_{\tau(t+1)}}\cdots\theta_{j_{\tau(k)}}) 
\\
		\stackrel{}{=} \ &
			\frac{1}{k!} \sum_{\substack{t+s=k \\ \tau \in W_{t,s}}} \sum_{\substack{ j_{\tau(1)},\ldots,j_{\tau(t)} \in I \\ \sigma\in W_{t,s} \\ \varpi_t \in \mf{S}_t }} \psi_{j_{\tau(1)}}(X_{\sigma\varpi_t(1)})\cdots\psi_{j_{\tau(t)}}(X_{\sigma\varpi_t(t)})\omega_{\mf{h}}(\theta_{j_{\tau(1)}})\cdots\omega_{\mf{h}}(\theta_{j_{\tau(t)}})(a) \ldots \\ 
			& \hspace{1cm} \ldots \Gamma\Bigg(\sum_{\substack{ j_{\tau(t+1)},\ldots,j_{\tau(k)} \in I \\ \varpi_s \in \mf{S}_s} }\psi_{j_{\tau(t+1)}}(X_{\sigma(t+\varpi_s(1))})\cdots\psi_{j_{\tau(k)}}(X_{\sigma(t + \varpi_s(s))}) \theta_{j_{\tau(t+1)}}\cdots\theta_{j_{\tau(k)}}\Bigg) 
\\
	\stackrel{\scriptscriptstyle\eqref{eq:symmetr},\eqref{eq:arciuffa}}{=} \ &
		\frac{1}{k!} \sum_{\substack{t+s=k \\ \tau \in W_{t,s}}} \sum_{\substack{ j_{\tau(1)},\ldots,j_{\tau(t)} \in I \\ \sigma\in W_{t,s} \\ \varpi_t \in \mf{S}_t }} \psi_{j_{\tau(1)}}(X_{\sigma\varpi_t(1)})\cdots\psi_{j_{\tau(t)}}(X_{\sigma\varpi_t(t)})\omega_{\mf{h}}(\theta_{j_{\tau(1)}})\cdots\omega_{\mf{h}}(\theta_{j_{\tau(t)}})(a) \dots \\
		& \hspace{1cm} \ldots s!\Gamma(\psi_{\mf{h}}(X_{\sigma(t+1)}\cdots X_{\sigma(t+s)})) 
\\
	\stackrel{}{=} \ &  
		\sum_{\substack{t+s=k \\ \tau \in W_{t,s}}} \sum_{\substack{ j_{\tau(1)},\ldots,j_{\tau(t)} \in I \\ \sigma\in W_{t,s} \\ \varpi_t \in \mf{S}_t }} \frac{s!}{k!} \varphi_{j_{\tau(1)}}\gamma(X_{\sigma\varpi_t(1)})\cdots\varphi_{j_{\tau(t)}}\gamma(X_{\sigma\varpi_t(t)})\omega_{\mf{h}}\pi(\chi_{j_{\tau(1)}})\cdots\omega_{\mf{h}}\pi(\chi_{j_{\tau(t)}})(a) \ldots \\ 
		& \hspace{1cm} \ldots \psi_{\mf{g}}(\gamma(X_{\sigma(t+1)})\cdots \gamma(X_{\sigma(t+s)})) 
\\
		\stackrel{\scriptscriptstyle\eqref{eq:symmetr}}{=} \ &
			\sum_{\substack{t+s=k \\ \tau \in W_{t,s}}} \sum_{\substack{ j_{\tau(1)},\ldots,j_{\tau(t)} \in I \\ \sigma\in W_{t,s} \\ \varpi_t \in \mf{S}_t }} \frac{s!}{k!} \varphi_{j_{\tau(1)}}\gamma(X_{\sigma\varpi_t(1)})\cdots\varphi_{j_{\tau(t)}}\gamma(X_{\sigma\varpi_t(t)})\omega_{\mf{g}}(\chi_{j_{\tau(1)}})\cdots\omega_{\mf{g}}(\chi_{j_{\tau(t)}})(a) \ldots \\ 
			& \hspace{1cm} \ldots \Bigg(\sum_{\substack{j_{\tau(t+1)},\ldots,j_{\tau(k)} \in I \\ \varpi_s \in \mf{S}_s}} \frac{1}{s!}\varphi_{j_{\tau(t+1)}}\gamma(X_{\sigma(t+\varpi_s(1))})\cdots \varphi_{j_{\tau(k)}}\gamma(X_{\sigma(t+\varpi_s(s))})\chi_{j_{\tau(t+1)}}\cdots \chi_{j_{\tau(k)}}\Bigg) 
\\
		\stackrel{}{=} \ &   
			\sum_{\substack{t+s=k \\ \tau \in W_{t,s}}} \sum_{\substack{ j_{1},\ldots,j_{k} \in I \\ \sigma\in W_{t,s} \\ (\varpi_t,\varpi_s) \in \mf{S}_t \times \mf{S}_s }} \frac{1}{k!} \varphi_{j_{\tau(1)}}\gamma(X_{\sigma\varpi_t(1)})\cdots\varphi_{j_{\tau(t)}} \gamma(X_{\sigma\varpi_t(t)}) \ldots \\ 
		 & \hspace{1cm}	\ldots\varphi_{j_{\tau(t+1)}}\gamma(X_{\sigma(t+\varpi_s(1))})\cdots \varphi_{j_{\tau(k)}}\gamma(X_{\sigma(t+\varpi_s(s))})\omega_{\mf{g}}(\chi_{j_{\tau(1)}})\cdots\omega_{\mf{g}}(\chi_{j_{\tau(t)}})(a)\chi_{j_{\tau(t+1)}}\cdots \chi_{j_{\tau(k)}} 
\\
		\stackrel{(\star)}{=} \ &
			\sum_{\substack{t+s=k \\ \tau \in W_{t,s}}} \sum_{\substack{ j_{1},\ldots,j_{k} \in I \\ \rho\in \mf{S}_k}} \frac{1}{k!} \varphi_{j_{\tau(1)}}\gamma(X_{\rho(1)})\cdots\varphi_{j_{\tau(t)}}\gamma(X_{\rho(t)})\varphi_{j_{\tau(t+1)}}\gamma(X_{\rho(t+1)})\cdots \varphi_{j_{\tau(k)}}\gamma(X_{\rho(t+s)}) \ldots \\ 
			& \hspace{1cm} \ldots \omega_{\mf{g}}(\chi_{j_{\tau(1)}})\cdots\omega_{\mf{g}}(\chi_{j_{\tau(t)}})(a)\chi_{j_{\tau(t+1)}}\cdots \chi_{j_{\tau(k)}} 
\\
	\stackrel{(*)}{=} \ & 
		\sum_{\substack{t+s=k \\ \tau \in W_{t,s}}} \sum_{\substack{ j_{1},\ldots,j_{k} \in I \\ \sigma\in \mf{S}_k}} \frac{1}{k!} \varphi_{j_{\tau(1)}}\gamma(X_{\sigma\tau(1)})\cdots \varphi_{j_{\tau(k)}}\gamma(X_{\sigma\tau(k)})  \omega_{\mf{g}}(\chi_{j_{\tau(1)}})\cdots\omega_{\mf{g}}(\chi_{j_{\tau(t)}})(a)\chi_{j_{\tau(t+1)}}\cdots \chi_{j_{\tau(k)}} 
\\
	\stackrel{(\bullet)}{=} \ &  
		\sum_{\substack{ j_{1},\ldots,j_{k} \in I \\ \sigma\in \mf{S}_k}} \frac{1}{k!} \varphi_{j_{1}}\gamma(X_{\sigma(1)})\cdots \varphi_{j_{k}}\gamma(X_{\sigma(k)}) \Bigg( \sum_{\substack{t+s=k \\ \tau \in W_{t,s}}}\omega_{\mf{g}}(\chi_{j_{\tau(1)}})\cdots\omega_{\mf{g}}(\chi_{j_{\tau(t)}})(a)\chi_{j_{\tau(t+1)}}\cdots \chi_{j_{\tau(k)}}\Bigg) 
\\
 \stackrel{\scriptscriptstyle\eqref{eq:maancheno}}{=} \ & \sum_{\substack{ j_{1},\ldots,j_{k} \in I \\ \sigma\in \mf{S}_k}} \frac{1}{k!} \varphi_{j_{1}}\gamma(X_{\sigma(1)})\cdots \varphi_{j_{k}}\gamma(X_{\sigma(k)}) \chi_{j_{1}} \cdots \chi_{j_{k}} a \\ 
 \stackrel{\scriptscriptstyle\eqref{eq:symmetr},\eqref{eq:arciuffa}}{=} \ & \psi_{\mf{g}}(\gamma(X_1)\cdots \gamma(X_k))a = (r_a \circ \psi_{\mf{g}} \circ \cS_\ahha (\gamma))(X_1\cdots X_k),
\end{align*}
}%
where in $(*)$ we used the fact that for any fixed $\tau \in W_{t,s}$, if $\sigma$ runs over all permutations in $\mf{S}_k$, then $\rho\coloneqq \sigma\tau$ still runs over all permutations in $\mf{S}_k$, in $(\star)$ we used the fact that for $t+s=k$ fixed, $W_{t,s} \times (\mf{S}_t \times \mf{S}_s) \to \mf{S}_k, (\sigma,\varpi) \mapsto \sigma\varpi,$ is a bijection as in Lemma \ref{lem:shuffles}, and in $(\bullet)$ we used commutativity of $A$ to reorder coefficients.
\end{proof}


\addtocontents{toc}{\SkipTocEntry}
\section*{Acknowledgements}

X.B.\ thanks J.\ Huebschmann for correspondence on universal enveloping algebras of Lie-Rinehart algebras before this project started, and C.\ Laurent-Gengoux for discussions about foliations.
N.K.\ and P.S.\ are members of the {\em Gruppo Nazionale per le Strutture Algebriche, Geometriche e le loro Applicazioni (GNSAGA-INdAM)}.
P.S.\ is a {\em Charg\'e de Recherches} of the {\em Fonds de la Recherche Scientifique (FNRS)}.

We are grateful to the Erwin Schr\"{o}dinger International Institute for Mathematics and Physics (Vienna) for hospitality and support during the thematic programme {\em ``Geometry for Higher Spin Gravity: Conformal Structures, PDEs, and Q-manifolds''} (August 23rd -- September 17th, 2021), where part of this work was completed.


\end{document}